\documentclass[11pt,a4paper]{amsart}
\usepackage[utf8]{inputenc}
\usepackage[T1]{fontenc}
\usepackage{soul}
\usepackage{amsfonts, amsthm, amsmath,amssymb}
\usepackage{mathrsfs}
\usepackage{orcidlink}
\usepackage{tikz,wasysym,subfigure}
\usetikzlibrary{shapes}
\usetikzlibrary{perspective}
\usepackage{cite,pifont}
\usepackage{hyperref}
\usepackage{paralist, mathdots}
\usepackage{arydshln}
\usepackage{scalerel}
\usepackage{tikz}
\tikzset{every node/.style={draw, circle, inner sep=2pt}}
\usetikzlibrary{shapes.geometric, arrows.meta, positioning}
\usepackage{cancel}
\newcommand{\xmark}{\text{\ding{55}}}

\newtheorem{theorem}{Theorem}[section]
\newtheorem{lemma}[theorem]{Lemma}
\newtheorem{proposition}[theorem]{Proposition}
\newtheorem{corollary}[theorem]{Corollary}

\theoremstyle{definition}
\newtheorem{definition}[theorem]{Definition}

\newtheorem{remark}[theorem]{Remark}
\newtheorem{example}[theorem]{Example}

\newcommand{\trans}{^\top}
\newcommand{\diam}{\operatorname{diam}}
\newcommand{\hdv}{V_{d\geq 3}}
\newcommand{\degnottwo}{V_{d\ne 2}}
\newcommand{\chdv}{\kappa}
\newcommand{\gapzvezdica}{\gap^{*}}
\newcommand{\leaves}{V_{d=1}}
\newcommand{\isolated}{V_{d=0}}
\newcommand{\clover}{\mathrm{clover}}
\newcommand{\cloverNoFour}{{\mathrm{clover}_{\square \mkern-12mu \xmark}}}
\newcommand{\R}{\mathbb{R}}
\newcommand{\bN}{\mathbb{N}}
\newcommand{\mc}{\mathcal}

\newcommand{\stp}{\mathrm{st}_+} 
\newcommand{\C}{\mathscr C}

\DeclareMathOperator{\starc}{star}
\DeclareMathOperator{\gstar}{gstar}
\DeclareMathOperator{\garlic}{garlic}
\DeclareMathOperator{\gap}{gap}
\newcommand{\GnoFourCycle}{{\mc G}_{\square \mkern-12mu \xmark}}
\newcommand{\pathk}[3]{{#1}\; \overset{{#3}}{\text{---}}\; {#2}}

\definecolor{polona}{rgb}{.2,.2,.8}

\definecolor{helena}{rgb}{.2,.8,.4}

\title{Optimization problem for star covers of graphs without four cycles}
\author[Kokol Bukovšek, Oblak,  \v Smigoc]{Damjana Kokol Bukovšek \orcidlink{0000-0002-0098-6784}, Polona Oblak \orcidlink{0000-0002-4876-5163}, Helena \v Smigoc \orcidlink{0000-0002-4523-8181}}

\address[Damjana Kokol Bukovšek]{University of Ljubljana, School of Economics and Business, and Institute of Mathematics, Physics, and Mechanics, Slovenia}
\email{damjana.kokol.bukovsek@ef.uni-lj.si}
\address[Helena \v Smigoc]{School of Mathematics and Statistics, University College Dublin, Ireland}
\email{helena.smigoc@ucd.ie}
\address[Polona Oblak]{University of Ljubljana, Faculty of Computer and Information Science
and Faculty of Mathematics and Physics, Slovenia; Institute of Mathematics, Physics, and Mechanics, Slovenia}
\email{polona.oblak@fri.uni-lj.si}
\date{\today}

\begin{document}

\newcommand{\hugeGamma}{

  \node[vertex] (L_t)  at (0,1)  {};
  \node[vertex] (L_m)  at (0,0.3)  {};
  \node[vertex] (L_b)  at (0,-0.4)    {};
  \node[vertex] (R_tL) at (1,1){};
  \node[vertex] (R_mL) at (1,0.3){};
  \node[vertex] (R_bL) at (1,-0.4)  {};

  \draw[cyanedge] (L_t) -- (L_m) -- (L_b);
  \draw[cyanedge] (R_tL) -- (R_mL) -- (R_bL);
  \draw[cyanedge] (L_m) -- (R_mL);
  \draw[blackedge] (L_b) -- (R_mL);
  \draw[blackedge] (L_m) -- (R_bL);

  \node[vertex] (D_t) at (3,1.0) {};
  \node[vertex] (D_m) at (3,0.3) {};
  \node[vertex] (D_b) at (3,-0.4){};
  \node[vertex] (D_rt) at (4,0.6){};
  \node[vertex] (D_rb) at (4,0){};

  \draw[orangeedge] (D_t) -- (D_rt) -- (D_rb) -- (D_b) -- (D_m) --(D_t) ;
\path (D_m)  edge [bend left,color=orange,thick] (D_t) edge [bend right,color=black,thick]  (D_b);
\path (D_rt)  edge [bend left,color=black,thick] (D_rb);

  \node[vertex] (P_t) at (5,1) {};
  \node[vertex] (P_m) at (5,0.3) {};
  \node[vertex] (P_b) at (5,-0.4){};
  \node[vertex] (P_r) at (5.5,0.7){};  
  
  \draw[tealedge] (P_m) -- (P_r) -- (P_t) -- (P_m) -- (P_b);
  
  \node[vertex] (B0) at (0,2.0) {};
  \node[vertex] (B1) at (1,2.0) {};
  \node[vertex] (B2) at (3,2.0) {};
  \node[vertex] (B3) at (5,2.0) {};
  \tikzset{every loop/.style={min distance=4mm,in=120,out=60,looseness=10}}
\draw[color=magenta] (B3) to[loop above] (B3);
\draw[blackedge] (B2) to[loop above] (B2);
\draw[tealedge] (P_t) to[loop above] (P_t);
\draw[blackedge] (P_r) to[loop above] (P_r);
\draw[orangeedge] (D_t) to[loop above] (D_t);
\tikzset{every loop/.style={min distance=4mm,in=130,out=50,looseness=15}}
\draw[blackedge] (B3) to[loop above] (B3);
\tikzset{every loop/.style={min distance=4mm,in=240,out=300,looseness=10}}
\draw[orangeedge] (D_b) to[loop below] (D_b);
\draw[blackedge] (P_b) to[loop below] (P_b);

  \draw[blackedge] (B0)--(B1) -- (B2) -- (B3);

  \draw[blackedge] (L_t) -- (B0)--(R_tL);
  \draw[blackedge] (L_t) -- (B1)--(R_tL);
  \draw[blackedge] (L_b)--(R_bL);
  
  \draw[blackedge] (D_t) -- (B2);
  \draw[blackedge] (P_t) -- (B3);

  \draw[blackedge] (R_bL) -- (D_b);
  \draw[blackedge] (R_mL) -- (D_m);
  \draw[blackedge] (B2) edge [bend left,color=black,thick] (D_m);
}

\newcommand{\tauGamma}{
 \node[vertex] (H) at (0.5,1.2){};
  
  \node[vertex] (B0) at (0,2.0) {};
  \node[vertex] (B1) at (1,2.0) {};
  \node[vertex] (B2) at (2,2.0) {};
  \phantom{\node[vertex, fill=orange] (C1) at (1,1) {};}
 
 \tikzset{every loop/.style={min distance=4mm,in=240,out=300,looseness=10}}
\draw[blackedge] (B2) to[loop below] (B2);
\draw[blackedge] (H) to [loop below] (H);
  \draw[blackedge] (B1)--(H) -- (B0) -- (B1) -- (B2);

  \node[rectangle,draw=none] at (0.5,0.6) {$w_1$};
  \node[rectangle,draw=none] at (0,2.5) {$w_2$};
  \node[rectangle,draw=none] at (1,2.5) {$w_3$};
  \node[rectangle,draw=none] at (2,2.5) {$w_4$};
}

\newcommand{\tauGammaNoLabel}{
 \node[vertex] (H) at (0.5,1.2){};
  
  \node[vertex] (B0) at (0,2.0) {};
  \node[vertex] (B1) at (1,2.0) {};
  \node[vertex] (B2) at (2,2.0) {};
  \phantom{\node[vertex, fill=orange] (C1) at (1,1) {};}
 
 \tikzset{every loop/.style={min distance=4mm,in=240,out=300,looseness=10}}
\draw[blackedge] (B2) to[loop below] (B2);
\draw[blackedge] (H) to [loop below] (H);
  \draw[blackedge] (B1)--(H) -- (B0) -- (B1) -- (B2);
}

\newcommand{\tauGammaDeg}{
\node[vertex] (H) at (-1,1.2){};
  
  \node[vertex] (B0) at (0,1.2) {};
 \phantom{ \node[vertex] (B1) at (1,0) {};}
  \node[vertex] (B2) at (1,1.2) {};
  
 \tikzset{every loop/.style={min distance=4mm,in=120,out=60,looseness=10}}
\draw[blackedge] (B2) to[loop above] (B2);
 \draw[blackedge] (H) to [loop above] (H);
 
 \draw (B0) edge [bend right,color=green,thick] (H);
 
  \draw[blackedge] (H) -- (B0) --  (B2);

}

\newcommand{\tauGammatauGamma}{
\node[vertex] (H) at (0,1.2){};
   \node[vertex] (B2) at (1,1.2) {};
 
 \tikzset{every loop/.style={min distance=4mm,in=120,out=60,looseness=10}}
\draw[blackedge] (B2) to[loop above] (B2);
 \draw[blackedge] (H) to [loop above] (H);
 
  \draw[blackedge] (H) --  (B2);
  }

\maketitle

\begin{abstract}
This work presents a 
study of star covers on graphs. Unlike traditional formulations that minimize the number of stars, our aim is to optimize the number of bipartite components used in the cover. This problem, motivated by a symmetric nonnegative trifactorization of matrices and the SNT-rank of graphs, is in general hard to solve. We consider a family of graphs that do not contain four cycles, and develop an algorithm to determine the SNT-rank of such graphs.
\end{abstract}



\section{Introduction and notation}

\subsection{Motivation and related work}

This work investigates star covers in graphs, in which we aim to minimize the number of bipartite components required to form the cover. While this setup may at first appear highly specialized, it arises naturally from the study of symmetric nonnegative matrix tri-factorizations. 
 
Matrix factorizations that require entrywise nonnegativity on their latent factors can be used as a tool for discovering structures that are inherently additive and physically interpretable. 
Studying the support pattern of a matrix provides  lower bounds on various factorization ranks defined through nonnegative factorization,~\cite{MR3578391,MR1871878, MR3386390}. 
Covering problems extend to broader matrix questions involving ranks and eigenvalues under structural constraints. For example, in the classic inverse eigenvalue problem for graphs, determining the minimum rank of a matrix family is closely connected to its path cover number,~\cite{MR4478249,MR1416462}.

The Nonnegative Matrix Factorization (NMF) is the most common such factorization. In the approximate NMF framework, a given nonnegative matrix $V$ is decomposed into nonnegative factors $W$ and $H$ such that $V \approx WH$,~\cite{1999-Lee-Seung-NMF}, while the exact NMF requires that $V = WH$,~\cite{MR2563025}. The nonnegative rank of a matrix $V$ is the minimal number of columns of $W$ over all such factorizations of $V$. The existence of an exact NMF of rank $k$ for a matrix $V$ requires the existence of a Boolean factorization of rank $r \leq k$ for its support pattern. In particular, the Boolean rank of a matrix’s zero-nonzero pattern serves as a combinatorial lower bound for its nonnegative rank. In addition, Boolean Matrix Factorization (BMF) can inform the construction of Nonnegative Matrix Factorization (NMF) models, when the approximate decomposition requires specific entries to remain zero to preserve the underlying support of the data, and is of independent interest for its capacity to extract interpretable, set-based patterns from discrete data,~\cite{CGP81}. For an overview of all the concepts mentioned above, we refer the reader to \cite{2021-Gillis-book}.

When dealing with symmetric data, this framework is mirrored by Symmetric Nonnegative Matrix Tri-Factorization (SN-Trifactorization).  SN-Trifactorization approximates a symmetric $n \times n$ input matrix $A$ as the product of nonnegative matrices: an $n \times r$ matrix $B$ and a symmetric $r \times r$ matrix $C$, such that $A=BCB\trans$ (for exact factorization), or $A \approx BCB\trans$ (for approximate factorization).  This framework has applications across a variety of data science applications, such as topic modeling and community detection,~\cite{song2019improved}. 
The symmetric nonnegative rank of a matrix $A$, denoted by $\stp(A)$, is defined as the minimum number of columns in $B$ required to satisfy this equality across all such factorizations, see \cite{KBS23}. Analogous to the asymmetric case, a combinatorial lower bound for this rank is provided by the equivalent factorization under Boolean arithmetic. 

 By treating the pattern of $A$ as a graph $G$, we can cast this setup as a combinatorial problem on graphs. This setup was first explored in~\cite{KBS24}, where the \emph{SNT-rank of a graph $G$} is defined as follows. Let $G$ be a graph that allows loops on $n$ vertices with $V(G) = [n]$ and let $\mc S^+(G)$ be the set of all symmetric nonnegative $n \times n$ matrices $A= [a_{ij}]$ with $a_{ij} > 0$ if and only if $\{i, j\} \in E(G)$. We define the SNT-rank of a graph $G$ as
$$\stp(G):=\min\{\stp(A)\colon A \in \mc S^+(G)\}.$$
In this paper, we consider  the gap of the graph $G$,  an equivalent parameter defined as
$$\gap(G):=|V(G)|-\stp(G).$$
If $G$ is an empty graph, then we define $\gap(G)=0$.
Determining the value of $\stp(G)$ reduces to minimizing the number of components required in a specific cover (called set-join cover) of $G$. A formal definition of set-join covers and their components, skipped here for brevity, can be found at the beginning of Section~\ref{Sec:2}. 

 This optimization problem simplifies significantly when restricted to the class of graphs that exclude a specific set of forbidden subgraphs listed in Figure~\ref{fig:forbidden-subgraphs-4-walk}. In such cases, the decomposition requires only the simplest form of set-joins—those corresponding to star graphs.  While our approach uses star covers, we emphasize an essential difference with more common optimization problem:  rather than minimizing the total number of stars, we focus on minimizing the number of components required to achieve the cover. This difference is highlighted on a small graph in our first example below.

\begin{example}\label{ex:W5}
    Let $W_5$ be a wheel graph on five vertices as shown in Figure~\ref{fig:wheel-5}. We highlight four edge covers of this graph, each satisfying a different optimization criterion. 

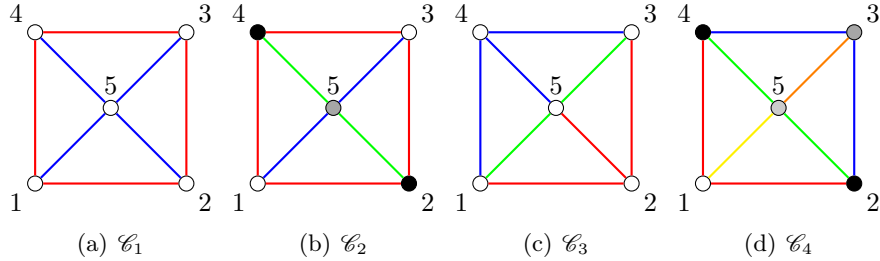
\begin{figure}[h]
\centering

\subfigure[${\mathscr C}_1$]{\begin{tikzpicture}\label{fig:W5-C1}
 \node[fill=white] (5) at (0,0){};
 \node[fill=white] (1) at (-1,-1){};
 \node[fill=white] (2) at (1,-1){};
 \node[fill=white] (3) at (1,1){};
 \node[fill=white] (4) at (-1,1){};
 \draw[thick,red] (1)--(2)--(3)--(4)--(1);
 \draw[thick,blue] (3)--(5)--(1);
 \draw[thick,blue] (2)--(5)--(4);
 \node[rectangle,draw=none] at (-1.25,-1.25) {\small $1$};
  \node[rectangle,draw=none] at (1.25,-1.25) {\small $2$};
   \node[rectangle,draw=none] at (1.25,1.25) {\small $3$};
  \node[rectangle,draw=none] at (-1.25,1.25) {\small $4$};
   \node[rectangle,draw=none] at (0,0.3) {\small $5$};
\end{tikzpicture}}
\subfigure[${\mathscr C}_2$]{\begin{tikzpicture}\label{fig:W5-C2}
 \node[fill=gray!70] (5) at (0,0){};
 \node[fill=white] (1) at (-1,-1){};
 \node[fill=black] (2) at (1,-1){};
 \node[fill=white] (3) at (1,1){};
 \node[fill=black] (4) at (-1,1){};
 \draw[thick,red] (1)--(2)--(3)--(4)--(1);
 \draw[thick,blue] (3)--(5)--(1);
 \draw[thick,green] (2)--(5)--(4);
 \node[rectangle,draw=none] at (-1.25,-1.25) {\small $1$};
  \node[rectangle,draw=none] at (1.25,-1.25) {\small $2$};
   \node[rectangle,draw=none] at (1.25,1.25) {\small $3$};
  \node[rectangle,draw=none] at (-1.25,1.25) {\small $4$};
   \node[rectangle,draw=none] at (0,0.3) {\small $5$};
\end{tikzpicture}}
\subfigure[${\mathscr C}_3$]{\begin{tikzpicture}\label{fig:W5-C3}
 \node (5) at (0,0){};
 \node (1) at (-1,-1){};
 \node (2) at (1,-1){};
 \node (3) at (1,1){};
 \node (4) at (-1,1){};
 \draw[thick,red] (2)--(5);
 \draw[thick,red] (1)--(2)--(3); 
 \draw[thick,blue] (3)--(4)--(1);
 \draw[thick,blue] (5)--(4);
 \draw[thick,green] (1)--(5)--(3);
 \node[rectangle,draw=none] at (-1.25,-1.25) {\small $1$};
  \node[rectangle,draw=none] at (1.25,-1.25) {\small $2$};
   \node[rectangle,draw=none] at (1.25,1.25) {\small $3$};
  \node[rectangle,draw=none] at (-1.25,1.25) {\small $4$};
   \node[rectangle,draw=none] at (0,0.3) {\small $5$};
\end{tikzpicture}}
\subfigure[${\mathscr C}_4$]{\begin{tikzpicture}\label{fig:W5-C4}
 \node[fill=gray!40] (5) at (0,0){};
 \node[fill=white] (1) at (-1,-1){};
 \node[fill=black] (2) at (1,-1){};
 \node[fill=gray!70] (3) at (1,1){};
 \node[fill=black] (4) at (-1,1){};
 \draw[thick,red](4)--(1)--(2);
 \draw[thick,blue] (4)--(3)--(2);
 \draw[thick,green] (4)--(5)--(2);
 \draw[thick,orange] (3)--(5);
 \draw[thick,yellow] (1)--(5);
 \node[rectangle,draw=none] at (-1.25,-1.25) {\small $1$};
  \node[rectangle,draw=none] at (1.25,-1.25) {\small $2$};
   \node[rectangle,draw=none] at (1.25,1.25) {\small $3$};
  \node[rectangle,draw=none] at (-1.25,1.25) {\small $4$};
   \node[rectangle,draw=none] at (0,0.3) {\small $5$};
\end{tikzpicture}}
\caption{Wheel graph $W_5$.}\label{fig:wheel-5}
\end{figure}

For sets of vertices $\mc S_1$ and $\mc S_2$,  we use the notation $\mc S_1\vee \mc S_2$ 
to denote the graph 
with edges $\{\{i_1,i_2\}\colon i_1 \in \mc S_1, i_2 \in \mc S_2\}$, and we say that $\mc S_1$ and $\mc S_2$ are the components of this bipartite graph. 

\begin{enumerate}[$(1)$]
\item The cover $\C_1=\{\{1,3\}\vee \{2,4\}, \{1,2,3,4\}\vee \{5\}\}$ consists of two complete bipartite graphs that together use four different components. This cover is constructed to minimize the number complete bipartite graphs needed to cover all edges of $W_5$. See Figure~\ref{fig:W5-C1}.
\item The cover $\C_2=\{\{1,3\}\vee \{2,4\}, \{1,3\}\vee \{5\}, \{2,4\}\vee \{5\}\}$ consists of three complete bipartite graphs and is constructed to minimize the number of components in any cover of $W_5$ with complete bipartite graphs. Note, the components used in $\C_1$ are  $\{1,3\}$, $\{2,4\}$ and $\{5\}$, hence the resulting minimum value of components needed is three. See~Figure \ref{fig:W5-C2}. 
\item  The cover $\C_3=\{\{4\}\vee \{1,3,5\}, \{2\}\vee \{1,3,5\}, \{5\}\vee \{1,3\}\}$ consists of three stars that that use five components, and minimizes the number of stars in any star cover of $W_5$. 
\item The cover $\C_4=\{\{1\}\vee \{2,4\}, \{1\}\vee \{5\}, \{2,4\}\vee \{3\}, \{2,4\}\vee \{5\}, \{3\}\vee \{5\}\}$ consists of five stars and is minimizing the number of components in any star cover of $W_5$. Note, the cover uses four components, $\{1\}$, $\{3\}$, $\{5\}$, and $\{2,4\}$. See Figure~\ref{fig:W5-C4}. 
\end{enumerate} 

 The cover $\C_1$ is an optimal cover of graph $W_5$ with bipartite graphs, see \cite{MR2519293}. The cover $\C_2$ is an optimal set-join cover of graph $W_5$. As proven in  \cite[Theorem 2.12]{KBS24} this means $\stp(W_5) = 3$. The cover $\C_3$ is an optimal star cover of the graph $W_5$, which means that $\starc(W_5) = 3$. In this paper, we will study covers as in item (4). We will show in Proposition~\ref{prop:gen4cycle} that for a graph $G$ that does not contain forbidden subgraphs, every set-join cover is a star cover, so minimizing the number of components in a star cover computes the SNT-rank of a graph $G$.   
\end{example}

\subsection{Structure of the paper}
We organize the paper as follows. Sections~\ref{Sec:2} and~\ref{sec:3} develop several observations concerning $\gap$: Section~\ref{Sec:2} treats the general case, while Section~\ref{sec:3} focuses on graphs that exclude certain $4$-walk configurations. In particular, Lemma~\ref{lem:P4toP2} shows that, in this restricted setting, $\gap(G)$ depends only on the parity of path lengths in $G$.
Motivated by this result, Section~\ref{sec:definition-kappa} introduces a weighted multigraph $\kappa(G)$ with binomial edge weights, together with the parameter $\gapzvezdica(\kappa(G))$, which coincides with $\gap(G)$.
In Sections~\ref{sec:operations-preserve-gap*} and~\ref{sec:tau}, we study general weighted multigraphs $\Gamma$ with binomial weights and present several graph operations that preserve $\gapzvezdica(\Gamma)$. These operations lead to an algorithm that reduces any such multigraph $\Gamma$ to a weighted simple graph $\tau(\Gamma)$ (possibly with loops), in which all edge weights are zero and all vertex degrees are at least three, while maintaining control over $\gapzvezdica$.
Since $\gapzvezdica(\tau(\Gamma))$ may still be impossible to compute directly, Section~\ref{sec:vtx-removal} introduces an alternative method based on the sequential removal of vertices. The paper concludes with several illustrative examples in Section~\ref{sec:app}.

\subsection{Definitions and notation}
The term graph in this paper will always mean a simple graph $G=(V(G), E(G))$ that allows loops.  Edges $e\in E(G)$ are subsets of $V(G)$ of the form $e=\{u,v\}$.  If the edge is of the form $\{u,u\}$, it means that a vertex $u$ has a loop.
Below, we list some basic notions associated with graphs that we will depend on:
\begin{itemize}
\item By $G \cong H$ we denote that graphs $G$ and $H$ are isomorphic. 
\item Let $P_n$ denote the path on $n$ vertices, $C_n$ a cycle on $n$ vertices, $K_n$ a complete graph on $n$ vertices and $K_{m,n}$ a complete bipartite graph on $m+n$ vertices, all without loops.
\item Vertices $u, v \in V(G)$ are \emph{neighbours} if $\{u,v\} \in E(G)$.
 The neighbourhood of $v$ in $G$, denoted by $N_G[v]$, is the set of all neighbours of $v$. Note that vertex $v \in V(G)$ has a loop if and only if $v \in N_G[v]$.
\item Vertices $u$ and $v$ are \emph{twin vertices} if $N_G[u]=N_G[v]$.
\item If a vertex $v \in V(G)$ does not have a loop, then the \emph{degree of $v$} in $G$ equals the cardinality of its neighbourhood, i.e., $\deg_G(v) = |N_G[v]|$. If a vertex $v \in V(G)$ has a loop, then $\deg_G(v) = |N_G[v]| + 1$, i.e., loops count twice towards the degree.
\item A vertex $v\in V(G)$ is called a \emph{leaf} if $|N_G[u]|=1$ and is called \emph{a high degree vertex} if $\deg_G(v) \geq 3$.  We denote by $\hdv(G)$ the set of all high degree vertices of $G$, by $V_{d=2}(G)$ the set of all degree two vertices, by $\leaves(G)$ the set of all leaves, by  $V_{d=0}(G)$ the set of all vertices of degree $0$, and by  $\degnottwo(G)=V(G)\setminus V_{d=2}(G)$.
\item Removing a set of vertices $\{v_1, \ldots, v_k\}$ from a simple graph $G$ is a well-known operation in which we remove all vertices $v_i$, $i\in [k]$ together with all edges $\{v_i,u\}$, $u\in N_{G}[v_i]$, $i\in [k]$. We denote this graph as $G(\{v_1, \ldots, v_k\})$.
When convenient, we write $G(v_1,\ldots,v_k)$ instead of $G(\{v_1,\ldots,v_k\})$.
Furthermore, the subgraph of $G$ induced on the set $\{v_1, \ldots, v_k\}$ is denoted by $G[\{v_1, \ldots, v_k\}]$.
\item 
Let $k\in \bN_0$. If for $u,v\in V(G)$ and for distinct vertices $x_i\in V_{d=2}(G)\setminus\{u,v\}$, $i\in [k]$, there exists a walk $u-x_1-\ldots-x_k-v$, we abbreviate it by $\pathk{u}{v}{k}$. 
  
\item If $G$ contains $w-\pathk{u}{v}{k}$ for  $k\geq 0$, $w\in \hdv(G)$, $u\in V_{d=2}(G)$ and  $v\in \leaves(G)$, then we call  $\pathk{u}{v}{k}$ a pendent path $P_{k+2}$ at $w$.

\end{itemize}

We follow~\cite{2010-Diestel-Graph-Theory} and consider undirected \emph{multigraphs} $\Gamma=(V(\Gamma), E(\Gamma))$, in which to every $e\in E(\Gamma)$ we assign one or two endpoints from the set of vertices $V(\Gamma)$. 
Note that this definition allows several edges between a pair of distinct vertices and several loops at vertex $u$, which we denote as $\{u,u\}$. If a multigraph $\Gamma$ is \emph{weighted}, then $\Gamma=(V(\Gamma),E(\Gamma),w_{\Gamma})$, where $w_{\Gamma}\colon E(\Gamma)\to \R$ is a function that assigns a weight to each edge. 

If a vertex of a (weighted) multigraph is incident to $s$ loops and  $t$ non-looped edges, then we define its {\emph degree} as $\deg_{\Gamma}(v)=t+2s$. (For example, 
$\deg_{\Gamma}(w)=6$ and $\deg_{\Gamma}(v)=4$ for vertices $w,v\in V(\Gamma)$ in a weighted multigraph $\Gamma$ shown in Figure~\ref{fig:removing-vtx}.)

The \emph{edge contraction along an edge} $e\in E(\Gamma)$ with different endpoints $u$ and $v$  is a graph obtained from $\Gamma$ by identifying vertices $u$ and $v$ (into a single vertex), deleting the edge $e$ and preserving all other incidences (and weights in the case of a weighted multigraph).

Other definitions from simple graphs to multigraphs extend in a natural way.

\section{Initial results on gap}\label{Sec:2}

Before limiting our analysis to graphs without forbidden subgraphs, we offer a few preliminary results on the gap of graphs satisfying a local assumption.

SNT-rank of a graph $G$ can be described combinatorially using set-join covers of the graph. Let $\mc K$ and $\mc L$ be two nonempty subsets of $V(G)$ with possibly nonempty intersection. We define \emph{the set-join of $\mc K$ and $\mc L$}, denoted by $\mc K \vee \mc L$, to be the graph with $V(\mc K \vee \mc L)=V(G)$, and $E(\mc K \vee \mc L)=\{\{i,j\}\colon i \in \mc K, j \in \mc L\}$. 
Note that $\mc K \vee \mc L$ has a loop on vertex $i$ precisely when $i \in \mc K \cap \mc L$. We say that $\C=\{\mc K_i \vee \mc L_i\colon i \in [t]\}$ is \emph{a set-join cover} of $G$ if  
$E(G)=\bigcup_{i=1}^t E(\mc K_i \vee \mc L_i)$. 
For a set-join cover $\C$ of graph  $G$ {the component set} of $\C$ is $$V(\C)=\{\mc K_i\colon i \in [t]\}\cup \{\mc L_i\colon i \in [t]\}$$
and we call each element of $V(\C)$ \emph{a component of $\C$}. \emph{The order} of $\C$ is the number of components of $\C$, i.e.~$|V(\C)|$.
In \cite[Theorem 2.12]{KBS24} it is proven that $$\stp(G)=\min\{|V(\C)|\colon \C \text{ a set-join cover of } G\}.$$
If $\C$ is a set-join cover of $G$ of order $\stp(G)$, then we say that $\C$ is an \emph{optimal set-join cover} of $G$.

We recall selected results from \cite{KBS24}, needed in the sequel. While these were originally stated in terms of $\stp(G)$, we restate them here using $\gap(G)$.

\begin{proposition}\label{prop:from KBS24} Let $G$ be a simple graph. 
    \begin{enumerate}[$(1)$]
        \item \cite[Lemma 3.3]{KBS24} If $v\in V(G)$ is a leaf and $H=G(\{v,w\})$ the graph obtained from $G$ by deleting $v$ and its unique neighbour $w$, then $\gap(G)=\gap(H)$. \label{prop:2.1(1)}
        \item \cite[Proposition 2.28]{KBS24} If $u, u' \in V(G)$ are twin vertices, then $\gap(G)=\gap(G(u)) + 1.$ \label{prop:2.1(2)}
        \item \cite[Example 3.4]{KBS24} $\gap(P_n) = 1$ for odd $n$ and $\gap(P_n) = 0$ for even $n$. \label{prop:2.1(3)}
        \item \cite[Proposition 3.10]{KBS24} $\gap(C_4) = 2$ and $\gap(C_n) = 0$ for $n\ne 4$. \label{prop:2.1(4)}
    \end{enumerate}
\end{proposition}

\begin{remark}\label{thm:tree-max-st+}
Item (\ref{prop:2.1(1)}) in Proposition \ref{prop:from KBS24} reduces the problem of determining the gap for graphs that contain a leaf to a smaller graph, and gives an algorithmic way to compute the gap for trees. 
 If $T$ is a tree with at least two vertices, then repeated application of this result can be used to prove that $\gap(T)=0$ if and only if $T$ is obtained from $P_2$ by consecutively adding a pendent  $P_2$ to any vertex in the existing graph.

To verify this, assume first that $T$ is obtained from $P_2$ by consecutively adding a pendent  $P_2$ to a vertex. Since $\gap(P_2)=0$ by Proposition~\ref{prop:from KBS24}~\eqref{prop:2.1(3)}, it follows by Proposition~\ref{prop:from KBS24}~\eqref{prop:2.1(1)} that $\gap(T)=0$.

Suppose that $\gap(T)=0$. If $|V(T)|=2$, then $T=P_2$ and the statement follows by Proposition~\ref{prop:from KBS24}~\eqref{prop:2.1(3)}. Assume now that $|V(T)|\geq 3$ and that all trees on at most $|V(T)|-1$ vertices of gap $0$ are obtained from $P_2$ by consecutively adding a pendent  $P_2$ to a vertex. Let $v$ and $w$ be two leaves of $T$, such that the distance between them is equal to the diameter of $T$, and let $v_1$ be the unique neighbor of $v$. We claim that $\deg_T(v_1)=2$. Indeed, $\deg_T(v_1)\geq 3$ and the assumption that $T$ does not contain any cycles, implies that there exists $v'\in N_T[v_1]$, $v'\neq v$, with $d(v',w)=d(v,w)=\diam(T)$. 
Note, $v'$ is not a leaf, since the assumptions that $\gap(T)=0$ implies that $T$ has no twin vertices, by Proposition~\ref{prop:from KBS24}~\eqref{prop:2.1(2)}. Therefore there exists a vertex $v_0\in N_T[v']$, such that $d(v_0,w)=d(v',w)+1=\diam(T)+1$, a contradiction. 

In particular there exists a vertex $v_2$ so that $N_T[v_1]=\{v,v_2\}$ and $v-v_1$ is $P_2$ a pendent  path at $v_2$ in $T$. 
By item (\ref{prop:2.1(1)}) in Proposition \ref{prop:from KBS24}, $T(\{v,v_1\})$ is a tree on $|V(T)|-2$ vertices with $\gap(T(\{v,v_1\}))=0$. By the induction hypothesis, $T(\{v,v_1\})$ is obtained from $P_2$ by consecutively adding to a vertex a pendent  $P_2$.  Hence, $T$ is obtained in the same way. 
\end{remark}

In the proposition below, we restrict our attention to graphs that do not contain a certain closed $4$-walk. This assumption will also be adopted in subsequent sections.

\begin{proposition}\label{prop:G1G2}
Let $G$ be a graph containing two vertex disjoint induced subgraphs $G_1$ and $G_2$ so that $V(G)=V(G_1)\cup V(G_2)$, and there do not exist vertices $u_1, v_1 \in V(G_1)$ and $u_2, v_2 \in V(G_2)$ with and $\{u_i,v_j\} \in E(G)$ for all $i,j\in\{1,2\}$. 
Then 
 $$\gap(G) \leq \gap(G_1)+\gap(G_2).$$ 

In particular, if $\gap(G_i)=0$ for $i=1,2$, then $\gap(G)=0$.
\end{proposition}

\begin{proof}
Let $\C$ be an optimal set-join cover of $G$, and let $\C_i \subseteq \C$, $i=1,2$, be the subset of $\C$ containing all the elements of $\C$ that cover at least one edge in $G_i$. Clearly, $|V(\C_i)| \geq \stp(G_i)$. 

We claim that $V(\C_1) \cap V(\C_2)=\emptyset$. Assuming the opposite, let $\mc K \in V(\C_1) \cap V(\C_2)$. Then for $i=1,2$: there exists  $u_i \in \mc K \cap V(G_i)$ and a component $\mc L_i \in V(\C_i)$ so that $\mc K \vee \mc L_i \in \C_i$. In particular, there exist $v_i \in \mc L_i \cap V(G_i)$.  From $\mc K \vee \mc L_i \in \C$ we conclude that $\{u_i,v_j\} \in E(G)$ for all $i,j=1,2$, a contradiction. 

Since $V(\C_1) \cup V(\C_2) \subseteq V(\C)$, this implies $\stp(G) \geq \stp(G_1)+\stp(G_2)$ and thus also $\gap(G) \leq \gap(G_1)+\gap(G_2)$.
\end{proof}

\begin{example} 
Note that the inequality in Proposition~\ref{prop:G1G2} can be strict. Let $G = P_6$ with vertices $v_1, ..., v_6$ in the natural cycle order. Let $G_1$, $G_2$ be induced subgraphs with $V(G_1) = \{v_1, v_2, v_3\}$ and $V(G_2) = \{v_4, v_5, v_6\}$. Then $G_1$ and $G_2$ are both isomorphic to $P_3$, so $\gap(G_1) = \gap(G_2) =1$, and $\gap(G) = 0$.
\end{example}

\begin{corollary} \label{cor:remove-C2k}
 Let $G$ be a graph containing an induced subgraph $H$ isomorphic to $C_{2k}$, $k \geq 3$, where vertices $V(H)=\{v_1,v_2,\ldots,v_{2k}\}$ are indexed in a natural cycle order, and $\{v_2,v_4,\ldots,v_{2k}\} \subseteq V_{d=2}(G)$. 
Then $$\gap(G)=\gap(G(V(H))).$$ 
\end{corollary}

\begin{proof}
  Let $W=V(H)$. Observe first that  we have $\gap(G)\le \gap(G(W))+\gap(H)=\gap(G(W))$ since $H$ and $G(W)$ satisfy the hypothesis of Proposition~\ref{prop:G1G2}.
Now, let $\C_W$ be an optimal set-join cover of $G(W)$, and $\C_1:=\{\{v_{2t-1}\}\vee N_G[v_{2t-1}]\colon t\in[k]\}$. Since $\C:=\C_W\cup \C_1$ is a set-join cover of $G$ with $|V(\C)| = |V(\C_W)|+|V(\C_1)|=\stp(G(W))+2k=|V(G)|-\gap(G(W))$, it follows that $\gap(G)\ge \gap(G(W))$.
\end{proof}

\section{On graphs without four cycles}\label{sec:3}

Now we limit our investigation to the family of graphs described in the following definition. 
\begin{definition}
By $\GnoFourCycle$ we denote the family of all graphs $G$ 
with no 
vertices $u_1, u_2, v_1, v_2 \in V(G)$, $u_1 \ne u_2$ and $v_1 \ne v_2$, such that $\{u_i, v_j\} \in E(G)$ for all $i,j \in \{1,2\}$. 
\end{definition}

The family $\GnoFourCycle$ can also be characterised with the set of forbidden subgraphs given in Figure~\ref{fig:forbidden-subgraphs-4-walk}. Note, all three graphs are closed $4$-walks in $G$.

\begin{figure}[h]
\centering
\subfigure[$C_4$]{\begin{tikzpicture}[scale=0.9]
\draw (0,0)--(0,2)--(2,2)--(2,0)--(0,0);
 \foreach \i in {0,2} {
 \foreach \j in {2,0} {
    \node[fill=white] at (\i,\j) {}; 
 }
 }
 \node[rectangle,draw=none] at (-0.3,-0.3) {$u_1$};
 \node[rectangle,draw=none] at (2.3,-0.3) {$v_1$};
 \node[rectangle,draw=none] at (2.3,2.3) {$u_2$};
 \node[rectangle,draw=none] at (-0.3,2.3) {$v_2$};
\end{tikzpicture}}
\qquad
\subfigure[$C_3^{\circ}$]{\begin{tikzpicture}[scale=0.9]
\draw (0,0)--(2,0)--(1,1.5) -- (0,0);
\tikzset{every loop/.style={min distance=18mm,in=60,out=120,looseness=30}}
\path (1,1.5) edge  [loop above]  ();
 \foreach \i in {(0,0),(2,0),(1,1.5)} {
    \node[fill=white] at \i {}; 
 }
 \node[rectangle,draw=none] at (-0.3,-0.3) {$u_1$};
 \node[rectangle,draw=none] at (2.3,-0.3) {$v_1$};
 \node[rectangle,draw=none] at (2,1.5) {$u_2=v_2$};
\end{tikzpicture}}
\qquad
\subfigure[$P_2^{\circ,\circ}$]{\label{fig:P2-OO}\begin{tikzpicture}[scale=0.9]
\draw (0,0)--(1,0);
\tikzset{every loop/.style={min distance=18mm,in=30,out=330,looseness=50}}
\path (1,0) edge  [loop right]  ();
\tikzset{every loop/.style={min distance=18mm,in=150,out=210,looseness=50}}
\path (0,0) edge  [loop left]  ();
 \foreach \i in {(0,0),(1,0)} {
    \node[fill=white] at \i {}; 
 }
  \node[rectangle,draw=none] at (0,-0.6) {$u_1=v_1$};
 \node[rectangle,draw=none] at (1.2,0.6) {$u_2=v_2$};
\end{tikzpicture}}
\caption{Forbidden subgraphs of graphs in $\GnoFourCycle$.}\label{fig:forbidden-subgraphs-4-walk}
\end{figure}

Set-join covers of graphs in $\GnoFourCycle$ have special properties listed in the proposition below. 

\begin{proposition}\label{prop:gen4cycle}
 Let $G \in \GnoFourCycle$. Any set-join cover  $\C$ of $G$ has the following properties:
 \begin{enumerate}[$(1)$]
 \item All elements in $\C$  are of the form $\{v\}\vee \mc K$ for some $v \in V(G)$.\label{lem:1}
 \item Let $\mc K \in V(\C)$ with $|\mc K|\geq 2$ and $\{v_1\}\vee \mc K \in \C$. Then:
 \begin{enumerate}[$(a)$]
 \item $\{v_1\}\vee \mc K$ is the only element of  $\C$ with a component $\mc K$. \label{lem:2a}
\item If $\C$ is an optimal set-join cover of $G$ and $\{v_1\}\vee \mc L \in \C$ for $\mc L \not= \mc K$, then $|\mc L|=1$.\label{lem:2b}
 \end{enumerate}
 \item If $\C$ is an optimal set-join cover of $G$, then any two components in $V({\C})$ intersect in at most one vertex of $G$.
 \end{enumerate}
\end{proposition}

\begin{proof}
To prove item (1), suppose $\mc K \vee \mc L \in \C$, with $|\mc K| \geq 2$, $|\mc L| \geq 2$. Then there exist $u_1, u_2 \in \mc K$  and $v_1, v_2 \in \mc L$,  $u_1 \ne u_2$ and $v_1 \ne v_2$, so that $\{u_i, v_j\} \in E(G)$ for all $i,j=1,2$, contradicting  the assumption $G\in \GnoFourCycle$.

For item (2) assume that $\{v_1\}\vee \mc K \in \C$ with $u_1, u_2 \in \mc K$,  $u_1 \ne u_2$. If  $\{v_1\} \vee \mc K, \{v_2\}\vee \mc K \in \C$ and $v_1 \ne v_2$, then $\{u_i, v_j\} \in E(G)$ for all $i,j=1,2$. As this contradicts $G\in \GnoFourCycle$, (\ref{lem:2a}) follows. 
If  $\{v_1\}\vee \mc L \in \C$, with $|\mc L|\ge 2$ and ${\mc L}\ne {\mc K}$, then $\{v_1\}\vee \mc L$ is  the only element in the optimal set-join cover $\C$ with a component $\mc L$  by~(\ref{lem:2a}). Hence, $\C' = \C \cup \{\{v_1\}\vee (\mc K \cup \mc L)\} \setminus\{\{v_1\}\vee \mc K, \{v_1\}\vee \mc L\}$ is a set-join cover of $G$ satisfying $|V(\C')| = |V(\C)|-1$, contradicting the optimality assumption for $\C$, and proving~(\ref{lem:2b}).

For item (3) suppose that $u_1, u_2 \in \mc K \cap \mc L$ with $u_1 \ne u_2$ and $\mc K \vee \{v_1\}, \mc L \vee \{v_2\} \in \C$. Then $\{u_i, v_j\} \in E(G)$ for all $i,j=1,2$. If $v_1 = v_2$ this contradicts \eqref{lem:2b}, if $v_1 \ne v_2$ this contradicts $G\in \GnoFourCycle$. 
\end{proof}

\begin{corollary}\label{cor:2c2d}
For $G \in \GnoFourCycle$, let $\C$ be a set-join cover of $G$, $\mc K \in V(\C)$ with $|\mc K|\geq 2$, and $\{v\}\vee \mc K \in \C$. Let $\C_1$ be obtained from $\C$ by removing all elements with a component $\{v\}$, and by removing $v$ from all remaining components. Then $\C_1$ is a set-join cover of $G(v)$ satisfying $|V(\C_1)|\leq |V(\C)|-2$. Furthermore, $\widehat{\C}:=\C_1 \cup \{\{v\}\vee N_G[v]\}$ is a set-join cover of $G$.
\end{corollary}

\begin{proof}
  Clearly, $\C_1$ is a set-join cover of $G(v)$. Since $\{v\}\vee \mc K \in \C$ is the unique element of $\C$ with a component $\mc K$ by~Proposition \ref{prop:gen4cycle}, $V(\C_1)$ does not contain $\{v\}$ nor $\mc K$, and the inequality $|V(\C_1)|\leq |V(\C)|-2$ follows. The second claim is straightforward.   
\end{proof}

Consider $G \in \GnoFourCycle$ with a set-join cover $\C$. Let $V_{1}(\C)$ be the set of all vertices $v \in V(G)$ with $\{v\}\in V(\C)$, and let $$N^1_{\C}[x]:=N_G[x]\cap V_{1}(\C),\, N^{\star}_{\C}[x]:=N_G[x]\setminus V_{1}(\C)$$
for all $x\in V(G)$. Let $\C[x]$ be the subset of  $\C$ that contains all the set-joins of $\C$ with at least one component containing $x$. Note that $\C[x]$ covers all the edges that are incident to $x$. We define three different types of vertices $x \in V(G)$ with respect to $\C$: 
\begin{enumerate}
\item[{\bf Type A:}] $|N^{\star}_{\C}[x]| \geq 2$ and $\C[x]=\{\{x\}\vee N^{\star}_{\C}[x]\}\cup \{\{x\}\vee \{v\}\colon v \in N^1_{\C}[x]\}$.  
\item[{\bf Type B:}] $N^{\star}_{\C}[x]=\emptyset$ and $\C[x]=\{\{x\}\vee \{v\}\colon v \in N^1_{\C}[x]\}$.  
\item[{\bf Type C:}] $x \not\in V^1(\C)$ and $\C[x]=\{N^{\star}_{\C}[v]\vee \{v\}\colon v \in N_{G}[x]\}$. 
\end{enumerate}

\begin{definition}
    An optimal set-join cover $\C$ in which every $x \in V(G)$ is of Type A, of Type B or of Type C with respect to $\C$,  
    is called an \emph{\text{$\mathrm{ABC}$} cover}.
\end{definition}

For a general set-join cover $\C$, there can be vertices that are not of any of the three types (A, B or C)  with respect to $\C$. The following proposition shows that any graph in $\GnoFourCycle$ has at least one $\mathrm{ABC}$ cover.

\begin{proposition}\label{prop:Types}
  For every $G \in \GnoFourCycle$ there exists an $\mathrm{ABC}$ cover of $G$.
\end{proposition}

\begin{proof}
 Starting with a set-join  cover $\C_0$ of $G$, we will construct a set-join cover $\C$ of $G$ with the desired property and $|V(\C)|\leq |V(\C_0)|$. 
 
 First consider $x\in V_1(\C_0)$ with $\C_0$ containing a set-join of the form $\mc K \vee \{x\}$ with $|\mc K| \geq 2$. Construct $\C_1$ by replacing all the set-joins in $\C_0[x]$ with $\{\{x\}\vee N^{\star}_{\C_0}[x]\}\cup \{\{x\}\vee \{v\}\colon v \in N^1_{\C_0}[x]\}$. Note that this replacement does not increase the number of components in the cover since $x,v \in V_1(G)$ and if $\{x\}\vee \mc K_j \in \C_0$ for $j\in [t]$, then $\{\mc K_j\}_{j\in [t]}$ are substituted by $N^{\star}_{\C_0}[x]$. If $|N^{\star}_{\C_0}[x]|\geq 2$, then $x$ is now of Type A with respect to $\C_1$, otherwise it is of Type B. Furthermore, all vertices in $G$ that are of Type B with respect to $\C_0$ remain of Type B with respect to $\C_1$, and all vertices in $G$ that are of Type A with respect to $\C_0$ either remain of Type A with respect to $\C_1$ or become Type B with respect to $\C_1$. 

Let $\C$ be the set-join that we obtain after repeatedly performing the above replacement for all $x \in G$ that are joined with a component that is not a singleton. In particular, all $x \in V^1(\C)$  are of Type A or of Type B, and all components of $\C$ that are not a singleton are of the form $N^{\star}_{\C}[v]$ for some $v \in V^1(\C)$. Hence, all $x \in V(G)$ that are not in $V^1(\C)$ are of Type C with respect to $\C$.
\end{proof}

\begin{remark}\label{rem:ABC}
 Below we list a few straightforward observations for an $\mathrm{ABC}$ cover $\C$   of $G \in \GnoFourCycle$.
\begin{itemize}
\item If $x \in V_1(\C)$, then $x$ is not part of any component in $V(\C)$ that contains more than one element. Hence, $V(G)\setminus V_1(\C)$ is equal to the set of all Type C vertices with respect to $\C$. 
\item If $x$ is of Type A, then it has at least two neighbors of Type C. 
\item If $x$ is of Type C, then all its neighbors are of Type A. 
In particular, two vertices of Type C cannot be neighbors. 
\end{itemize}
\end{remark}

\begin{lemma}\label{lem:remove-x1-x2}
   Let $H \in \GnoFourCycle$ and  $\C$ an $\mathrm{ABC}$ cover of $H$. If for $u,v\in V(H)$ and $x_1,x_2\in V_{d=2}(H)\setminus\{u,v\}$ there exists  $u-x_1-x_2-v$ in $H$, then
      \begin{itemize}
       \item $x_1$ is of Type $\mathrm{C}$ and $x_2$ is of Type $\mathrm{A}$ with respect to  $\C$ or vice versa, or
       \item $x_1$ and $x_2$ are both of Type $\mathrm{B}$ in $\C$. 
   \end{itemize}
   In particular, when $u = v$, both $x_1$ and $x_2$ are of Type $\mathrm{B}$ in $\C$.
\end{lemma}

\begin{proof}
 We consider the Type of $x_1$ with respect to $\C$. If $x_1$ is of Type A, then $x_2\not\in V_1(\C)$, hence it is of Type C. Note that in this case $u$ must be of Type C and $v$ of Type A, and so $u\ne v$. Similarly, if $x_1$ is of Type C, then $x_2\in V_1(\C_0)$ and it must be of Type A. In this case $u$ must be of Type A and $v$ of Type C, so again $u\ne v$. Finally, if $x_1$ is of Type B, then $x_2$ must also be of Type B. It follows that in the case $u=v$ the only possible option is that $x_1$ and $x_2$ are both of Type B.
\end{proof}

The following lemma enables us to establish families of graphs that have the same gap, as detailed in Section \ref{sec:definition-kappa}.

\begin{lemma} \label{lem:P4toP2} Let $G \in \GnoFourCycle$ be a graph with a walk  $u- x_1-x_2-v$ for $u,v\in V(H)$ and $x_1,x_2\in V_{d=2}(H)\setminus\{u,v\}$.
Let $G'$ be a graph obtained from $G$ by removing vertices $x_1$ and $x_2$ and adding an edge between $u$ and $v$. Then $\gap(G) \le \gap(G')$. 
Furthermore, if $G' \in \GnoFourCycle$, then $\gap(G) = \gap(G')$.
\end{lemma}

\begin{proof}
Let  $\C$ be an $\mathrm{ABC}$ cover of $G$. Using Lemma \ref{lem:remove-x1-x2} and symmetry, we need to consider two cases.

\noindent{\bf Case 1:} $x_1$ is of Type A and $x_2$ is of Type C. In this case we have $\{\{x_1\}, \{u,x_2\}, \{v\}, N^{\star}_{\C}[v]\} \subseteq V(\C)$ and $x_2 \in N^{\star}_{\C}[v]$. To construct a set-join cover $\C'$ of $G'$ from $\C$ we remove the set-join $\{x_1\} \vee \{u,x_2\}$ from $\C$ and replace  $\{v\} \vee N^{\star}_{\C}[v]$ by $\{v\} \vee (N^{\star}_{\C}[v]\setminus\{x_2\}\cup\{u\})$. Since $|V(\C')|=|V(\C)|-2$ we have $\stp(G')\le \stp(G)-2$  and the conclusion follows.

\noindent{\bf Case 2:} Both $x_1$ and $x_2$ are of Type B. In this case $u,v,x_1,x_2 \in V_1(\C)$. A set-join cover $\C'$ of $G'$ is obtained by removing set-joins $\{u\}\vee \{x_1\}$, $\{x_1\}\vee \{x_2\}$, $\{x_2\}\vee \{v\}$ from $\C$ and adding a set-join $\{u\}\vee \{v\}$. Again we have $|V(\C')|=|V(\C)|-2$ and the conclusion follows.

Suppose now that $G' \in \GnoFourCycle$ and  $\C'$ is an $\mathrm{ABC}$ cover of $G'$. If $u$ and $v$ are both in $V_1(\C')$, then $\{u\} \vee \{v\} \in \C'$ and $$\C = \C' \setminus \{\{u\} \vee \{v\}\} \cup \{\{u\} \vee \{x_1\}, \{x_1\} \vee \{x_2\}, \{x_2\} \vee \{v\}\}$$
is a set-join cover of $G$ of order $|V(\C')|+2$. Otherwise, either $u$ or $v$ is of Type C. Without loss of generality we suppose that $u$ is of Type C and then $v$ is of Type A. It follows that $u \in N^{\star}_{\C'}[v]$ and 
$$\C = \C' \setminus \{\{v\} \vee N^{\star}_{\C'}[v]\} \cup \{\{v\} \vee (N^{\star}_{\C'}[v]\setminus\{u\}\cup\{x_2\}),  \{x_1\} \vee \{u, x_2\}\}$$
is a set-join cover of $G$ of order $|V(\C')|+2$.
This proves $\stp(G) \leq \stp(G') +2$, or equivalently, $\gap(G') \leq \gap(G)$.
\end{proof}

Note, in Lemma \ref{lem:P4toP2} the assumption $G \in \GnoFourCycle$ excludes the possibility $\{u,v\} \in E(G)$. Furthermore, in the case $u=v$, $G'$ contains a loop at $u$.

\section{A weighted multigraph corresponding to a simple graph}
\label{sec:definition-kappa}

Lemma~\ref{lem:P4toP2} enables the substitution of an induced path in a graph $G \in \GnoFourCycle$ with another path of the same length-parity, while preserving the gap. For example, the graphs in Figures~\ref{fig:example-Smmax-graph} and~\ref{fig:example-Smmax-graph1} have the same $\gap$. This motivates our next definition, where we introduce a multigraph $\kappa(G)$ with binomial weights (see Figure~\ref{fig:chdv-example-Smmax3}), in which edges of weight $0$ represent even paths in $G$ and edges of weight $1$ in $\kappa(G)$ represent odd paths in $G$. 

\begin{definition}
For a graph $G=(V(G),E(G)) \in \GnoFourCycle$, we define $\chdv(G)$ to be a weighted multigraph $\chdv(G)=\Gamma=(V(\Gamma),E(\Gamma),w_\Gamma)$ with binomial weights $w_{\Gamma}\colon E(\Gamma) \to \{0,1\}$, defined by:
\begin{itemize}
\item $V(\Gamma)=\degnottwo(G)$,
\item  For any $u,v\in V(\Gamma)$, the set of $\pathk{u}{v}{k}$ in $G$, $k\in \bN_0$, is in bijective correspondence with edges  $e\in E(\Gamma)$ with endpoints $u$ and $v$ in $\Gamma$ and 
$w_{\Gamma}(e)=k \!\! \mod 2$.
\end{itemize}
In a weighted multigraph $\Gamma$ an edge of weight 1 is called \emph{$1$-edge} and an edge of weight 0 is called \emph{$0$-edge}; \emph{$1$-loops} and \emph{$1$-loops} are defined equivalently.
\end{definition}

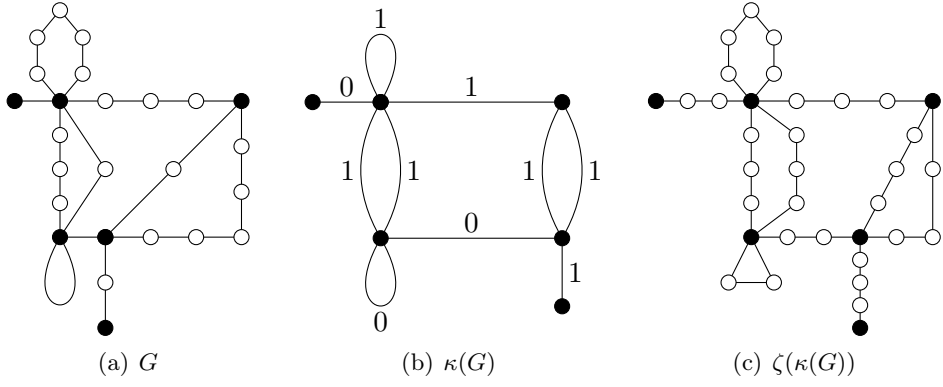
\begin{figure}[h]
\centering
\subfigure[$G$]{\begin{tikzpicture}[scale=0.6]\label{fig:example-Smmax-graph}
\draw (3,3)--(4,4.5)--(3,6)--(7,6)--(7,3)--(3,3);
\draw (4,1)--(4,2) -- (4,3)--(7,6);
\draw (3,3)--(3,6);
\draw (2,6)--(3,6);
\draw (3,6)--(3.5,6.6)--(3.5,7.4)--(3,8)--(2.5,7.4)--(2.5,6.6)--(3,6);
 \foreach \i in {4,5,6} {
 \foreach \j in {3,6} {
    \node[fill=white] at (\i,\j) {}; 
 }
 }
\node[fill=white] at (7,3) {}; 
\node[fill=white] at (4,2) {}; 
\foreach \i in {(7,6),(3,6),(2,6),(4,1),(3,3),(4,3)} {
    \node[fill=black] at \i {}; 
}
 \foreach \j in {4,5} {
    \node[fill=white] at (7,\j) {}; }
 \foreach \j in {3.75,4.5,5.25} {
    \node[fill=white] at (3,\j) {}; 
 }
\foreach \i in {4,5.5} {
    \node[fill=white] at (\i,4.5) {}; 
 }
 \node[fill=white] at (2.5,7.4) {};  
 \node[fill=white] at (3.5,7.4) {};  
 \node[fill=white] at (2.5,6.6) {};  
 \node[fill=white] at (3.5,6.6) {};  
 \node[fill=white] at (3,8) {};  
\tikzset{every loop/.style={min distance=23mm,in=240,out=300,looseness=50}}
\path (3,3) edge  [loop below] (3,6);
\end{tikzpicture}}
\subfigure[$\chdv(G)$]{\begin{tikzpicture}[scale=0.6]\label{fig:chdv-example-Smmax3}
\tikzset{
every node/.style={inner sep=2pt}
}
\path (3,3) edge node [above]  {$0$} (7,3)  edge [bend left] node [left] {$1$} (3,6) edge [bend right] node [right] {$1$} (3,6);
\path (7,6) edge node [above]  {$1$} (3,6)  edge [bend left] node [right] {$1$} (7,3) edge [bend right] node [left] {$1$} (7,3);
\path (7,3) edge node [right]  {$1$} (7,1.5) ;
\path (3,6) edge node [above]  {$0$} (1.5,6) ;
\tikzset{every loop/.style={min distance=23mm,in=240,out=300,looseness=50}}
\path (3,3) edge  [loop below] node {$0$} ();
\tikzset{every loop/.style={min distance=23mm,in=120,out=60,looseness=50}}
\path (3,6) edge  [loop above] node {$1$} ();
\tikzset{
every node/.style={draw, circle, inner sep=2pt}
}
\foreach \i in {(7,6),(3,3),(3,6),(7,1.5),(1.5,6),(7,3),(3,6)} {
    \node[fill=black] at \i {}; 
}
\phantom{  \node[fill=black] at (0.5,3) {}; 
\node[fill=black] at (8.5,3) {}; }
\end{tikzpicture}}
\subfigure[$\zeta(\kappa(G))$]{\begin{tikzpicture}[scale=0.6]\label{fig:example-Smmax-graph1}
\draw (3,3)--(2.5,2)--(3.5,2)--(3,3)--(4,3.75)--(4,5.25)--(3,6)--(7,6)--(7,3)--(3,3);
\draw (5.4,1)-- (5.4,3)--(7,6);
\draw (3,3)--(3,6);
\draw (0.9,6)--(3,6);
\draw (3,6)--(3.5,6.6)--(3.5,7.4)--(3,8)--(2.5,7.4)--(2.5,6.6)--(3,6);
 \foreach \i in {4,5,6} {
    \node[fill=white] at (\i,6) {}; 
 }
\node[fill=white] at (7,3) {}; 
\node[fill=white] at (5.4,2.5) {}; 
\node[fill=white] at (5.4,2) {}; 
\node[fill=white] at (5.4,1.5) {}; 
\node[fill=white] at (5.8,3.75) {};
\node[fill=white] at (6.6,5.25) {};
 \node[fill=white] at (2.5,7.4) {};  
 \node[fill=white] at (3.5,7.4) {};  
 \node[fill=white] at (2.5,6.6) {};  
 \node[fill=white] at (3.5,6.6) {};  
 \node[fill=white] at (3,8) {};  
 \node[fill=white] at (2.3,6) {};  
 \node[fill=white] at (1.6,6) {};  
 \node[fill=black] at (0.9,6) {};  
\foreach \i in {(7,6),(3,6), (5.4,1),(3,3),(5.4,3)} {
    \node[fill=black] at \i {}; 
}
\foreach \i in {(3.8,3),(4.6,3),(6.2,3)} {
    \node[fill=white] at \i {}; 
}
 \foreach \j in {3.75,4.5,5.25} {
    \node[fill=white] at (3,\j) {}; 
     \node[fill=white] at (4,\j) {}; 
 }
\foreach \i in {6.2,7} {
    \node[fill=white] at (\i,4.5) {}; 
 }
\foreach \i in {2.5,3.5} {
    \node[fill=white] at (\i,2) {}; 
 }
\end{tikzpicture}}
\caption{An example of graph $G$, its weighted multigraph $\chdv(G)$ and the graph $\zeta(\chdv(G))$.
 All high degree vertices and leaves of simple graphs are  colored black.}\label{fig:compressed-hdv-graph} 
\end{figure}

Note that $\kappa(P_n)$ is $P_2$ with the only edge of weight $1$ if $n$ is odd and of weight $0$ if $n$ is even.
Moreover, a cycle $C_n$ has $V(C_n)=V_{d=2}(C_n)$ and so $\kappa(C_n)$ is an empty graph, $V(\kappa(C_n))=\emptyset$.

\begin{definition}
Let $\Gamma$ be a weighted multigraph with binomial weights. Let $\zeta(\Gamma)$ be a simple graph (without loops) obtained from $\Gamma$ in the following way:
\begin{itemize} 
\item $V(\Gamma)\subseteq V(\zeta(\Gamma))$. The set $V(\zeta(\Gamma)) \setminus V(\Gamma)$ contains precisely the vertices, defined in the following items.  
\item Substitute each edge $e=\{u,v\} \in E(\Gamma)$ of weight $w_{\Gamma}(e)=0$ by $\pathk{u}{v}{2}$. This operation adds two new vertices $x_1$ and  $x_2$ to $V(\zeta(G))$, where $\deg_{\zeta(\Gamma)}(x_j)=2$ for $j\in[2]$.
\item Substitute each edge $e=\{u,v\} \in E(\Gamma)$, where $u\ne v$ and $w_{\Gamma}(e)=1$, by $\pathk{u}{v}{3}$. This operation adds three new vertices $x_1,x_2,x_3$ to $V(\zeta(\Gamma))$, where $\deg_{\zeta(\Gamma)}(x_j)=2$ for $j\in[3]$.
\item Substitute each loop $e=\{u,u\}\in  E(\Gamma)$ of weight $w_{\Gamma}(e)=1$ by $\pathk{u}{u}{5}$. This operation adds five new vertices $x_1,\ldots,x_5$ to $V(\zeta(\Gamma))$ and, where $\deg_{\zeta(\Gamma)}(x_j)=2$ for $j\in[5]$.
\end{itemize} 
\end{definition}

An illustration of mappings $\kappa$  and $\zeta$ is given in Figure~\ref{fig:compressed-hdv-graph}.

Observe that $\zeta(\kappa(G))$ may contain more or less vertices than $G$. Furthermore, $\zeta(\kappa(G))\in \GnoFourCycle$ and has no loops. Therefore, by Lemma~\ref{lem:P4toP2} we have
\begin{equation}\label{eq:gap-zeta}\gap(G) = \gap(\zeta(\chdv(G))).
\end{equation}
Note that the mapping $\kappa$ is not injective, but $\zeta$ is. Hence, we define the $\gap$ of a multigraph $ \Gamma$ as the $\gap$ of its image via $\zeta$ mapping.

\begin{definition}
 For every multigraph $\Gamma$ with binomial weights we define $$\gapzvezdica(\Gamma):=\gap(\zeta(\Gamma)).$$
\end{definition}

Note that the definition of $\gapzvezdica$ and~\eqref{eq:gap-zeta} imply 
\begin{equation}\label{eq:gap-gapzvezdica}
    \gap(G) = \gapzvezdica(\kappa(G))
\end{equation}
for $G\in \GnoFourCycle$. We use this equality to compute the gap of a graph $G\in \GnoFourCycle$, and in order to do this, we present operations on $\Gamma$ that preserve $\gapzvezdica(\Gamma)$ in the following section.

\section{Operations on $\Gamma$ that preserve $\gapzvezdica$}\label{sec:operations-preserve-gap*}

This section examines how the parameter $\gapzvezdica$ behaves under several fundamental operations on weighted multigraphs. In each case, we show that the $\gapzvezdica$ is preserved, and these operations
will later enable us to design an algorithm for computing $\gapzvezdica$.

\subsection{Removing a vertex from $\Gamma$} To remove a vertex $v\in V(\Gamma)$ from the multigraph $\Gamma$ together with all edges incident to it, we follow the notation used for simple graphs and denote the resulting graph by $\Gamma(v)$. Note that $\gap(\zeta(\Gamma(v)))$ is not necessarily equal to $\gap(\zeta(\Gamma)(v))$, so we define a different way of vertex removal from a multigraph $\Gamma$ with binomial weights, which we denote by $\Gamma_{-v}$.

\begin{definition}
 Let $\Gamma=(V(\Gamma),E(\Gamma),w_\Gamma)$ be  a weighted multigraph  with binomial weights and let $v\in V(\Gamma)$. 
 Let the edges in $E(\Gamma)$ that are incident to $v$ be:
\begin{itemize}
\item $\ell_j$ loops with weight  $j\in\{0,1\}$,
\item $f_j$ edges incident to leaves with weight  $j\in\{0,1\}$. We denote the leaves incident to edges of weight $1$ by $v_i$, $i\in [f_1]$.  
\item $k_j$ edges  incident to non-leaf vertices with weight $j\in\{0,1\}$. The edges of weight $1$ we denote by $e_{i}=\{u_i,v\}$, $j\in [k_1]$. 
\end{itemize}
We define multigraph $\Gamma_{-v}$ as
\begin{align*} 
V(\Gamma_{-v})&=(V(\Gamma(v))\setminus\{v_i\colon i \in [f_1]\}) \cup \{v'_i\colon i\in [\ell_1]\} \cup \{u'_i\colon i \in [k_1]\},\\
E(\Gamma_{-v})&=E(\Gamma(v))\cup \{e'_i=\{u_i,u'_i\}  \colon i \in [k_1] \} \text { and}\\
 w_{\Gamma_{-v}}(e)&=\begin{cases}
 0,& \text{if }e\in \{\{u_i,u'_i\} \colon i \in [k_1] \},\\
 w_{\Gamma}(e), & \text{otherwise.}
 \end{cases}
\end{align*}
\end{definition}

\begin{remark}
To construct the multigraph $\Gamma_{-v}$ from $\Gamma$, we first remove the vertex $v$  with all the incident edges, which results in the graph $\Gamma(v)$.  Note that $V(\Gamma(v))$ has $f_0+f_1$ new vertices of degree $0$. To obtain $\Gamma_{-v}$:
\begin{itemize}
\item Remove $f_ 1$ of vertices of degree $0$. 
\item Add vertices $v_i'$, $i \in [\ell_1]$, to  $V(\Gamma(v))$. These are kept as vertices of degree $0$.
\item  Add vertices $u'_i$, $i\in[k_1]$, to $V(\Gamma(v))$. Note that a vertex $u\in V(\Gamma)$ may appear several times in the multiset $\{u_i\colon i\in [k_1]\}$, but $u'_i$, $i\in [k_1]$, are all distinct in $V(\Gamma_{-v})$.
\item $E(\Gamma_{-v})$ consists of all the edges in $E(\Gamma(v))$ with the same weights as in $\Gamma$, together with new edges $\{e'_i=\{u_i,u'_i\}  \colon i \in [k_1] \} $. 
\end{itemize}
\end{remark}

\begin{example} Consider the weighted multigraph $\Gamma=\chdv(G)$ for the graph $G$ from Figure~\ref{fig:compressed-hdv-graph} and  vertices $v,w\in V(\Gamma)$ as shown in Figure~\ref{fig:chdv-example-Smmax4}. Multigraphs $\Gamma_{-v}$ and $\Gamma_{-w}$ are shown in Figures~\ref{fig:example-remove-v} and~\ref{fig:example-remove-w}.

\begin{remark}\label{rem:gamma-v}
 We acknowledge that the operation $\Gamma_{-v}$ is technical, and we justify its inclusion by Lemma~\ref{lem:zetaGamma-v}.  
However, the operation simplifies if $v$ is incident only to edges of weight $0$. In this case,  $\Gamma_{-v} = \Gamma(v)$. 
\end{remark}

 \begin{figure}[htb]
\centering
\subfigure[$\Gamma$]{
\begin{tikzpicture}[scale=0.6]\label{fig:chdv-example-Smmax4}
\tikzset{
every node/.style={inner sep=2pt}
}
\path (3,3) edge node [above]  {$0$} (7,3)  edge [bend left] node [left] {$1$} (3,6) edge [bend right] node [right] {$1$} (3,6);
\path (7,6) edge node [above]  {$1$} (3,6)  edge [bend left] node [right] {$1$} (7,3) edge [bend right] node [left] {$1$} (7,3);
\path (7,3) edge node [right]  {$1$} (7,1.5) ;
\path (3,6) edge node [above]  {$0$} (1.5,6) ;
\tikzset{every loop/.style={min distance=23mm,in=240,out=300,looseness=50}}
\path (3,3) edge  [loop below] node {$0$} ();
\tikzset{every loop/.style={min distance=23mm,in=120,out=60,looseness=50}}
\path (3,6) edge  [loop above] node {$1$} ();
\tikzset{
every node/.style={draw, circle, inner sep=2pt}
}
\foreach \i in {(7,6),(3,3),(3,6),(7,1.5),(1.5,6)} {
    \node[fill=black] at \i {}; 
}
\node[rectangle,draw=none] at (3.5,6.3) {\small $w$};
\node[rectangle,draw=none] at (7.5,3) {\small $v$};
 \node[fill=cyan] at  (7,3) {}; 
  \node[fill=purple] at  (3,6) {}; 
\end{tikzpicture}}
\subfigure[$\Gamma_{-v}$]{\begin{tikzpicture}[scale=0.6]\label{fig:example-remove-v}
\tikzset{
every node/.style={inner sep=2pt}
}
\path (3,3)  edge [bend left] node [left] {$1$} (3,6) edge [bend right] node [right] {$1$} (3,6);
\path (7,6) edge node [above]  {$1$} (3,6)  edge [bend left] node [right] {$0$} (7.5,3.5) edge [bend right] node [left] {$0$} (6.5,3.5);
\path (3,6) edge node [above]  {$0$} (1.5,6) ;
\tikzset{every loop/.style={min distance=23mm,in=240,out=300,looseness=50}}
\path (3,3) edge  [loop below] node {$0$} ();
\tikzset{every loop/.style={min distance=23mm,in=120,out=60,looseness=50}}
\path (3,6) edge  [loop above] node {$1$} ();
\tikzset{
every node/.style={draw, circle, inner sep=2pt}
}
\foreach \i in {(7,6),(3,6),(3,3),(6.5,3.5),(7.5,3.5),(1.5,6)} {
    \node[fill=black] at \i {}; 
}
\end{tikzpicture}}
\subfigure[$\Gamma_{-w}$]{\begin{tikzpicture}[scale=0.6]\label{fig:example-remove-w}
\tikzset{
every node/.style={inner sep=2pt}
}
\path (3,3)  edge [bend left] node [left] {$0$} (2.5,5.5) edge [bend right] node [right] {$0$} (3.5,5.5);
\path (7,6) edge node [above]  {$0$} (3.5,6)  edge [bend left] node [right] {$1$} (7,3) edge [bend right] node [left] {$1$} (7,3);
\path (7,3) edge node [right]  {$1$} (7,1.5);
\path (3,3) edge node [above]  {$0$} (7,3);
\tikzset{every loop/.style={min distance=23mm,in=240,out=300,looseness=50}}
\path (3,3) edge  [loop below] node {$0$} ();
\tikzset{
every node/.style={draw, circle, inner sep=2pt}
}
\foreach \i in {(7,6),(3.5,5.5),(3.5,6),(2.5,5.5),(3,3),(1.5,6),(3,7),(7,1.5),(7,3)} {
    \node[fill=black] at \i {}; 
}
\end{tikzpicture}}
\caption{Two examples of removal of vertices.}\label{fig:removing-vtx}
\end{figure}
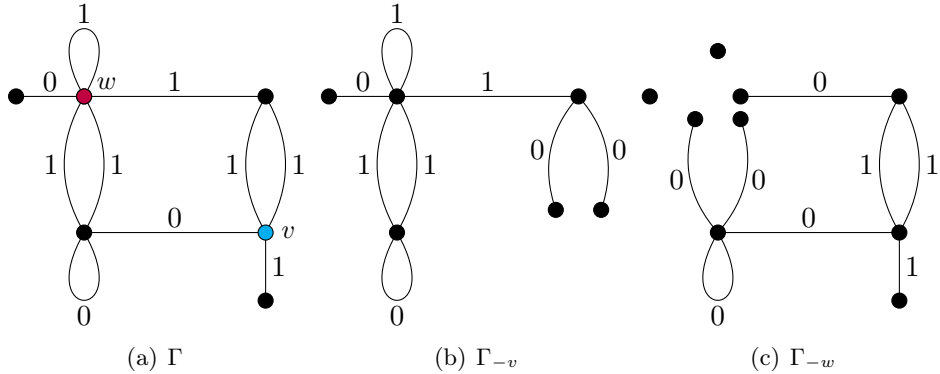
\end{example}

\begin{lemma}\label{lem:zetaGamma-v}
Let $\Gamma=(V(\Gamma),E(\Gamma),w_{\Gamma})$ be  a weighted multigraph  with binomial weights and $v\in V(\Gamma)$. Then
\begin{equation*}
\gapzvezdica(\Gamma_{-v})=\gap(\zeta(\Gamma_{-v}))=\gap(\zeta(\Gamma)(v)).
\end{equation*}
\end{lemma} 
\begin{proof}
We consider different cases, depending on the edges that are incident to $v$ in $\Gamma$.

If $v$ has a loop of weight $0$,  it is removed in  $\Gamma_{-v}$, and  $\zeta(\Gamma)(v)$ contains a component isomorphic to the path $P_2$ with $\gap(P_2) = 0$. 
If $v\in V(\Gamma)$ has a loop of weight $1$, $\zeta(\Gamma)(v)$ contains a component isomorphic to the path $P_5$. Graphs  $\Gamma_{-v}$ and $\zeta(\Gamma_{-v})$ both contain a component isomorphic to $K_1$, and $\gap(P_5) = \gap(K_1) = 1$. 

If $v$ is adjacent to a leaf with the edge of weight $0$, $\zeta(\Gamma)(v)$ contains a component isomorphic to the path $P_3$, and  $\Gamma_{-v}$ contains a component isomorphic to $K_1$, and $\gap(P_3) = \gap(K_1) = 1$. 
If $v$ is adjacent to a leaf with the edge of weight 1, $\zeta(\Gamma)(v)$ contains a component isomorphic to the path $P_4$ with $\gap(P_4) = 0$, and this edge together with the leaf is removed in $\Gamma_{-v}$. 

If there is an edge $e$ of weight 0 incident to $v$ and $u$ with $u\notin \leaves(\Gamma)$, $u\ne v$, then $\zeta(\Gamma)(v)$ contains a pendent  $P_2$ at $u$ which can be removed by Proposition~\ref{prop:from KBS24}~\eqref{prop:2.1(1)}. Also, this edge is removed from  $\zeta(\Gamma_{-v})$. Finally, if there is an edge $e$ of weight 1 incident to $v$ and $u$ with $u\notin \leaves(\Gamma)$, $u\ne v$, then $\zeta(\Gamma)(v)$ contains a pendent  $P_3$  at $u$. In  $\Gamma_{-v}$, this edge is substituted by an edge to a new leaf of weight $0$, thus  $\zeta(\Gamma_{-v})$ contains a pendent  $P_3$  on vertex $u$ as well. 

Hence,  $\gap(\zeta(\Gamma_{-v}))=\gap(\zeta(\Gamma)(v))$, which completes the proof.
\end{proof}

\subsection{Elimination of low degree vertices}

\begin{lemma}[Elimination of leaves]\label{lem:removing-leaves-on-chdv}
Suppose $\Gamma$ is a multigraph with binomial weights. Let $v$ be a leaf in $\Gamma$, $u$ its unique neighbour and $e=\{u,v\}$.
\begin{enumerate}
\item If $w_{\Gamma}(e)=1$, then $\gapzvezdica(\Gamma)=\gapzvezdica(\Gamma(v))$.
\item If $w_{\Gamma}(e)=0$, then $\gapzvezdica(\Gamma)=\gapzvezdica(\Gamma(v)_{-u})$.
 \end{enumerate}
\end{lemma}
\begin{proof}
The graph $\zeta(\Gamma)$ contains either pendent  $P_4$  (in the case $w_{\Gamma}(e)=1$) or pendent  $P_3$  (in the case $w_{\Gamma}(e)=0$) at vertex $u$. Using Proposition~\ref{prop:from KBS24}(\ref{prop:2.1(1)}) we remove a leaf twice from $\zeta(\Gamma)$. In the first case we obtain the graph $\zeta(\Gamma(v))$, so (1) follows. In the second case we obtain the graph $\zeta(\Gamma(v))(u)$ and (2) follows by Lemma~\ref{lem:zetaGamma-v}. 
\end{proof}

\begin{example}\label{ex:5.9}
Consider weighted multigraph $\Gamma = \chdv(G)$ from Figure~\ref{fig:compressed-hdv-graph} and let $u_1$ and $u_2$ be the leaves in it (see Figure \ref{fig:chdv-example-Smmax}). By Lemma~\ref{lem:removing-leaves-on-chdv}, we have  $\gapzvezdica(\Gamma)=\gapzvezdica(\Gamma(u_2))$ and so the gap does not change by elimination of the leaf $u_2$.
However, the elimination of leaf $u_1$ in $\Gamma$ results in the removal of a high degree vertex, which will result in a graph $\Gamma_1$ with new leaves  (see Figure~\ref{fig:chdv-example-Smmax1}). So, by consequently eliminating some new leaves in $\Gamma_1$ (denoted by blue in Figure~\ref{fig:chdv-example-Smmax1}), we obtain a graph with fewer vertices, and in some cases this reduces to a graph with a known $\gapzvezdica$ (see Figure~\ref{fig:chdv-example-Smmax2}).  
Thus 
$$\gapzvezdica(\Gamma)=\gapzvezdica(\Gamma_1)=\gapzvezdica(\Gamma_2)=2+\gapzvezdica(P_3) = 2+\gap(P_7) = 3.$$

\begin{figure}[h]
\centering
\subfigure[$\Gamma$]{
\begin{tikzpicture}[scale=0.6]\label{fig:chdv-example-Smmax}
\tikzset{
every node/.style={inner sep=2pt}
}
\path (3,3) edge node [above]  {$0$} (7,3)  edge [bend left] node [left] {$1$} (3,6) edge [bend right] node [right] {$1$} (3,6);
\path (7,6) edge node [above]  {$1$} (3,6)  edge [bend left] node [right] {$1$} (7,3) edge [bend right] node [left] {$1$} (7,3);
\path (7,3) edge node [right]  {$1$} (7,1.5) ;
\path (3,6) edge node [above]  {$0$} (1.5,6) ;
\tikzset{every loop/.style={min distance=23mm,in=240,out=300,looseness=50}}
\path (3,3) edge  [loop below] node {$0$} ();
\tikzset{every loop/.style={min distance=23mm,in=120,out=60,looseness=50}}
\path (3,6) edge  [loop above] node {$1$} ();
\tikzset{
every node/.style={draw, circle, inner sep=2pt}
}
\foreach \i in {(7,6),(3,3),(7,3),(3,6)} {
    \node[fill=black] at \i {}; 
}
\node[rectangle,draw=none] at (1,6) {\small $u_1$};
\node[rectangle,draw=none] at (7.5,1.5) {\small $u_2$};
 \node[fill=cyan] at  (7,1.5) {}; 
  \node[fill=cyan] at  (1.5,6) {}; 
\end{tikzpicture}}
\subfigure[$\Gamma_1$]{\begin{tikzpicture}[scale=0.6]\label{fig:chdv-example-Smmax1}
\tikzset{
every node/.style={inner sep=2pt}
}
\path (3,3)  edge [bend left] node [left] {$0$} (2.5,5.5) edge [bend right] node [right] {$0$} (3.5,5.5);
\path (7,6) edge node [above]  {$0$} (3.5,6)  edge [bend left] node [right] {$1$} (7,3) edge [bend right] node [left] {$1$} (7,3);
\path (3,3) edge node [above]  {$0$} (7,3);
\tikzset{every loop/.style={min distance=23mm,in=240,out=300,looseness=50}}
\path (3,3) edge  [loop below] node {$0$} ();
\tikzset{
every node/.style={draw, circle, inner sep=2pt}
}
\foreach \i in {(7,6),(3,3),(3,7),(3.5,5.5),(7,3)} {
    \node[fill=black] at \i {}; 
}
\foreach \i in {(3.5,6),(2.5,5.5)} {
    \node[fill=cyan] at \i {}; 
}
\phantom{  \node[fill=black] at (1.5,3) {}; 
\node[fill=black] at (8.5,3) {}; }
\end{tikzpicture}}
\subfigure[$\Gamma_2$]{\begin{tikzpicture}[scale=0.6]\label{fig:chdv-example-Smmax2}
\tikzset{
every node/.style={inner sep=2pt}
}
\phantom{ \node[fill=black] at (3,0.8) {};}
\path (7,3)  edge [bend left] node [left] {$0$} (6.5,5.5) edge [bend right] node [right] {$0$} (7.5,5.5);
\tikzset{
every node/.style={draw, circle, inner sep=2pt}
}
\foreach \i in {(4.5,5.5),(4,7),(7,3),(6.5,5.5),(7.5,5.5)} {
    \node[fill=black] at \i {}; 
}
\end{tikzpicture}}
\caption{Elimination of blue leaves in each step from Example~\ref{ex:5.9}.}\label{fig:chdv-removing-leaves}
\end{figure}
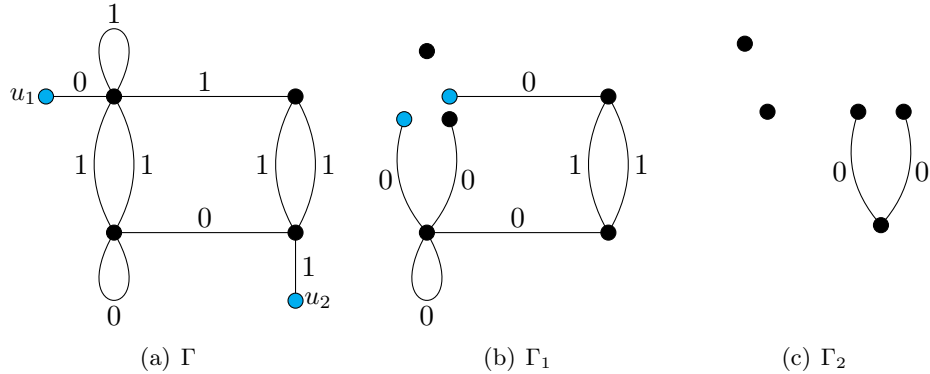
\end{example}

For any graph $G\in \GnoFourCycle$ the multigraph $\Gamma = \chdv(G)$ does not contain vertices of degree 2 by construction. But it may happen that after some operations introduced in this section, the resulting multigraph also contains vertices of degree 2, see e.g.~Figure~\ref{fig:chdv-example-Smmax2}. We can eliminate them as described below. 

\begin{lemma}[Elimination of vertices of degree 2]\label{lem:removing-ver-deg-2}
      Suppose $\Gamma$ is a multigraph with binomial weights and $u\in V(\Gamma)$ with $\deg_{\Gamma}(u)=2$. If $u$ is a vertex with one loop and incident to no other edges, let $\Gamma' = \Gamma(u)$. If not, denote the  only incident edges to $u$ by $e_1=\{v_1,u\}$ and $e_2=\{v_2,u\}$and let $\Gamma'$ be obtained from $\Gamma$ by replacing $v_1-u-v_2$ by an edge $e'$ from $v_1$ to $v_2$ of weight $w_{\Gamma'}(e')=|w_{\Gamma}(e_1)+w_{\Gamma}(e_2)-1|$. Then $$\gapzvezdica(\Gamma)=\gapzvezdica(\Gamma').$$
\end{lemma}

\begin{proof}
If $u$ is a vertex with one loop and incident to no other edges, then $\zeta(\Gamma)\cong\zeta(\Gamma') \cup C_k$ where $k=3$ or $k=6$, so $\gap(\zeta(\Gamma))=\gap(\zeta(\Gamma'))$ by Proposition \ref{prop:from KBS24}~(\ref{prop:2.1(4)}). If not, the edges $e_1$ and $e_2$ are not loops since $\deg(u) =2$. It follows that $\zeta(\Gamma)$ contains $\pathk{v_1}{v_2}{\ell}$, where $\ell = w_{\Gamma}(e_1)+w_{\Gamma}(e_2)+5$. Let $e'\in E(\Gamma')$ be as in the statement. If $e'$ is a loop with $w_{\Gamma'}(e')=1$ then $\zeta(\Gamma')$ contains $\pathk{v_1}{v_2}{5}$. In this case $w_{\Gamma}(e_1)+w_{\Gamma}(e_2)$ is even, so $\gap(\zeta(\Gamma'))=\gap(\zeta(\Gamma))$  by Theorem \ref{lem:P4toP2}. In all other cases $\zeta(\Gamma')$ contains  $\pathk{v_1}{v_2}{m}$, where $m=w_{\Gamma'}(e')+2$, so again $\gap(\zeta(\Gamma'))=\gap(\zeta(\Gamma))$  by Theorem \ref{lem:P4toP2}. The conclusion follows.
\end{proof}

\subsection{Elimination of $1$-edges}

\begin{lemma}[$1$-edge contraction]\label{lem:1-edge-contraction}
Suppose $\Gamma$ is a multigraph with binomial weights and a non-looped edge $e$ of weight $1$. Let $\Gamma'$ be the edge contraction of $\Gamma$ along $e$. Then 
$$\gapzvezdica(\Gamma)=\gapzvezdica(\Gamma').$$
\end{lemma}  

\begin{proof}
Let $\Gamma$ be a multigraph, and $\Gamma'$ its edge contraction along the edge $e= \{u,v\}$ of weight 1 with $u\ne v$. Let $G=\zeta(\Gamma)$ and $G'=\zeta(\Gamma')$. We denote by $w$ the vertex of $\Gamma'$ obtained from the edge $e$ and $u-s_1-s_2-s_3-v$ the path in $G$ obtained from $e$. We claim: $\gap(G)=\gap(G')$.

Assume first that there are $k$ $1$-edges between $u$ and $v$ in $\Gamma$, where $k\ge 2$. In this case, the path $u-s_1-s_2-s_3-v$ is part of an $8$-cycle $u-s_1-s_2-s_3-v-t_3-t_2-t_1-u$ in $G$, where $u$ and $v$ are the only vertices of this cycle that can have a degree greater than $2$ in $G$.  By Corollary~\ref{cor:remove-C2k} we have $\gap(G) = \gap(G_1)$, where 
$G_1 = G(\{u,s_1,s_2,s_3,v,t_3,t_2,t_1\})$. Note that $G_1 \cong G_0 \cup (k-2)P_3$ for some graph $G_0$. In particular, $\gap(G) = \gap(G_0) +k-2$.

On the other hand, while the edge $e$ is contracted in $\Gamma'$, the remaining $k-1$ $1$-edges between $u$ and $v$ become $1$-loops at $w$ in $\Gamma'$. Those edges correspond to $6$-cycles in $G'$ that have $w$ as the only vertex of degree greater than $2$. Let $w-t'_1-t'_2-t'_3-t'_4-t'_5-w$ be one of those cycles, and let $G'_1 = G'(\{u,t'_1,t'_2,t'_3,t'_4,t'_5\})$. By Corollary~\ref{cor:remove-C2k} we have $\gap(G') = \gap(G'_1)$. Observing that $G'_1 \cong G_0 \cup (k-2)P_5$, with the graph $G_0$ the same as above, we conclude $\gap(G') = \gap(G_0) +k-2 = \gap(G)$, as desired.

From now on, we assume that there is only one $1$-edge between $u$ and $v$ in $\Gamma$. In this case $V(G')=V(G)\setminus\{u,s_1,s_2,s_3,v\}\cup\{w\}$, hence $|V(G')|=|V(G)|-4$, and to establish our claim we need to prove $\stp(G)=\stp(G') +4$.

We will first prove that $\stp(G')\le\stp(G)-4$. Let $\C$ be an ABC cover of $G$. We consider different cases, depending on the type of vertex $s_2$ with respect to $\C$.  

\noindent The vertex $s_2$ is of {\bf Type A}. In this case $\{s_2\}\vee\{s_1,s_3\} \in \C$, hence $s_1$ and $s_3$ are of Type C, and $u$ and $v$ are of Type A. In particular, $\C$ contains the following set-joins: 
$$\{u\} \vee N^{\star}_{\C}[u], \,\{v\} \vee N^{\star}_{\C}[v],\, \{s_2\}\vee\{s_1,s_3\}.$$
Let $V_R=\{\{u\},  N^{\star}_{\C}[u],\{v\},  N^{\star}_{\C}[v], \{s_2\}, \{s_1,s_3\}\}$. We define 
$\C'$ to be the set that contains all set-joins of $\C$ that do not involve any of the components from $V_R$, as well as
$\{w\}\vee ( N^{\star}_{\C}[u]\cup N^{\star}_{\C}[v])$, and $\{w\}\vee \{y\}$ for $y \in  N^1_{\C}[u]\cup  N^1_{\C}[v]$. In particular, 
$$V(\C')=(V(\C)\setminus V_R)\cup \{\{w\}, N^{\star}_{\C}[u]\cup N^{\star}_{\C}[v] \},$$
and $|V(\C')|=|V(\C)|-4$.

\noindent The vertex $s_2$ is of {\bf Type B}. In this case $\{s_2\}\vee\{s_1\},\{s_2\}\vee\{s_3\} \in \C$, and $s_1$ and $s_3$ are of Type B. Hence, $\{u\}\vee\{s_1\},\{v\}\vee\{s_3\} \in \C$, and $u$ and $v$ are either of Type A or of Type B. Let $V_R=\{\{u\},\{v\},\{s_1\},\{s_2\},\{s_3\},N^{\star}_{\C}[u], N^{\star}_{\C}[v] \}$, where $N^{\star}_{\C}[u]$ or $N^{\star}_{\C}[v]$ may be empty. We define 
$\C'$ to be the set that contains all set-joins of $\C$ that do not involve any of the components from $V_R$, as well as
$\{w\}\vee ( N^{\star}_{\C}[u]\cup N^{\star}_{\C}[v])$ (provided $N^{\star}_{\C}[u]\cup N^{\star}_{\C}[v] \neq \emptyset$), and $\{w\}\vee \{y\}$ for $y \in  (N^1_{\C}[u]\setminus \{s_1\})\cup  (N^1_{\C}[v]\setminus \{s_3\})$. 
In particular, 
$$V(\C')=(V(\C)\setminus V_R)\cup \{\{w\}, N^{\star}_{\C}[u]\cup N^{\star}_{\C}[v] \},$$
and $|V(\C')|\leq|V(\C)|-4$. 

\noindent The vertex $s_2$ is of {\bf Type C}. In this case $\{s_1\}\vee\{u,s_2\},\{s_3\}\vee\{s_2,v\} \in \C$, and $u$, $v$ and $s_2$ are of Type C. Let 
$$
V_R=\{\{s_1\},\{s_3\},\{u,s_2\},\{s_2,v\}\}
\cup \{N^{\star}_{\C}[s], s \in N_G[u]\setminus\{s_1\}\}
\cup \{N^{\star}_{\C}[t], t \in N_G[v]\setminus\{s_3\}\}.
$$
Let $\C'$ be the set that contains all set-joins in $\C$ that do not involve any of the components in $V_R$, as well as
\begin{align*}   
 &\{s\}\vee ((N^{\star}_{\C}[s]\setminus \{u\})\cup \{w\}),\, s \in N_G[u]\setminus\{s_1\}, \text{ and }\\
&\{t\}\vee ((N^{\star}_{\C}[t]\setminus \{v\})\cup \{w\}), t \in N_G[v]\setminus\{s_3\}.
\end{align*}
In particular, $|V(\C')|=|V(\C)|-4$, as desired. 

Now we will prove that $\stp(G)\le\stp(G') +4$. Let $\C'$ be an ABC cover of $G'$. We consider different cases, depending on the type of vertex $w$ with respect to $\C'$. We will construct a set-join cover $\C$ of $G$ with $|V(\C)| \le |V(\C')| +4$. Let
$$ \widehat{\C'} = \{\mc K \vee \mc L \in \C', w\notin \mc K \cup \mc L\}.$$
Note that $\C' = \widehat{\C'} \cup \C'[w]$. 

\noindent The vertex $w$ is of {\bf Type A}. In this case, let 
$$ \C = \widehat{\C'} \cup \{\{u\} \vee N_G[u], \{v\} \vee N_G[v], \{s_2\} \vee \{s_1, s_3\}\}.$$
Since any component of cardinality at least 2 appears only once in $\C'$ by Proposition \ref{prop:gen4cycle}~\eqref{lem:2a}, we have $|V(\widehat{\C'})| \le |V(\C')| -2$ ($\{w\}$ and $N^{\star}_{\C'}[w]$ do not appear in $\widehat{\C'}$). Since $\C'[w]$ cover only edges adjacent to $w$ in $G'$, $\widehat{\C'}$ covers all edges in $G$ that are not adjacent to $u,s_1,s_2,s_3$, or $v$. Furthermore, $\{\{u\} \vee N_G[u], \{v\} \vee N_G[v], \{s_2\} \vee \{s_1, s_3\}\}$ covers all edges in $G$ that are adjacent to $u,s_1,s_2,s_3$, or $v$. So $\C$ is a 
a set-join cover of $G$ with 
\begin{align*}
|V(\C)| &\le |V(\widehat{\C'})| + |V(\{\{u\} \vee N_G[u], \{v\} \vee N_G[v], \{s_2\} \vee \{s_1, s_3\}\})| \\
&\le  |V(\C')| - 2 + 6 = |V(\C')| + 4.    
\end{align*}

\noindent The vertex $w$ is of {\bf Type B}. In this case $\C'[w] = \{\{w\} \vee \{x\}, x \in N_{G'}[w]\}$. Note that any $x \in N_{G'}[w]$ has $\deg_{G'}(x) = 2$, so it is of Type B and thus $\{x\} \in \widehat{\C'}$. Let
$$
\C = \widehat{\C'} \cup \{\{u\} \vee \{x\}, x \in N_{G}[u]\} \cup \{\{v\} \vee \{x\}, x \in N_{G}[v]\} 
   \cup \{\{s_1\} \vee \{s_2\}, \{s_2\} \vee \{s_3\}\}.
$$
We have $|V(\widehat{\C'})| = |V(\C')| -1$ ($\{w\}$ does not appear in $\widehat{\C'}$). As in the previous case, $\widehat{\C'}$ covers all edges in $G$ that are not adjacent to $u,s_1,s_2,s_3$, or $v$ and $\{\{u\} \vee \{x\}, x \in N_{G}[u]\} \cup \{\{v\} \vee \{x\}, x \in N_{G}[v]\} \cup \{\{s_1\} \vee \{s_2\}, \{s_2\} \vee \{s_3\}\}$ covers all edges in $G$ that are adjacent to $u,s_1,s_2,s_3$, or $v$. Since $N_{G}[u] \cup N_{G}[v] = N_{G'}[w]\cup\{s_1, s_3\}$ we have
$$
|V(\C)| \le |V(\widehat{\C'})| + |\{\{u\},\{s_1\},\{s_2\},\{s_3\},\{v\}\}| 
=  |V(\C')| - 1 + 5 = |V(\C')| + 4.    
$$
The vertex $w$ is of {\bf Type C}. In this case $\C'[w] = \{\{x\} \vee N^{\star}_{\C'}[x], x \in N_{G'}[w]\}$. 
Let 
\begin{align*}
\C_2 &= \{\{x\} \vee N^{\star}_{\C'}[x]\cup\{u\}\setminus\{w\}, x \in N_G[u]\},  \\   
\C_3 &= \{\{x\} \vee N^{\star}_{\C'}[x]\cup\{v\}\setminus\{w\}, x \in N_G[v]\}, \text{ and} \\
\C &= \widehat{\C'} \cup \C_2 \cup \C_3 \cup \{\{s_1\} \vee \{u, s_2\}, \{s_3\} \vee \{v, s_2\}\}.
\end{align*}
Since any component of cardinality at least 2 appears only once in $\C'$, we have $|V(\widehat{\C'} \cup \C_2 \cup \C_3)| = |V(\C')|$. Furthermore, $\widehat{\C'} \cup \C_2 \cup \C_3$ covers all edges in $G$ except those on the path $u-s_1-s_2-s_3-v$ and $\{\{s_1\} \vee \{u, s_2\}, \{s_3\} \vee \{v, s_2\}\}$ covers all edges on the path $u-s_1-s_2-s_3-v$. 
So $\C$ is a set-join cover of $G$ with $|V(\C)| = |V(\C')| +4$, which finishes the proof.
\end{proof}

\begin{lemma}[Elimination of $1$-loops]\label{lem:removing-1-loops}
Suppose $\Gamma$ is a multigraph with binomial weights and suppose there exists $v\in V(\Gamma)$ with a loop $e$ of weight $w_{\Gamma}(e)=1$. Then
$$\gapzvezdica(\Gamma)=\gapzvezdica((\Gamma\setminus\{e\})_{-v})=\gapzvezdica(\Gamma_{-v})-1.$$
\end{lemma}
\begin{proof}
Observe that $\zeta(\Gamma)$ contains an induced subgraph $G_1$ isomorphic to $C_6$ with one of the vertices $v$ and all other vertices of degree 2. By Corollary~\ref{cor:remove-C2k} we have $\gap(G) = \gap(G_2)$, where $G_2 = G(V(G_1))$. Furthermore, $\gap(G_2) = \gap(\zeta((\Gamma\setminus\{e\})_{-v}))$ by Lemma~\ref{lem:zetaGamma-v}, and the first equality follows. Note that $\Gamma_{-v} \cong (\Gamma\setminus\{e\})_{-v} \cup K_1$, so we obtain also the second equality.
\end{proof}

\begin{example}\label{ex:1-edge-contraction} 
Reconsider weighted multigraph $\Gamma = \chdv(G)$ from Figure~\ref{fig:compressed-hdv-graph} (see also Figure~\ref{fig:chdv-example}). After a sequence of $1$-edge contractions of blue edges, we obtain multigraph graph $\Gamma'$ in Figure~\ref{fig:1-contracted chdv-example}, and by Lemma~\ref{lem:1-edge-contraction} we have $\gapzvezdica(\Gamma)=\gapzvezdica(\Gamma')$. Now the high degree vertex $v$ of $\Gamma' $ has 3 loops of weight $1$, 2 loops of weight $0$, and a pendent vertex. Let $e$ be one of the loops of weight $1$. Using Lemma~\ref{lem:removing-1-loops} we have $\gapzvezdica(\Gamma)=\gapzvezdica(\Gamma'\setminus\{e\})_{-v}$, where the later graph consists of three copies of $K_1$, as shown in~Figure~\ref{fig:1-contracted chdv-example1}. Thus, $$\gapzvezdica(\Gamma)=\gapzvezdica(\Gamma')=\gapzvezdica((\Gamma'\setminus\{e\})_{-v})=\gapzvezdica(3 K_1)=\gap(3 K_1)=3.$$

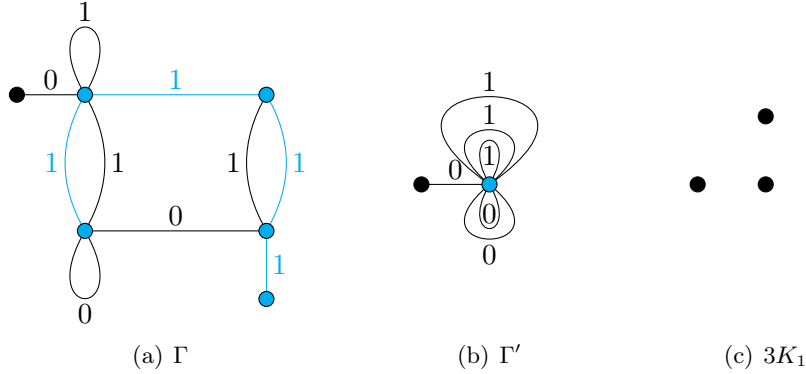
\begin{figure}[h]
\centering
\subfigure[$\Gamma$]{\begin{tikzpicture}[scale=0.6]\label{fig:chdv-example}
\tikzset{
every node/.style={inner sep=2pt}
}
\path (3,3) edge node [above]  {$0$} (7,3)  edge [bend left,color=cyan] node [left] {$1$} (3,6) edge [bend right] node [right] {$1$} (3,6);
\path (7,6) edge [color=cyan] node [above]  {$1$} (3,6)  edge [bend left,color=cyan] node [right] {$1$} (7,3) edge [bend right] node [left] {$1$} (7,3);
\path (7,3) edge [color=cyan] node [right]  {$1$} (7,1.5) ; 
\path (3,6) edge node [above]  {$0$} (1.5,6) ;
\tikzset{every loop/.style={min distance=23mm,in=120,out=60,looseness=50}}
\path (3,6) edge  [loop above] node {$1$} ();
\tikzset{every loop/.style={min distance=23mm,in=240,out=300,looseness=50}}
\path (3,3) edge  [loop below] node {$0$} ();
\tikzset{
every node/.style={draw, circle, inner sep=2pt}
}
\foreach \i in {(7,6),(3,6),(3,3),(7,3),(7,1.5)} {
    \node[fill=cyan] at \i {}; }
    \foreach \i in {(1.5,6)} {
    \node[fill=black] at \i {}; }
\end{tikzpicture}}
\subfigure[$\Gamma'$]{\begin{tikzpicture}[scale=0.6]\label{fig:1-contracted chdv-example}
\tikzset{
every node/.style={inner sep=2pt}
}
\path (-1.5,0) edge node [above]  {$0$} (0,0) ;
\tikzset{every loop/.style={min distance=45mm,in=35,out=145,looseness=50}}
\path (0,0) edge  [loop above] node {$1$} ();
\tikzset{every loop/.style={min distance=25mm,in=40,out=140,looseness=50}}
\path (0,0) edge  [loop above] node {$1$} ();
\tikzset{every loop/.style={min distance=15mm,in=60,out=120,looseness=50}}
\path (0,0) edge  [loop above] node [below] {$1$} ();
\tikzset{every loop/.style={min distance=25mm,in=-40,out=-140,looseness=50}}
\path (0,0) edge  [loop below] node [below] {$0$} ();
\tikzset{every loop/.style={min distance=15mm,in=-60,out=-120,looseness=50}}
\path (0,0) edge  [loop below] node [above] {$0$} ();
\tikzset{
every node/.style={draw, circle, inner sep=2pt}
}
 \node[fill=cyan] at (0,0) {}; 
 \node[fill=black] at (-1.5,0) {};
\phantom{  \node[fill=cyan] at (0,-3) {}; 
\node[fill=black] at (2,-3) {}; 
\node[fill=black] at (-2,-3) {}; }
\end{tikzpicture}}
\subfigure[$3 K_1$]{\begin{tikzpicture}[scale=0.6]\label{fig:1-contracted chdv-example1}
 \node[fill=black] at (0,2) {}; 
\node[fill=black] at (0,3.5) {}; 
\node[fill=black] at (-1.5, 2) {}; 
\phantom{  \node[fill=black] at (0,-1) {}; 
\node[fill=black] at (2,-1) {}; 
\node[fill=black] at (-2,-1) {}; 
}
\end{tikzpicture}}
\caption{Graph $\Gamma'$ is obtained from $\Gamma$ by a sequence of four edge contractions (denoted by blue), and  $(\Gamma'\setminus\{e\})_{-v}=3 K_1$ is obtained by further elimination of a $1$-loop $e$ of $\Gamma'$.}\label{fig:complete-1-edge-contraction-Smmax}
\end{figure}

Observe that we used a different method than in Example~\ref{ex:5.9} to obtain the same result.
\end{example}

\subsection{Elimination of parallel $0$-edges}

\begin{lemma}[Elimination of  parallel $0$-edges] \label{lem:parallel-0-edges} 
Let $\Gamma$ be a multigraph with binomial weights containing $k\ge 2$ $0$-edges between vertices $u$ and $v$ (where we allow $u=v$), and let $\Gamma'$  be obtained from $\Gamma$ by removing $k-1$ $0$-edges between those two vertices. Then
$$\gapzvezdica(\Gamma)=\gapzvezdica(\Gamma').$$
\end{lemma}

\begin{proof}
 Let $\Gamma_1$ be the multi-graph obtained from $\Gamma$ by removing one edge between $u$ and $v$. It is sufficient to prove $\gap(\zeta(\Gamma)) = \gap(\zeta(\Gamma_1))$. Indeed, if there are more than two edges $u$ between $v$, we can repeatedly use this result to reduce $\Gamma$ to $\Gamma'$. 

 Let $G=\zeta(\Gamma)$, and note that $\zeta(\Gamma_1)=G(\{y_1,y_2\})$, where $y_1,y_2$ arise from the edge which is removed from $\Gamma$ to obtain $\Gamma_1$. The graph $G$ contains $u-x_1-x_2-v-y_2-y_1-u$, where $x_1,x_2,y_1,y_2 \in V_{d=2}(G)\setminus\{u,v\}$ are distinct vertices. Note that we allow $u=v$ here.
 
To prove $\stp(G)\leq \stp(G(\{y_1,y_2\}))+2$, assume that $\C'$ is an $\mathrm{ABC}$ cover of $G(\{y_1,y_2\})$. Using Lemma~\ref{lem:remove-x1-x2} for $H=G(\{y_1,y_2\})$ and symmetry, we need to consider two cases.

\noindent{\bf Case 1:} With respect o $\C'$, $x_1$ is of Type A and $x_2$ is of Type C. Then $u$ is of Type C and $v$ of Type A. In this case we have $\{x_1\} \vee \{u,x_2\}, \{v\} \vee N^{\star}_{\C'}[v] \in \C'$ and $x_2 \in N^{\star}_{\C'}[v]$. To construct a cover $\C$ of $G$ from $\C'$ we replace  $\{v\} \vee N^{\star}_{\C'}[v]$ by $\{v\} \vee (N^{\star}_{\C'}[v]\cup\{y_2\})$ and add $\{y_1\}\vee \{u,y_2\}$. Since $|V(\C)|=|V(\C')|+2$ the inequality holds.

\noindent{\bf Case 2:} Both $x_1$ and $x_2$ are of Type B in $\C'$. In this case $u,v,x_1,x_2 \in V_1(\C')$. A cover  $\C$ of $G$ is obtained by adding the set-joins $\{u\}\vee \{y_1\}$,$\{y_1\}\vee \{y_2\}$,$\{y_2\}\vee \{v\}$ to $\C'$. Again we have $|V(\C)|=|V(\C')|+2$.

To prove $\stp(G)\geq \stp(G(\{y_1,y_2\}))+2$, assume that $\C$ is an $\mathrm{ABC}$ cover of $G$. Using Lemma~\ref{lem:remove-x1-x2} for $H=G$ and symmetry, we need to consider two cases.

\noindent{\bf Case 1:} With respect o $\C$, $x_1$ is of Type A and $x_2$ is of Type C. Since $u$ is of Type C, $y_1$ must be of Type A, and  $y_2$ must be of Type C. In particular, $\C$ contains $\{y_1\}\vee \{y_2,u\}$. Removing this set-join from $\C$ reduces the number of components by $2$. A set join cover $\C'$ of $G(\{y_1,y_2\})$ is obtained by also removing $y_2$ from $N_G^{\star}[v]$.

\noindent{\bf Case 2:} Both $x_1$ and $x_2$ are of Type B in $\C$. In this case $u,v,x_1,x_2 \in V_1(\C)$.  This implies that $y_1$ and $y_2$ are of Type B, hence $\{y_i\}$, $i=1,2$, are components in $\C$. Removing these two components, and the three set-joins that use them, gives us a set-join cover $\C'$ of  $G(\{y_1,y_2\})$ with $|V(\C')|=|V(\C)|-2$. 
\end{proof}

\begin{example}\label{ex:garlic}
Let $G=\garlic(k_1,\ldots,k_t)$ be the graph on $|V(G)|=2+\sum_{i\in[t]} k_i$ vertices, $t\geq 2$ and $k_i\geq 2$, obtained as a vertex identification of all leaves of a generalised star graph $\gstar(k_1+1,\ldots,k_t+1)$ into a single vertex. We claim that 
$$\gap(G)= \max\{|\{i\colon k_i \text{ odd}\}|-2, 0\}.$$
Let $o=|\{i\colon k_i \text{ odd}\}|$ and $e=t-o$. The multigraph $\Gamma=\kappa(G)$ contains two vertices $u$ and $v$, and $o$ edges of weight 1 and $e$ edges of weight 0 between $u$ and $v$. By Lemma \ref{lem:parallel-0-edges} we may assume that $e$ is equal to 0 or 1. If $o=0$, then $e=1$ and $\zeta(\Gamma) = P_4$, so $\gap(G)=0$. If $o>0$, then using $1$-edge contraction we obtain a multigraph $\Gamma'$ with only one vertex $v$ and and $o-1$ loops of weight 1 and $e$ loops of weight 0 on it, which by Lemma~\ref{lem:1-edge-contraction} has $\gapzvezdica(\Gamma)=\gapzvezdica(\Gamma')$. If $o=1$, then $\gapzvezdica(\Gamma')=0$, so $\gap(G)=0$ again. If $o>1$, then the graph $\Gamma'_{-v}$ consists of $o-1$ vertices of degree $0$, so by Lemma~\ref{lem:removing-1-loops} $\gapzvezdica(\Gamma')=o-2$, and the claim follows. 
\end{example}

\section{Reduction of a multigraph $\Gamma$ to $\tau(\Gamma)$}\label{sec:tau}

We have found several operations that reduce the computation of $\gapzvezdica(\Gamma)$ to a simplified multigraph. In Example~\ref{ex:huge-gamma-chapter-4} we give an example of a reduction of a multigraph $\Gamma$ to a smaller graph with the same $\gapzvezdica$. While those operations can be performed in any order, we will later present an algorithm that makes this simplification unique and efficient. 

\begin{example}\label{ex:huge-gamma-chapter-4}
    Consider the multigraph with binomial weights $\Gamma$ in Figure~\ref{fig:huge-gamma}. Instead of writing edge weights, we color all edges of weight $0$ black and all edges of weight $1$ by either blue, orange, green or pink. 
    
    To obtain $\gapzvezdica(\Gamma)$ we follow reductions described in Section~\ref{sec:operations-preserve-gap*}. First, we consecutively contract five blue $1$-edges, four orange $1$-edges and three green $1$-edges. After all $12$ contractions we obtain graph $\Gamma'$, see Figure~\ref{fig:huge-gamma'}.
    Next we eliminate one orange, one green, and one pink loop of weight $1$ and the resulting graph is $\Gamma''$, see Figure~\ref{fig:huge-gamma''}. After elimination of all parallel edges of weight $0$ and all parallel loops of weight $0$, the resulting graph $\Gamma'''$ is shown in Figure~\ref{fig:huge-gamma'''}. Observe that $\Gamma'''$ is uniquely obtained from $\Gamma$ by the described process and that it does not contain $1$-edges, $1$-loops or parallel $0$-edges. Let us denote by $\Psi$ the connected component of $\Gamma'''$ on four vertices. Using Lemmas~\ref{lem:1-edge-contraction}, \ref{lem:removing-1-loops} and \ref{lem:parallel-0-edges} it follows that 
    $$
\gapzvezdica(\Gamma)=\gapzvezdica(\Gamma')=\gapzvezdica(\Gamma'')=\gapzvezdica(\Gamma''')
=\gapzvezdica(\Psi)+4\gapzvezdica(K_1)=\gapzvezdica(\Psi)+4.
$$

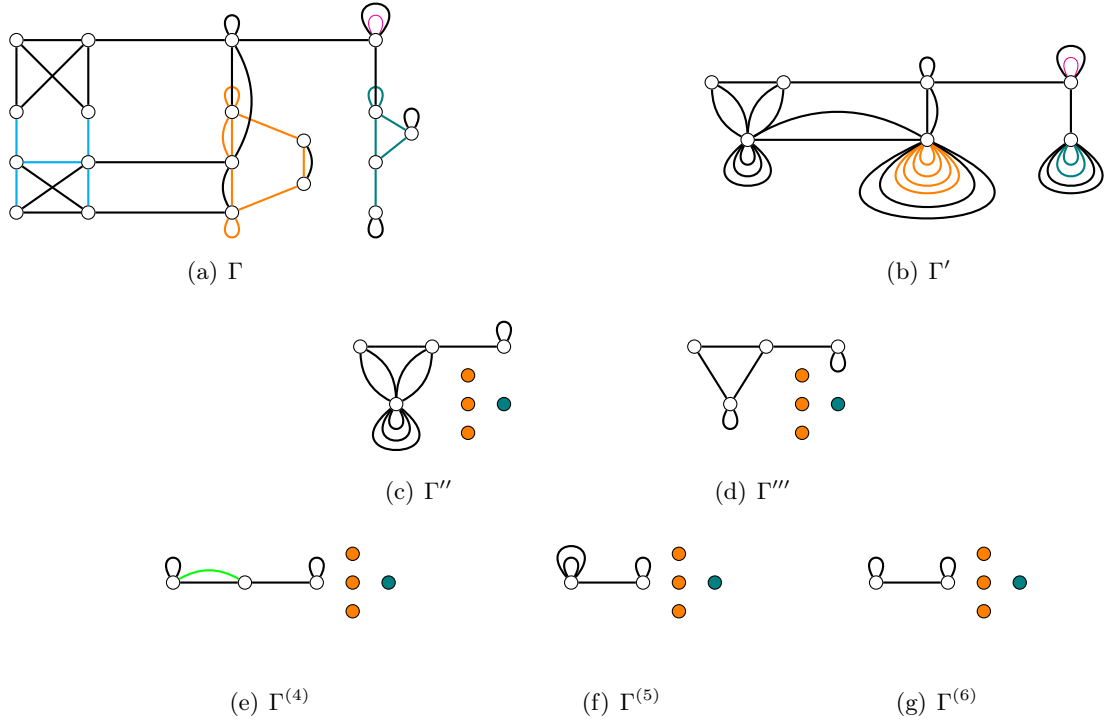
\begin{figure}[htb]
\centering
  \subfigure[$\Gamma$]{ \label{fig:huge-gamma}
    \begin{tikzpicture}[scale=0.95,
  vertex/.style={circle,draw,fill=white,inner sep=1.2pt,minimum size=5pt},
  cyanedge/.style={cyan,thick},
  orangeedge/.style={orange,thick},
  tealedge/.style={teal,thick},
  blackedge/.style={black,thick}
  ]
 \hugeGamma
\end{tikzpicture}}
\hfill \subfigure[$\Gamma'$]{  \label{fig:huge-gamma'}
    \begin{tikzpicture}[scale=0.95,
  vertex/.style={circle,draw,fill=white,inner sep=1.2pt,minimum size=5pt},
  cyanedge/.style={cyan,thick},
  orangeedge/.style={orange,thick},
  tealedge/.style={teal,thick},
  blackedge/.style={black,thick}
  ]

  \node[vertex] (H) at (0.5,1.2){};
  \node[vertex] (D) at (3,1.2) {};
  \node[vertex] (P) at (5,1.2) {};
  
  \node[vertex] (B0) at (0,2.0) {};
  \node[vertex] (B1) at (1,2.0) {};
  \node[vertex] (B2) at (3,2.0) {};
  \node[vertex] (B3) at (5,2.0) {};
  \tikzset{every loop/.style={min distance=4mm,in=120,out=60,looseness=10}}
\draw[color=magenta] (B3) to[loop above] (B3);
\draw[blackedge] (B2) to[loop above] (B2);
\tikzset{every loop/.style={min distance=4mm,in=130,out=50,looseness=15}}
\draw[blackedge] (B3) to[loop above] (B3);
 \tikzset{every loop/.style={min distance=4mm,in=240,out=300,looseness=10}}
\draw[tealedge] (P) to[loop below] (P);
\draw[orangeedge] (D) to[loop below] (D);
\draw[blackedge] (H) to [loop below] (H);
 \tikzset{every loop/.style={min distance=4mm,in=230,out=310,looseness=15}}
\draw[tealedge] (P) to[loop below] (P);
\draw[orangeedge] (D) to[loop below] (D);
\draw[blackedge] (H) to [loop below] (H);
\tikzset{every loop/.style={min distance=4mm,in=220,out=320,looseness=20}}
\draw[orangeedge] (D) to[loop below] (D);
\draw[blackedge] (H) to [loop below] (H);
\draw[blackedge] (P) to[loop below] (P);
\tikzset{every loop/.style={min distance=4mm,in=215,out=325,looseness=25}}
\draw[orangeedge] (D) to[loop below] (D);
\draw[blackedge] (P) to[loop below] (P);
\tikzset{every loop/.style={min distance=4mm,in=210,out=330,looseness=35}}
\draw[blackedge] (D) to[loop below] (D);
\tikzset{every loop/.style={min distance=4mm,in=205,out=335,looseness=47}}
\draw[blackedge] (D) to[loop below] (D);

  \draw[blackedge] (B0) -- (B1) -- (B2) -- (B3);

  \draw[blackedge] (H) -- (D) -- (B2);
  \draw[blackedge] (P) -- (B3);

  \draw[blackedge] (B2) edge [bend left,color=black,thick] (D);
  \draw[blackedge] (B0) edge [bend left,color=black,thick] (H);
   \draw[blackedge] (B0) edge [bend right,color=black,thick] (H);
   \draw[blackedge] (B1) edge [bend left,color=black,thick] (H);
   \draw[blackedge] (B1) edge [bend right,color=black,thick] (H);
  \draw[blackedge] (H) edge [bend left,color=black,thick] (D);
\end{tikzpicture}}

 \subfigure[$\Gamma''$]{  \label{fig:huge-gamma''}
    \begin{tikzpicture}[scale=0.95,
  vertex/.style={circle,draw,fill=white,inner sep=1.2pt,minimum size=5pt},
  cyanedge/.style={cyan,thick},
  blackedge/.style={black,thick}
  ]

  \node[vertex] (H) at (0.5,1.2){};
  
  \node[vertex] (B0) at (0,2.0) {};
  \node[vertex] (B1) at (1,2.0) {};
  \node[vertex] (B2) at (2,2.0) {};

  \node[vertex, fill=orange] (C1) at (1.5,1.6) {};
  \node[vertex, fill=orange] (C2) at (1.5,1.2) {};
  \node[vertex, fill=orange] (C3) at (1.5,0.8) {};
  \node[vertex, fill=teal] (c4) at (2,1.2) {};

 \tikzset{every loop/.style={min distance=4mm,in=120,out=60,looseness=10}}
\draw[blackedge] (B2) to[loop above] (B2);
 \tikzset{every loop/.style={min distance=4mm,in=240,out=300,looseness=10}}
\draw[blackedge] (H) to [loop below] (H);
 \tikzset{every loop/.style={min distance=4mm,in=230,out=310,looseness=15}}
\draw[blackedge] (H) to [loop below] (H);
\tikzset{every loop/.style={min distance=4mm,in=220,out=320,looseness=20}}
\draw[blackedge] (H) to [loop below] (H);

  \draw[blackedge] (B0) -- (B1) -- (B2);

  \draw[blackedge] (B0) edge [bend left,color=black,thick] (H);
   \draw[blackedge] (B0) edge [bend right,color=black,thick] (H);
   \draw[blackedge] (B1) edge [bend left,color=black,thick] (H);
   \draw[blackedge] (B1) edge [bend right,color=black,thick] (H);
   
\end{tikzpicture}}
\qquad\qquad
\subfigure[$\Gamma'''$]{  \label{fig:huge-gamma'''}
    \begin{tikzpicture}[scale=0.95,
  vertex/.style={circle,draw,fill=white,inner sep=1.2pt,minimum size=5pt},
  cyanedge/.style={cyan,thick},
  blackedge/.style={black,thick}
  ]

 \tauGammaNoLabel
 
  \node[vertex, fill=orange] (C1) at (1.5,1.6) {};
  \node[vertex, fill=orange] (C2) at (1.5,1.2) {};
  \node[vertex, fill=orange] (C3) at (1.5,0.8) {};
  \node[vertex, fill=teal] (c4) at (2,1.2) {};
 \phantom{\tikzset{every loop/.style={min distance=4mm,in=220,out=320,looseness=20}}
\draw[blackedge] (H) to [loop below] (H);}
\end{tikzpicture}}

\subfigure[$\Gamma^{(4)}$]{  \label{fig:huge-gamma-4}
    \begin{tikzpicture}[scale=0.95,
  vertex/.style={circle,draw,fill=white,inner sep=1.2pt,minimum size=5pt},
  cyanedge/.style={cyan,thick},
  blackedge/.style={black,thick}
  ]

  \node[vertex] (H) at (-1,1.2){};
  
  \node[vertex] (B0) at (0,1.2) {};
 \phantom{ \node[vertex] (B1) at (1,0) {};}
  \node[vertex] (B2) at (1,1.2) {};
  \node[vertex, fill=orange] (C1) at (1.5,1.6) {};
  \node[vertex, fill=orange] (C2) at (1.5,1.2) {};
  \node[vertex, fill=orange] (C3) at (1.5,0.8) {};
  \node[vertex, fill=teal] (C4) at (2,1.2) {};

 \tikzset{every loop/.style={min distance=4mm,in=120,out=60,looseness=10}}
\draw[blackedge] (B2) to[loop above] (B2);
 \draw[blackedge] (H) to [loop above] (H);
 
 \draw (B0) edge [bend right,color=green,thick] (H);
 
  \draw[blackedge] (H) -- (B0) --  (B2);

\end{tikzpicture}}
\qquad\qquad
\subfigure[$\Gamma^{(5)}$]{  \label{fig:huge-gamma-5}
    \begin{tikzpicture}[scale=0.95,
  vertex/.style={circle,draw,fill=white,inner sep=1.2pt,minimum size=5pt},
  cyanedge/.style={cyan,thick},
  blackedge/.style={black,thick}
  ]

  \node[vertex] (H) at (0,1.2){};
  
 \phantom{ \node[vertex] (B1) at (1,0) {};}
  \node[vertex] (B2) at (1,1.2) {};
  \node[vertex, fill=orange] (C1) at (1.5,1.6) {};
  \node[vertex, fill=orange] (C2) at (1.5,1.2) {};
  \node[vertex, fill=orange] (C3) at (1.5,0.8) {};
  \node[vertex, fill=teal] (C4) at (2,1.2) {};

 \tikzset{every loop/.style={min distance=4mm,in=120,out=60,looseness=10}}
\draw[blackedge] (B2) to[loop above] (B2);
 \draw[blackedge] (H) to [loop above] (H);
  \tikzset{every loop/.style={min distance=4mm,in=130,out=50,looseness=15}}
  \draw[blackedge] (H) to [loop above] (H);
 
  \draw[blackedge] (H) --  (B2);

\end{tikzpicture}}
\qquad\qquad
\subfigure[$\Gamma^{(6)}$]{  \label{fig:huge-gamma-6}
    \begin{tikzpicture}[scale=0.95,
  vertex/.style={circle,draw,fill=white,inner sep=1.2pt,minimum size=5pt},
  cyanedge/.style={cyan,thick},
  blackedge/.style={black,thick}
  ]

   \node[vertex] (H) at (0,1.2){};
  
 \phantom{ \node[vertex] (B1) at (1,0) {};}
  \node[vertex] (B2) at (1,1.2) {};
  \node[vertex, fill=orange] (C1) at (1.5,1.6) {};
  \node[vertex, fill=orange] (C2) at (1.5,1.2) {};
  \node[vertex, fill=orange] (C3) at (1.5,0.8) {};
  \node[vertex, fill=teal] (C4) at (2,1.2) {};
  
 \tikzset{every loop/.style={min distance=4mm,in=120,out=60,looseness=10}}
\draw[blackedge] (B2) to[loop above] (B2);
 \draw[blackedge] (H) to [loop above] (H);

  \draw[blackedge] (H) --  (B2);

\end{tikzpicture}}
    \caption{A weighted multigraph $\Gamma$ and its reduction to smaller graphs, where all the blue, orange, green and pink edges and loops have weight $1$ and all black edges and loops have weight $0$. All seven graphs have the same $\gapzvezdica$.}
    \label{fig:huge-multigraph}
\end{figure}

To compute $\gapzvezdica(\Gamma''')$ we can eliminate its vertex of degree $2$.
With this operation we obtain graph $\Gamma^{(4)}$ (Figure~\ref{fig:huge-gamma-4}), which contains an edge of weight $1$. Now it is possible to continue the same reduction as described above. After $1$-edge contraction in  $\Gamma^{(4)}$ and  elimination of a parallel $0$-loop on $\Gamma^{(5)}$ we obtain  graph $\Gamma^{(6)}=P_2^{\circ,\circ}\cup 4K_1$ and by Lemmas~\ref{lem:removing-ver-deg-2}, \ref{lem:1-edge-contraction} and \ref{lem:parallel-0-edges} we get
$$ \gapzvezdica(\Gamma)=\gapzvezdica(\Psi)+4=\gapzvezdica(P_2^{\circ,\circ})+4. $$
\end{example}
In the next subsection we will find a shorter way to reduce $\Gamma$ to $\Psi$ 
and compute $\gapzvezdica(\Gamma)$.

\subsection{Operation $\tau_1$}

    Let $\Gamma=(V(\Gamma),E(\Gamma),w_{\Gamma})$ be a connected weighted multigraph with binomial weights. For every $i\in \{0,1\}$, let us denote by $\varepsilon_i(\Gamma)$ the number of edges (including loops) in $E(\Gamma)$ of weight $i$.

    \begin{definition}\label{def:1-subgraph}
    We say that an induced subgraph $\Gamma'$ of $\Gamma$ is a \emph{$1$-component} of $\Gamma$ if it contains a spanning tree consisting only of edges of weight $1$, and any edges incident to $u\in V(\Gamma')$ and $v\in V(\Gamma)\setminus V(\Gamma')$ have weight $0$.

    In particular, it can happen that a $1$-component of $\Gamma$ has only one vertex $u$ (with possible $0$-loops or $1$-loops): $u$ is the only vertex in a $1$-component if and only if $u$ is a vertex of degree $0$ or is incident (in $\Gamma$) only to loops and edges $e=\{u,v\}$, $u\ne v$, with $w_{\Gamma}(e)=0$.
    \end{definition}
    
    Observe that every vertex $v \in V(\Gamma)$ is a vertex in precisely one $1$-component of $\Gamma$ and that for a $1$-component $\Gamma'$ of $\Gamma$ we have $\varepsilon_1(\Gamma') \geq |V(\Gamma')|-1$. 
    
\begin{definition}\label{def:tau1}
     Let  $\Gamma=(V(\Gamma),E(\Gamma),w_{\Gamma})$ be a connected weighted multigraph with binomial weights and $1$-components
     $\Gamma_1,\ldots,\Gamma_k$. 
     Assume the components are ordered so that for some $\ell \in \{0,1,\ldots,k\}$:
     \begin{itemize}
   \item $\varepsilon_1(\Gamma_i)=|V(\Gamma_j)|-1$  for $i\in [\ell]$, and \item $\varepsilon_1(\Gamma_i) \geq |V(\Gamma_i)|$ for $j\in \{\ell+1,\ldots,k\}$.
     \end{itemize}

     We denote  $$t_1(\Gamma):=\sum_{i=\ell+1}^{k}(\varepsilon_1(\Gamma_i) - |V(\Gamma_i)|)$$  and define a weighted graph $\tau_1(\Gamma)$ by
     \begin{itemize}
         \item $V(\tau_1(\Gamma))=\{w_1,\ldots,w_{\ell}\}$;
         \item for $i,j\in [\ell]$ there is a unique edge $\{w_i,w_j\}\in E(\tau_1(\Gamma))$ if and only if there exists $e=\{u,v\}\in E(\Gamma)$ of weight $0$ with $u\in \Gamma_i$ and $v\in \Gamma_j$, where  we allow $i=j$;
         \item all edges and loops in $\tau_1(\Gamma)$ are of weight $0$.
     \end{itemize}
\end{definition}

If a multigraph $\Gamma$ does not contain any $1$-edges, parallel $0$-edges or loops then $\tau_1(\Gamma) = \Gamma$ and $t_1(\Gamma) = 0$.

\begin{example}\label{ex:tau1-reduction-huge-Gamma}
     For multigraph $\Gamma$ in Figure~\ref{fig:huge-gamma} we have $k=7$ and $\ell=4$. More precisely, let  $\Gamma_1$ be the $1$-component of $\Gamma$ on six vertices that has a blue spanning tree ($V(\Gamma_1)=6$ and $\varepsilon_1(\Gamma_1)=5$), $\Gamma_2$ and $\Gamma_3$ single vertices without loops  ($V(\Gamma_2)=V(\Gamma_3)=1$ and $\varepsilon_1(\Gamma_2)=\varepsilon_1(\Gamma_3)=0$), $\Gamma_4$ is a vertex with a black loop  ($V(\Gamma_4)=1$ and $\varepsilon_1(\Gamma_4)=0$). These are four induced subgraphs, such that $\varepsilon_1(\Gamma_i)=|V(\Gamma_i)|-1$  for $i\in [4]$. Moreover, let $\Gamma_5$ be the induced subgraph of $\Gamma$  on five vertices with an orange spanning tree ($V(\Gamma_5)=5$ and $\varepsilon_1(\Gamma_5)=8$), $\Gamma_6$ on four vertices having green  spanning tree ($V(\Gamma_6)=4$ and $\varepsilon_1(\Gamma_6)=5$), $\Gamma_7$ a vertex with two loops ($V(\Gamma_7)=1$  and $\varepsilon_1(\Gamma_7)=1$), as shown in Figure~\ref{fig:spanning-of-huge-gamma}. 
     
   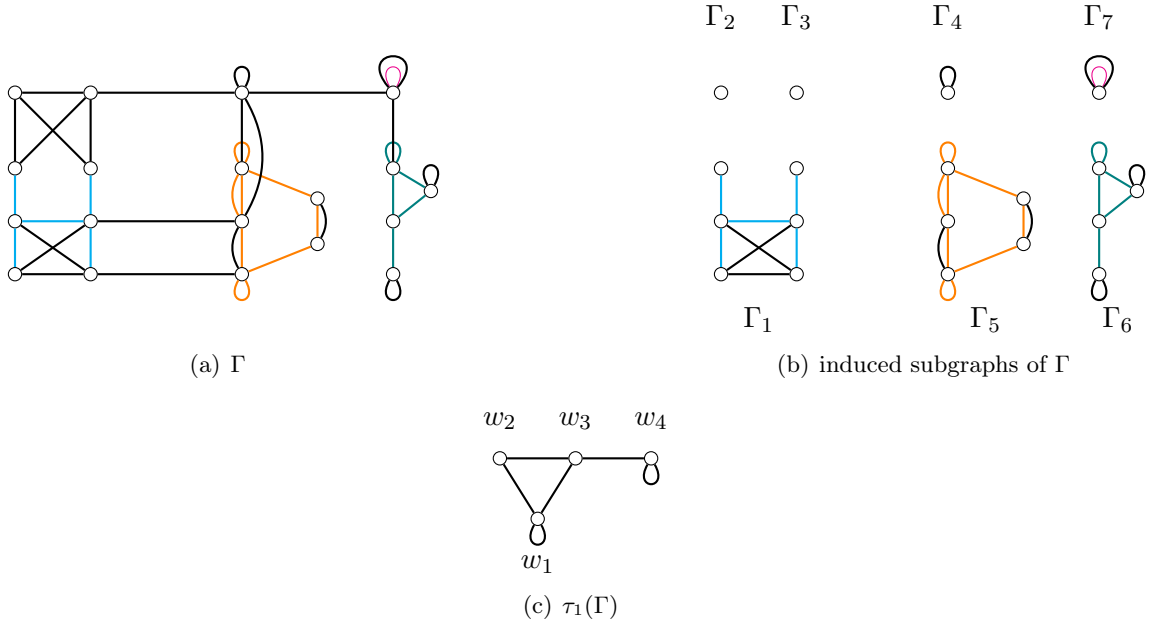
\begin{figure}[htb]
\centering
  \subfigure[$\Gamma$]{ \label{fig:huge-gamma-repeated}
    \begin{tikzpicture}[
  vertex/.style={circle,draw,fill=white,inner sep=1.2pt,minimum size=5pt},
  cyanedge/.style={cyan,thick},
  orangeedge/.style={orange,thick},
  tealedge/.style={teal,thick},
  blackedge/.style={black,thick}
  ]
 
  \node[vertex] (L_t)  at (0,1)  {};
  \node[vertex] (L_m)  at (0,0.3)  {};
  \node[vertex] (L_b)  at (0,-0.4)    {};
  \node[vertex] (R_tL) at (1,1){};
  \node[vertex] (R_mL) at (1,0.3){};
  \node[vertex] (R_bL) at (1,-0.4)  {};

  \draw[cyanedge] (L_t) -- (L_m) -- (L_b);
  \draw[cyanedge] (R_tL) -- (R_mL) -- (R_bL);
  \draw[cyanedge] (L_m) -- (R_mL);
  \draw[blackedge] (L_b) -- (R_mL);
  \draw[blackedge] (L_m) -- (R_bL);

  \node[vertex] (D_t) at (3,1.0) {};
  \node[vertex] (D_m) at (3,0.3) {};
  \node[vertex] (D_b) at (3,-0.4){};
  \node[vertex] (D_rt) at (4,0.6){};
  \node[vertex] (D_rb) at (4,0){};

  \draw[orangeedge] (D_t) -- (D_rt) -- (D_rb) -- (D_b) -- (D_m) --(D_t) ;
\path (D_m)  edge [bend left,color=orange,thick] (D_t) edge [bend right,color=black,thick]  (D_b);
\path (D_rt)  edge [bend left,color=black,thick] (D_rb);

  \node[vertex] (P_t) at (5,1) {};
  \node[vertex] (P_m) at (5,0.3) {};
  \node[vertex] (P_b) at (5,-0.4){};
  \node[vertex] (P_r) at (5.5,0.7){};  
  
  \draw[tealedge] (P_m) -- (P_r) -- (P_t) -- (P_m) -- (P_b);

  \node[vertex] (B0) at (0,2.0) {};
  \node[vertex] (B1) at (1,2.0) {};
  \node[vertex] (B2) at (3,2.0) {};
  \node[vertex] (B3) at (5,2.0) {};
  \phantom{\node[rectangle,draw=none] at (3.5,-1) {$\Gamma_5$};
\node[rectangle,draw=none] at (5.25,-1) {$\Gamma_6$};
\node[rectangle,draw=none] at (5,3) {$\Gamma_7$};
\node[rectangle,draw=none] at (0.5,-1) {$\Gamma_1$};
\node[rectangle,draw=none] at (3,3) {$\Gamma_4$};
\node[rectangle,draw=none] at (0,3) {$\Gamma_2$};
\node[rectangle,draw=none] at (1,3) {$\Gamma_3$};}
  \tikzset{every loop/.style={min distance=4mm,in=120,out=60,looseness=10}}
\draw[color=magenta] (B3) to[loop above] (B3);
\draw[blackedge] (B2) to[loop above] (B2);
\draw[tealedge] (P_t) to[loop above] (P_t);
\draw[blackedge] (P_r) to[loop above] (P_r);
\draw[orangeedge] (D_t) to[loop above] (D_t);
\tikzset{every loop/.style={min distance=4mm,in=130,out=50,looseness=15}}
\draw[blackedge] (B3) to[loop above] (B3);
\tikzset{every loop/.style={min distance=4mm,in=240,out=300,looseness=10}}
\draw[orangeedge] (D_b) to[loop below] (D_b);
\draw[blackedge] (P_b) to[loop below] (P_b);

  \draw[blackedge] (B0)--(B1) -- (B2) -- (B3);

  \draw[blackedge] (L_t) -- (B0)--(R_tL);
  \draw[blackedge] (L_t) -- (B1)--(R_tL);
  \draw[blackedge] (D_t) -- (B2);
  \draw[blackedge] (P_t) -- (B3);
  \draw[blackedge] (L_b)--(R_bL);

  \draw[blackedge] (R_bL) -- (D_b);
  \draw[blackedge] (R_mL) -- (D_m);
  \draw[blackedge] (B2) edge [bend left,color=black,thick] (D_m);
\end{tikzpicture}}
\hfill\subfigure[induced subgraphs of $\Gamma$]{ \label{fig:spanning-of-huge-gamma}
    \begin{tikzpicture}[
  vertex/.style={circle,draw,fill=white,inner sep=1.2pt,minimum size=5pt},
  cyanedge/.style={cyan,thick},
  orangeedge/.style={orange,thick},
  tealedge/.style={teal,thick},
  blackedge/.style={black,thick}
  ]
 
  \node[vertex] (L_t)  at (0,1)  {};
  \node[vertex] (L_m)  at (0,0.3)  {};
  \node[vertex] (L_b)  at (0,-0.4)    {};
  \node[vertex] (R_tL) at (1,1){};
  \node[vertex] (R_mL) at (1,0.3){};
  \node[vertex] (R_bL) at (1,-0.4)  {};

  \draw[cyanedge] (L_t) -- (L_m) -- (L_b);
  \draw[cyanedge] (R_tL) -- (R_mL) -- (R_bL);
  \draw[cyanedge] (L_m) -- (R_mL);
  \draw[blackedge] (L_b) -- (R_mL);
  \draw[blackedge] (L_m) -- (R_bL)--(L_b);

  \node[vertex] (D_t) at (3,1.0) {};
  \node[vertex] (D_m) at (3,0.3) {};
  \node[vertex] (D_b) at (3,-0.4){};
  \node[vertex] (D_rt) at (4,0.6){};
  \node[vertex] (D_rb) at (4,0){};

  \draw[orangeedge] (D_t) -- (D_rt) -- (D_rb) -- (D_b) -- (D_m) --(D_t) ;
\path (D_m)  edge [bend left,color=orange,thick] (D_t) edge [bend right,color=black,thick]  (D_b);
\path (D_rt)  edge [bend left,color=black,thick] (D_rb);

  \node[vertex] (P_t) at (5,1) {};
  \node[vertex] (P_m) at (5,0.3) {};
  \node[vertex] (P_b) at (5,-0.4){};
  \node[vertex] (P_r) at (5.5,0.7){};  
  
  \draw[tealedge] (P_m) -- (P_r) -- (P_t) -- (P_m) -- (P_b);
  
  \node[vertex] (B0) at (0,2.0) {};
  \node[vertex] (B1) at (1,2.0) {};
  \node[vertex] (B2) at (3,2.0) {};
  \node[vertex] (B3) at (5,2.0) {};
  \tikzset{every loop/.style={min distance=4mm,in=120,out=60,looseness=10}}
\draw[color=magenta] (B3) to[loop above] (B3);
\draw[blackedge] (B2) to[loop above] (B2);
\draw[tealedge] (P_t) to[loop above] (P_t);
\draw[blackedge] (P_r) to[loop above] (P_r);
\draw[orangeedge] (D_t) to[loop above] (D_t);
\tikzset{every loop/.style={min distance=4mm,in=130,out=50,looseness=15}}
\draw[blackedge] (B3) to[loop above] (B3);
\tikzset{every loop/.style={min distance=4mm,in=240,out=300,looseness=10}}
\draw[blackedge] (P_b) to[loop below] (P_b);
\draw[orangeedge] (D_b) to[loop below] (D_b);

\node[rectangle,draw=none] at (3.5,-1) {$\Gamma_5$};
\node[rectangle,draw=none] at (5.25,-1) {$\Gamma_6$};
\node[rectangle,draw=none] at (5,3) {$\Gamma_7$};
\node[rectangle,draw=none] at (0.5,-1) {$\Gamma_1$};
\node[rectangle,draw=none] at (3,3) {$\Gamma_4$};
\node[rectangle,draw=none] at (0,3) {$\Gamma_2$};
\node[rectangle,draw=none] at (1,3) {$\Gamma_3$};
\end{tikzpicture}}

\subfigure[$\tau_1(\Gamma)$]{  \label{fig:tau-1-Gamma}
    \begin{tikzpicture}[
  vertex/.style={circle,draw,fill=white,inner sep=1.2pt,minimum size=5pt},
  cyanedge/.style={cyan,thick},
  blackedge/.style={black,thick}
  ]

   \node[vertex] (H) at (0.5,1.2){};
  
  \node[vertex] (B0) at (0,2.0) {};
  \node[vertex] (B1) at (1,2.0) {};
  \node[vertex] (B2) at (2,2.0) {};
  \phantom{\node[vertex, fill=orange] (C1) at (1,1) {};}
 \tikzset{every loop/.style={min distance=4mm,in=240,out=300,looseness=10}}
\draw[blackedge] (B2) to[loop below] (B2);
\draw[blackedge] (H) to [loop below] (H);
  \draw[blackedge] (B1)--(H) -- (B0) -- (B1) -- (B2);

  \node[rectangle,draw=none] at (0.5,0.6) {$w_1$};
  \node[rectangle,draw=none] at (0,2.5) {$w_2$};
  \node[rectangle,draw=none] at (1,2.5) {$w_3$};
  \node[rectangle,draw=none] at (2,2.5) {$w_4$};
  
\end{tikzpicture}}
    \caption{A weighted multigraph $\Gamma$, where all the blue, orange, green and pink edges and loops have weight $1$ and all black edges and loops have weight $0$, its $1$-components and $\tau_1(\Gamma)$.}
    \label{fig:huge-multigraph-tau-reduction}
\end{figure}

Using Definition~\ref{def:tau1} we obtain $V(\tau_1(\Gamma))=\{w_1,w_2,w_3,w_4\}$ and $$E(\tau_1(\Gamma))=\{\{w_1,w_1\},\{w_1,w_2\},\{w_1,w_3\},\{w_2,w_3\},\{w_3,w_4\},\{w_4,w_4\}\},$$
all with weights $0$. This results in graph $\tau_1(\Gamma)$ in Figure~\ref{fig:tau-1-Gamma}.

Note that $\tau_1(\Gamma)$ is equal to connected component $\Psi$ of $\Gamma'''$ in Figure~\ref{fig:huge-gamma'''}. The other four vertices of degree $0$ in $\Gamma'''$ do not appear in $\tau_1(\Gamma)$, but their number is equal to $$t_1(\Gamma)=(8-5)+(5-4)+(1-1)=4.$$
\end{example}

\begin{proposition}\label{prop:tau1}
 If  $\Gamma=(V(\Gamma),E(\Gamma),w_{\Gamma})$ is a weighted multigraph with binomial weights, then     $$\gapzvezdica(\Gamma)=\gapzvezdica(\tau_1(\Gamma))+t_1(\Gamma).$$
\end{proposition}
\begin{proof}
 Let $\Gamma_1,\ldots,\Gamma_k$ be all the $1$-components of $\Gamma$ and let $\ell\in\{0,1,\ldots,k\}$ be as in Definition~\ref{def:tau1}.
 Let $\Gamma'$ be a graph obtained from $\Gamma$ using $1$-edge contraction consecutively on $\sum_{i\in[k]} (|V(\Gamma_i)|-1)$ edges of weight $1$. This results in a graph $\Gamma'$ with $V(\Gamma')=\{w_1,\ldots,w_k\}$, where the vertex $w_i$ is the $1$-edge contraction of $\Gamma_i$, $i\in [k]$. 
 Note that each $w_i$ has $\varepsilon_1(\Gamma_i)-|V(\Gamma_i)|+1$ loops of weight $1$ and $\varepsilon_0(\Gamma_i)$ loops of weight $0$. Moreover, for each $e\in E(\Gamma)$ which is incident to $u\in V(\Gamma_i)$ and $v\in V(\Gamma_{j})$ for some $i,j \in [k]$, 
 there is $e'\in \Gamma'$ with $w_{\Gamma'}(e')=0$  and is incident to $w_i$ and $w_j$. By Lemma~\ref{lem:1-edge-contraction} we have $\gap(\Gamma)=\gap(\Gamma')$.
    
 Note that for $i\in [\ell]$ vertices $w_i$ have no loops of weight $1$ and all vertices $w_i$, $\ell < i\leq k$,  
 have at least one loop of weight $1$. Let $\Gamma''$ be a graph we obtain from  $\Gamma'$ by eliminating one loop of $w_i$ with weight $1$, for all  $i=\ell+1,\ldots,k$, together with $w_i$ and all edges (and loops) of weight $0$ incident to $w_i$. 
 Graph $\Gamma''$ contains $t_1(\Gamma)$ copies of $K_1$ arising from remaining loops of weight $1$ in $\Gamma'$, and let us denote by $\Psi$ the union of other components of $\Gamma''$. 
 Using Lemma~\ref{lem:removing-1-loops} it follows that $\gapzvezdica(\Gamma')=\gapzvezdica(\Gamma'')=\gapzvezdica(\Psi)+t_1(\Gamma)$. 
 Note that $V(\Psi)=\{w_1,\ldots,w_{\ell}\}$. Lastly, we eliminate all remaining parallel $0$-edges of $\Psi$ and we obtain the graph $\tau_1(\Gamma)$ and using Lemma~\ref{lem:parallel-0-edges} the proposition is proved.
\end{proof}

 \subsection{Operation $\tau_2$} Recall that in a multigraph $\Gamma$ we denote the set of all leaves by $\leaves(\Gamma)$.

 \begin{definition}
 For every weighted multigraph $\Gamma=(V(\Gamma),E(\Gamma),w_{\Gamma})$ with binomial weights let 
  $$N_\Gamma[\leaves(\Gamma)] = \bigcup_{v \in \leaves(\Gamma)} N_\Gamma[v]$$
 the set of all neighbours of leaves in $\Gamma$. We define
 $$\tau_2(\Gamma) := \Gamma\left(\leaves(\Gamma)\cup N_\Gamma[\leaves(\Gamma)]\right)$$
 to be the weighted multigraph obtained from $\Gamma$ by eliminating all leaves and their neighbours, and
 $$t_2(\Gamma) := |\leaves(\Gamma)| - |N_\Gamma[\leaves(\Gamma)]|.$$
 \end{definition}

To clarify the definition above, we offer some comments: 
\begin{itemize}
\item $t_2(\Gamma) \ge 0$ since every leaf has a unique neighbour.
\item If $\Gamma$ has no leaves then $\tau_2(\Gamma) = \Gamma$ and $t_2(\Gamma) =0$. 
\item The multigraph $\tau_2(\Gamma)$ may contain (new) leaves. 
\item The operation $\tau_2$ eliminates any connected components of $\Gamma$ that are isomorphic to $P_2$ with at most one looped vertex.
\end{itemize}

 \begin{proposition}\label{prop:tau2}
 If a weighted multigraph $\Gamma=(V(\Gamma),E(\Gamma),w_{\Gamma})$ with binomial weights contains no edges of weight $1$, then 
 $$\gapzvezdica(\Gamma) = \gapzvezdica(\tau_2(\Gamma)) + t_2(\Gamma).$$
 \end{proposition}
     
\begin{proof}
Note, eliminating connected components of $\Gamma$ that are isomorphic to $P_2$ with at most one looped vertex does not change $\gapzvezdica$. Hence, we can assume that no such components exist. In particular, $\leaves(\Gamma)\cap N_\Gamma[\leaves(\Gamma)]  = \emptyset$.

Choose $u \in N_\Gamma[\leaves(\Gamma)]$ and let $v\in \leaves(\Gamma)$ be one of its neighbouring leaves. Since  $\Gamma$ contains no edges of weight $1$ we have $\Gamma(v)_{-u} = \Gamma(\{u,v\})$. Moreover, $\Gamma(\{u,v\})$ is equal to $\Gamma\left((N_\Gamma[u]\cap\leaves(\Gamma))\cup\{u\}\right)$ together with $t_2(u):=|\leaves(\Gamma)\cap N_\Gamma[u]|-1$ vertices of degree $0$, arising from leaves that are neighbours of $u$. 
 By Lemma~\ref{lem:removing-leaves-on-chdv}
 $$\gapzvezdica(\Gamma)=\gapzvezdica(\Gamma(v)_{-u})=\gapzvezdica(\Gamma\left((N_\Gamma[u]\cap\leaves(\Gamma))\cup\{u\}\right))+t_2(u).$$
The result follows by repeating this process for all $u \in N_\Gamma[\leaves(\Gamma)]$. 
\end{proof}

\begin{example}
Consider the graph $\Gamma'$ in Figure~\ref{fig:Gamma-twin-leaves}. 
      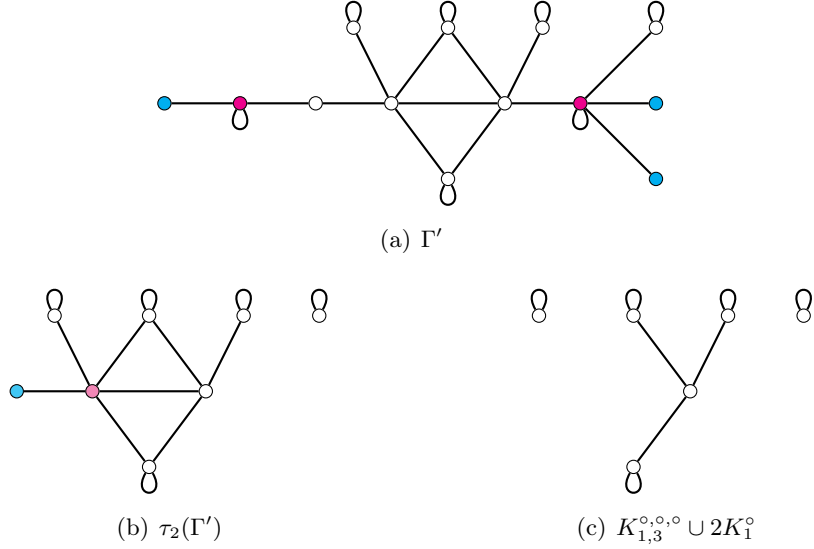
\begin{figure}[htb]
\centering
\subfigure[$\Gamma'$]{  \label{fig:Gamma-twin-leaves}
    \begin{tikzpicture}[
  vertex/.style={circle,draw,fill=white,inner sep=1.2pt,minimum size=5pt},
  blackedge/.style={black,thick}
  ]

   \node[vertex] (H+) at (0.25,1){};
   \node[vertex] (H-) at (0.25,-1){};

  \node[vertex] (BLL) at (-1,1) {};
  \node[vertex] (BL) at (-0.5,0) {};
  \node[vertex] (BR) at (1,0) {};
  \node[vertex,fill=magenta] (N) at (2,0) {};
  
  \node[vertex] (L+) at (3,1) {};
  \node[vertex,fill=cyan] (L0) at (3,0) {};
  \node[vertex,fill=cyan] (L-) at (3,-1) {};

  \node[vertex,fill=cyan] (AL) at (-3.5,0) {};
  \node[vertex,fill=magenta] (AN) at (-2.5,0) {};
  \node[vertex] (E) at (-1.5,0) {};
  \node[vertex] (LL) at (1.5,1) {};
  
 \tikzset{every loop/.style={min distance=4mm,in=120,out=60,looseness=10}}
\draw[blackedge] (H+) to [loop above] (H+);
\draw[blackedge] (L+) to [loop above] (L+);
\draw[blackedge] (LL) to [loop above] (LL);
\draw[blackedge] (BLL) to [loop above] (BLL);

 \tikzset{every loop/.style={min distance=4mm,in=240,out=300,looseness=10}}
\draw[blackedge] (N) to[loop below] (N);
\draw[blackedge] (H-) to [loop below] (H-);
\draw[blackedge] (AN) to [loop below] (AN);

  \draw[blackedge] (AL)--(AN)--(E)--(BL)--(H+) -- (BR) --(H-) -- (BL)-- (BR)-- (N) -- (L0);
  \draw[blackedge] (L+)--(N)--(L-);
  \draw[blackedge] (LL)--(BR);
  \draw[blackedge] (BLL)--(BL);

\end{tikzpicture}}

\subfigure[$\tau_2(\Gamma')$]{  \label{fig:tau2-Gamma-twin-leaves}
    \begin{tikzpicture}[
  vertex/.style={circle,draw,fill=white,inner sep=1.2pt,minimum size=5pt},
  blackedge/.style={black,thick}
  ]

   \node[vertex] (H+) at (0.25,1){};
   \node[vertex] (H-) at (0.25,-1){};

  \node[vertex] (BLL) at (-1,1) {};
  \node[vertex,fill=magenta!60] (BL) at (-0.5,0) {};
  \node[vertex] (BR) at (1,0) {};

  \node[vertex,fill=cyan!60] (E) at (-1.5,0) {};
  \node[vertex] (LL) at (1.5,1) {};

  \node[vertex] (sing) at (2.5,1) {};
  
 \tikzset{every loop/.style={min distance=4mm,in=120,out=60,looseness=10}}
\draw[blackedge] (H+) to [loop above] (H+);
\draw[blackedge] (LL) to [loop above] (LL);
\draw[blackedge] (BLL) to [loop above] (BLL);
\draw[blackedge] (sing) to [loop above] (sing);

 \tikzset{every loop/.style={min distance=4mm,in=240,out=300,looseness=10}}
\draw[blackedge] (H-) to [loop below] (H-);

  \draw[blackedge] (E)--(BL)--(H+) -- (BR) --(H-) -- (BL)-- (BR);
  \draw[blackedge] (LL)--(BR);
  \draw[blackedge] (BLL)--(BL);
  
\end{tikzpicture}}
\hspace{2cm}
\subfigure[$K_{1,3}^{\circ,\circ,\circ}\cup 2K_1^{\circ}$]{  \label{fig:tau2^2-Gamma-twin-leaves-reduced}
    \begin{tikzpicture}[
  vertex/.style={circle,draw,fill=white,inner sep=1.2pt,minimum size=5pt},
  blackedge/.style={black,thick}
  ]

   \node[vertex] (H+) at (0.25,1){};
   \node[vertex] (H-) at (0.25,-1){};
  
  \node[vertex] (BR) at (1,0) {};
  
  \node[vertex] (LL) at (1.5,1) {};
  
  \node[vertex] (sing2) at (2.5,1) {};
  \node[vertex] (sing1) at (-1,1) {};

 \tikzset{every loop/.style={min distance=4mm,in=120,out=60,looseness=10}}
\draw[blackedge] (H+) to [loop above] (H+);
\draw[blackedge] (LL) to [loop above] (LL);
\draw[blackedge] (sing1) to [loop above] (sing1);
\draw[blackedge] (sing2) to [loop above] (sing2);

 \tikzset{every loop/.style={min distance=4mm,in=240,out=300,looseness=10}}
\draw[blackedge] (H-) to [loop below] (H-);

  \draw[blackedge] (H+) -- (BR) --(H-);
  \draw[blackedge] (LL)--(BR);
  
\end{tikzpicture}}
    \caption{A multigraph $\Gamma'$ and its reductions to smaller graphs using $\tau_2$. In all three graphs all the edge weights are assumed to be equal to $0$, the vertices from $\leaves(\Gamma')$ and $\leaves(\tau_2(\Gamma'))$ are colored blue,  and the vertices from $N_\Gamma[V_{d=1}(\Gamma')]$ and $N_{\tau_2(\Gamma')}[V_{d=1}(\tau_2(\Gamma'))]$ are colored pink.}
    \label{fig:removal-twin-leaves}
\end{figure}

Observe that $\tau_2(\Gamma')$ has a vertex of degree $1$ and $\tau_2(\tau_2(\Gamma'))=K_{1,3}^{\circ,\circ,\circ}\cup 2K_1^{\circ}$ is shown in Figure~\ref{fig:tau2^2-Gamma-twin-leaves}. Since $t_2(\Gamma)=3-2=1$  and $t_2(\tau_2(\Gamma'))=0$, Proposition~\ref{prop:tau2} implies that 
$$\gapzvezdica(\Gamma')=\gapzvezdica(\tau_2(\Gamma'))+1=\gapzvezdica(K_{1,3}^{\circ,\circ,\circ}\cup 2K_1^{\circ})+1=\gapzvezdica(K_{1,3}^{\circ,\circ,\circ})+1.$$
\end{example}
 
 \subsection{Operation $\tau_3$}

The multigraph $\Gamma = \kappa(G)$ contains no vertices of degree $2$, but after operations $\tau_1$ and $\tau_2$ the multigraph $\tau_2(\tau_1(\Gamma))$ may contain such vertices. 
The multigraph obtained by eliminating all vertices of degree $2$ in $\Gamma$ and replacing the corresponding induced paths with edges of weight $0$ or $1$, using Lemma~\ref{lem:removing-ver-deg-2} consequently, is the multigraph $\kappa(\zeta(\Gamma))$. 
Furthermore, we count vertices of degree $0$ which we obtain from parallel $1$-loops and twin leaves using functions $t_1$ and $t_2$, but the multigraph $\tau_2(\tau_1(\Gamma))$ may still contain vertices of degree $0$. 
We eliminate them now. Note that operation $\kappa(\zeta(\Gamma))$ also eliminates vertices with a single 0-loop and no other incident edges. 

Recall that in a multigraph $\Gamma$ we denote the set of all vertices of degree $0$ by $\isolated(\Gamma)$. 
We define
$$\tau_3(\Gamma) := \kappa(\zeta(\Gamma(\isolated(\Gamma)))) \; \text{ 
and } \;
t_3(\Gamma) := |\isolated(\Gamma)|.$$
If $\Gamma$ does not contain any vertices of degree $0$ or $2$ then $\tau_3(\Gamma) = \Gamma$ and $t_3(\Gamma) = 0$.
By equation~\eqref{eq:gap-gapzvezdica} we have 
\begin{equation}\label{eq:tau3}
    \gapzvezdica(\Gamma) = \gapzvezdica(\tau_3(\Gamma)) + t_3(\Gamma).
\end{equation}
Note that the multigraph $\tau_3(\Gamma)$ may contain edges of weight $1$, even if $\Gamma$ does not. 

As an example, consider a graph 
$\tau_1(\Gamma)$ from Figure~\ref{fig:tau-1-Gamma}. The graph $\tau_3(\tau_1(\Gamma))$, shown in Figure~\ref{fig:tau-31-Gamma} contains an edge of weight $1$. On the other hand, observe that for
$\Gamma''=K_{1,3}^{\circ,\circ,\circ}\cup 2K_1^{\circ}$ in Figure~\ref{fig:tau2^2-Gamma-twin-leaves} we have $\Gamma(\isolated(\Gamma''))=\Gamma''$, therefore $\tau_3(\Gamma)=\kappa(\zeta(\Gamma''))=K_{1,3}^{\circ,\circ,\circ}$ and $t_3(\Gamma'')=0$.

   \begin{figure}[htb]
\centering
  
\subfigure[$\tau_3(\tau_1(\Gamma))$]{  \label{fig:tau-31-Gamma}
\begin{tikzpicture}[
  vertex/.style={circle,draw,fill=white,inner sep=1.2pt,minimum size=5pt},
  cyanedge/.style={cyan,thick},
  blackedge/.style={black,thick}
  ]

  \node[vertex] (B0) at (0,2.0) {};
  \node[vertex] (B1) at (1,2.0) {};
  \node[vertex] (B2) at (2,2.0) {};
  \phantom{\node[vertex]  at (2,0) {};}
  
 \tikzset{every loop/.style={min distance=4mm,in=240,out=300,looseness=10}}
\draw[blackedge] (B2) to [loop below] (B2);
\draw[blackedge] (B0) to [loop below] (B0);
  \draw[blackedge] (B1) -- (B2);
  \draw[blackedge] (B0)edge [bend right,color=black,thick] (B1);
  \draw[blackedge] (B0) edge [bend left,color=black,thick] (B1);
  \node[rectangle,draw=none] at (0.5,1.5) {$1$};
  \node[rectangle,draw=none] at (0.5,2.5) {$0$};
  \node[rectangle,draw=none] at (1.5,2.2) {$0$};
  \node[rectangle,draw=none] at (0,1.3) {$0$};
  \node[rectangle,draw=none] at (2,1.3) {$0$};
\end{tikzpicture}}
\hspace{2cm}
\subfigure[$\tau_3(K_{1,3}^{\circ,\circ,\circ}\cup 2K_1^{\circ})$]{  \label{fig:tau2^2-Gamma-twin-leaves}
    \begin{tikzpicture}[
  vertex/.style={circle,draw,fill=white,inner sep=1.2pt,minimum size=5pt},
  blackedge/.style={black,thick}
  ]

   \node[vertex] (H+) at (0.25,1){};
   \node[vertex] (H-) at (0.25,-1){};
  \node[vertex] (BR) at (1,0) {};
  \node[vertex] (LL) at (1.5,1) {};
    \phantom{\node[vertex]  at (-1,1) {};}
    \phantom{\node[vertex]  at (2,1) {};}
  
 \tikzset{every loop/.style={min distance=4mm,in=120,out=60,looseness=10}}
\draw[blackedge] (H+) to [loop above] (H+);
\draw[blackedge] (LL) to [loop above] (LL);

\node[rectangle,draw=none] at (1.5,1.6) {$0$};
\node[rectangle,draw=none] at (0.25,1.6) {$0$};
\node[rectangle,draw=none] at (0.25,-1.6) {$0$};
\node[rectangle,draw=none] at (0.65,-0.8) {$0$};
\node[rectangle,draw=none] at (0.65,0.8) {$0$};
\node[rectangle,draw=none] at (1.2,0.8) {$0$};

 \tikzset{every loop/.style={min distance=4mm,in=240,out=300,looseness=10}}
\draw[blackedge] (H-) to [loop below] (H-);

  \draw[blackedge] (H+) -- (BR) --(H-);
  \draw[blackedge] (LL)--(BR);

\end{tikzpicture}}
    \caption{Two examples of $\tau_3$ operations.}
    \label{fig:tau3-multigraph}
\end{figure}
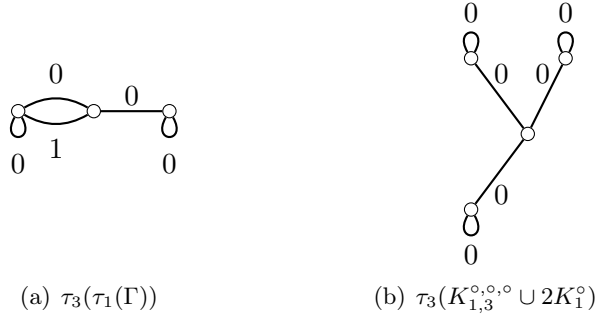

\subsection{An algorithm for computing $\tau(\Gamma)$}

Using operations $\tau_1$, $\tau_2$ and $\tau_3$ we can define an algorithm that reduces a multigraph $\Gamma=(V(\Gamma),E(\Gamma),w_{\Gamma})$ with binomial weights to a multigraph that is invariant for those operations. That means that no further reductions developed in Section \ref{sec:operations-preserve-gap*} are possible. This algorithm is described in Figure~\ref{fig:algorithm}. 
    
\begin{figure}[htb]
\begin{tikzpicture}[
    startstop/.style={rectangle, rounded corners, minimum width=2.5cm, minimum height=0.7cm, text centered, draw=black, fill=gray!20, font=\bfseries},
    function/.style={rectangle, minimum width=2.0cm, minimum height=0.6cm, text centered, draw=black, fill=white},
    decision/.style={diamond, aspect=2, minimum width=2.5cm, minimum height=0.7cm, text centered, draw=black, fill=gray!10},
    arrow/.style={-{Stealth}, thick}
]

\node (start) [startstop]{$\Gamma_0:=$ input};

\node (f0) [function, below=0.5cm of start] {$\Gamma:=\Gamma_0$,  $t:=0$};
\draw [arrow] (start) -- (f0);

\node (f1) [function, below=0.5cm of f0] {$\Gamma_{\text{old}}:=\Gamma$};
\draw [arrow] (f0) -- (f1);

\node (f2) [function, below=0.5cm of f1] {$(\Gamma,t):=(\tau_1(\Gamma),t+t_1(\Gamma))$};
\draw [arrow] (f1) -- (f2);

\node (f3) [function, below=0.5cm of f2] {$(\Gamma,t):=(\tau_2(\Gamma),t+t_2(\Gamma))$};
\draw [arrow] (f2) -- (f3);

\node (f4) [function, below=0.5cm of f3] {$(\Gamma,t):=(\tau_3(\Gamma),t+t_3(\Gamma))$};
\draw [arrow] (f3) -- (f4);

\node (dec) [decision, below=0.5cm of f4, align=center] {$\Gamma_{\text{old}}= \Gamma$?};
\draw [arrow] (f4) -- (dec);

\node (stop) [startstop, below=1.3cm of dec] {Return $(\Gamma,t)=:(\tau(\Gamma_0), t(\Gamma_0))$};
\draw [arrow] (dec) -- node[right] {Yes} (stop);

\draw [arrow] (dec.west) -| ([xshift=-1.5cm]dec.west) |- (f1.west);
\node [left] at ([xshift=-1.5cm]dec.west |- f3.west) {No};
\end{tikzpicture}

    \caption{An algorithm for computing $\tau(\Gamma_0)$ and $t(\Gamma_0)$}
    \label{fig:algorithm}
\end{figure}
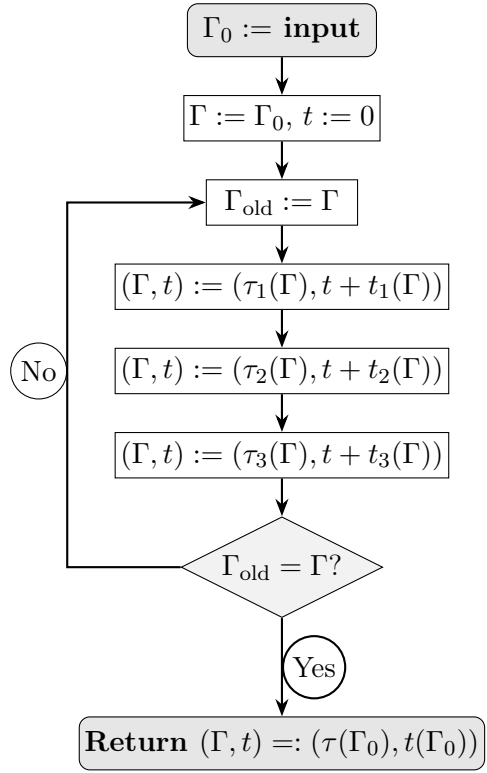

\begin{definition}
Let $\Gamma=(V(\Gamma),E(\Gamma),w_{\Gamma})$ be a multigraph with binomial weights.  The output of the algorithm described in Figure \ref{fig:algorithm}, given the input $\Gamma$, is defined as the pair $(\tau(\Gamma), t(\Gamma))$.     
\end{definition}

Note that $\tau(\Gamma)$ is a weighted simple graph with possible loops, with all edge weights equal to $0$ and all vertices of degree at least $3$.

\begin{proposition}
    Suppose that for a weighted multigraph with binomial weights $\Gamma$ we have $\tau(\Gamma)=\Gamma$. Let $H$ be the graph $\Gamma$ viewed as a simple graph (i.e., we disregard the weights in $\Gamma$).  If $H\in \GnoFourCycle$, then $\gapzvezdica(\Gamma)=\gap(H)$. 
\end{proposition}

\begin{proof}
Note that $\kappa(H)=\Gamma$ since $H$ contains only vertices of at least $3$. So $\gapzvezdica(\Gamma)=\gap(H)$ by equation~\eqref{eq:gap-gapzvezdica}.
\end{proof}

The next example shows that the above proposition is not true if $H \notin \GnoFourCycle$.

\begin{example}
Let $\Gamma = W_5$ be the wheel weighted multigraph with all edge weights equal to $0$. Then $\tau(\Gamma) = \Gamma$ and the graph $H = W_5$ has $\gap(H) = 2$ by Example~\ref{ex:W5}. On the other hand, we will show in Proposition~\ref{prop:at-most-6-vertices} that $\gapzvezdica(\Gamma) = 0$. To do that, we will need another tool, developed in Section~\ref {sec:vtx-removal}.
\end{example}

We summarize the findings of this section in the following theorem.
  
  \begin{theorem}\label{thm:gap-tau}
Let $\Gamma=(V(\Gamma),E(\Gamma),w_{\Gamma})$ be a weighted multigraph with binomial weights. Then
    $$\gapzvezdica(\Gamma)=\gapzvezdica(\tau(\Gamma))+t(\Gamma).$$
\end{theorem}

\begin{proof}
The claim follows directly from Propositions \ref{prop:tau1}, \ref{prop:tau2}, and equation~\eqref{eq:tau3}. Note that Proposition~\ref{prop:tau2} can be applied, since in the algorithm we use operation $\tau_2$ on the multigraph $\tau_1(\Gamma)$ which does not contain any $1$-edges.
\end{proof}

\begin{example}\label{ex:big-example-algorithm}
     For a detailed example of the algorithm see Figure~\ref{fig:algorithm-example}, where we present each step on the graph $\Gamma$ from Figure~\ref{fig:huge-gamma}. All black edges in graphs are of weight $0$, and all colored edges are of weight $1$. The input is the graph $\Gamma$ shaded gray in the top left corner. For $\Gamma$ the algorithm runs through three loops, and those are presented in the three columns. Since in the first two columns the top ($\Gamma_{\text{old}}$) and the bottom
       ($\Gamma$) graphs are not the same, the algorithm runs through another loop. In the third column the top and the bottom graph coincide ($\Gamma_{\text{old}}=\Gamma$), hence the algorithm stops. The output is the graph $\Gamma_2^{\circ,\circ}:=\tau(\Gamma)$ (shaded gray in the bottom right), and $t(\Gamma)=4$.
       By Theorem~\ref{thm:gap-tau} it follows that $$\gapzvezdica(\Gamma)=\gapzvezdica(\Gamma_2^{\circ,\circ})+4,$$
which gives a faster approach to Example~\ref{ex:huge-gamma-chapter-4}. In Example~\ref{ex:P200} we will compute $\gapzvezdica(\Gamma_2^{\circ,\circ})=0$.
\end{example}

   \begin{figure}[htb]
       \centering
     \begin{tikzpicture}[
  vertex/.style={circle,draw,fill=white,inner sep=1.2pt,minimum size=5pt},
  cyanedge/.style={cyan,thick},
  orangeedge/.style={orange,thick},
  tealedge/.style={teal,thick},
  blackedge/.style={black,thick}
  ]
  
      \begin{scope}[shift={(-0.7,1.4)}]
      \draw[fill=gray!30] (-0.5,2.75) rectangle (6,-1.25);
      \hugeGamma
       \node[rectangle,draw=none] at (3.25,-1) {$t=0$};        
       \node[rectangle,draw=none] at (2.75,-2) {$\downarrow$};
       \end{scope}

       \begin{scope}[scale=0.7,shift={(2,-4.5)}]
      \draw (-0.5,2.75) rectangle (3,0.5);
      \tauGammaNoLabel
       \node[rectangle,draw=none] at (1.25,0.75) {$t=4$};        
       \node[rectangle,draw=none] at (1.25,0.25) {$\downarrow$};
       \end{scope}

       \begin{scope}[scale=0.7,shift={(2,-7.5)}]
      \draw (-0.5,2.75) rectangle (3,0.5);
      \tauGammaNoLabel
       \node[rectangle,draw=none] at (1.25,0.75) {$t=4$};        
       \node[rectangle,draw=none] at (1.25,0.25) {$\downarrow$};
       \end{scope}
       
       \begin{scope}[scale=0.7,shift={(2,-9.5)}]
      \draw (-0.5,2) rectangle (3,0.5);
      \begin{scope}[shift={(1,0)}]\tauGammaDeg
      \end{scope}
       \node[rectangle,draw=none] at (1.25,0.75) {$t=4$};        
       \end{scope}

       \begin{scope}[scale=0.7,shift={(7.25,-2)}]
      \draw (-0.5,2) rectangle (3,0.5);
      \begin{scope}[shift={(1,0)}]\tauGammaDeg
      \end{scope}
       \node[rectangle,draw=none] at (1.25,0.75) {$t=4$};
       \node[rectangle,draw=none] at (1.25,0.25) {$\downarrow$};
       \end{scope}

      \begin{scope}[scale=0.7,shift={(8,-4.5)}]
      \draw (-0.5,2) rectangle (1.5,0.5);
      \tauGammatauGamma
       \node[rectangle,draw=none] at (0.5,0.75) {$t=4$};        
       \node[rectangle,draw=none] at (0.5,0.25) {$\downarrow$};
       \end{scope}
      
       \begin{scope}[scale=0.7,shift={(8,-7)}]
      \draw (-0.5,2) rectangle (1.5,0.5);
      \tauGammatauGamma
       \node[rectangle,draw=none] at (0.5,0.75) {$t=4$};        
       \node[rectangle,draw=none] at (0.5,0.25) {$\downarrow$};
       \end{scope}
      
       \begin{scope}[scale=0.7,shift={(8,-9.5)}]
      \draw (-0.5,2) rectangle (1.5,0.5);
      \tauGammatauGamma
       \node[rectangle,draw=none] at (0.5,0.8) {$t=4$};        
       \end{scope}
       
        \begin{scope}[scale=0.7,shift={(12.5,-2)}]
      \draw (-0.5,2) rectangle (1.5,0.5);
      \tauGammatauGamma
       \node[rectangle,draw=none] at (0.5,0.75) {$t=4$};        
       \node[rectangle,draw=none] at (0.5,0.25) {$\downarrow$};
       \end{scope}

      \begin{scope}[scale=0.7,shift={(12.5,-4.5)}]
      \draw (-0.5,2) rectangle (1.5,0.5);
      \tauGammatauGamma
       \node[rectangle,draw=none] at (0.5,0.75) {$t=4$};        
       \node[rectangle,draw=none] at (0.5,0.25) {$\downarrow$};
       \end{scope}
       \begin{scope}[scale=0.7,shift={(12.5,-7)}]
      \draw (-0.5,2) rectangle (1.5,0.5);
      \tauGammatauGamma
       \node[rectangle,draw=none] at (0.5,0.75) {$t=4$};        
       \node[rectangle,draw=none] at (0.5,0.25) {$\downarrow$};
       \end{scope}
       \begin{scope}[scale=0.7,shift={(12.5,-9.5)}]
      \draw[fill=gray!30] (-0.5,2) rectangle (1.5,0.5);
      \tauGammatauGamma
       \node[rectangle,draw=none] at (0.5,0.8) {$t=4$};        
       \end{scope}
       
       \draw[thick] (3.5,-5.8)--(4,-5.8)--(4,-0.5);
       \draw[->,thick](4,-0.5)--(4.72,-0.5);
       
       \begin{scope}[shift={(3.18,0)}]
       \draw[thick] (3.5,-5.8)--(4.5,-5.8)--(4.5,-0.5);
       \draw[->,thick](4.5,-0.5)--(5.23,-0.5);
        \end{scope}
       \end{tikzpicture}
       \caption{An illustration of an algorithm presented in Figure~\ref{fig:algorithm} on graph $\Gamma$ from Figure~\ref{fig:huge-gamma}. All black edges in graphs are of weight $0$ and all colored edges are of weight $1$.}
       \label{fig:algorithm-example}
   \end{figure}
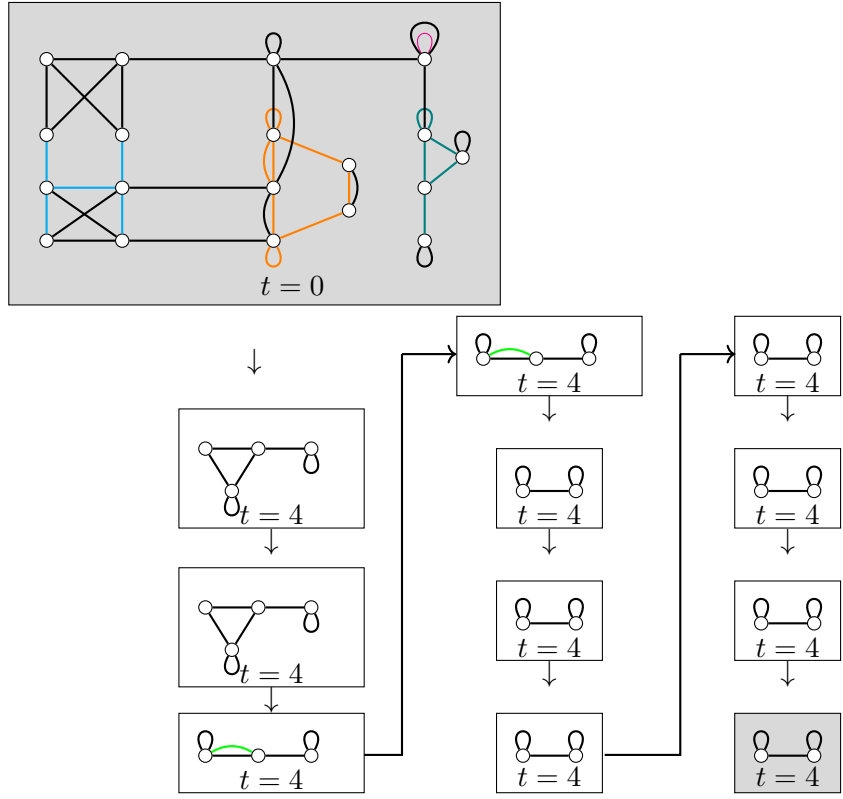

\section{Computing $\gapzvezdica(\Gamma)$ with vertex removal}\label{sec:vtx-removal}

The algorithm developed in Section \ref{sec:tau} allows us to compute the gap in many but not all cases. 
In this section we develop a complementary method to compute $\gapzvezdica(\Gamma)$ of a reduced multigraph $\tau(\Gamma)$ inductively by removing vertices. This method has high computational complexity but it allows us to compute the gap of for every graph $G \in \GnoFourCycle$.

To prove the main result of this section, Theorem~\ref{thm:gapzvezdica-one-vertex-removed}, we need to return to simple graphs and their gaps. 

\subsection{Vertex removal for simple graphs}

\begin{lemma}\label{lem:st and gap inq}
Let $G$ be a simple graph with $|V(G)|\geq 2$, and $v\in V(G)$. Then
\begin{equation}\label{eq:with_gap}
\gap(G(v))-1\leq \gap(G)\leq \gap(G(v))+1.
\end{equation}
 Moreover, if $G\in \GnoFourCycle$ and there exists an $\mathrm{ABC}$ cover $\C$ of $G$ such that $v$ is of Type A with respect to $\C$, then the left inequality is an equality. 
\end{lemma}

\begin{proof}
    Note that \eqref{eq:with_gap} is equivalent to 
    $$\stp(G(v))+2\geq \stp(G)\geq \stp(G(v)).$$
    The inequality $\stp(G)\geq \stp(G(v))$ follows from the fact that $G(v)$ is an induced subgraph of $G$. If we add $\{v\}\vee N_G[v]$ to a set-join cover $\C'$ of $G(v)$, we get a set-join cover of $G$. As this introduces at most two new components, this implies  $\stp(G(v))+2\geq \stp(G)$. 

   Now, assume that $G\in \GnoFourCycle$ and there exists an ABC cover $\C$ of $G$, where $\mc K\vee \{v\} \in \C$ with $|\mc K| \geq 2$. By Corollary~\ref{cor:2c2d} there exists a set-join cover $\C'$ of $G(v)$ with $|V(\C')|\leq |V(\C)|-2$. Hence, $\stp(G(v))\leq \stp(G)-2$, as desired. 
\end{proof}

\begin{remark}\label{cor:G-vtx-removal}
If $|V(G)|\geq 2$ and there exist  $v,w \in V(G)$ and $k\in\bN$ so that $\gap(G(v))= k-1$ and $\gap(G(w))= k+1$, then $\gap(G) = k$.
\end{remark}

Next, we express the gap of a graph $G$ in terms of gaps of graphs $G(v)$, where $v$ runs over a certain subset of $V(G)$. Using this lemma inductively, we can in principle compute $\gap(G)$ for any graph $G \in \GnoFourCycle$. Since the gap of a graph is the sum of the gaps of its components, we may assume that the graph is connected.

\begin{lemma}\label{lem:verify-thm-for-HDV-NHDV-only}
   Let $G\in \GnoFourCycle$ be a connected graph with $|V(G)|\geq 2$. Let $\mc S$ be the set of all  vertices in $v \in V(G)$ satisfying one of the following conditions:
     \begin{enumerate}[(A)]
         \item\label{thm:part:HDV-only} $\deg(v)\ne 2$, 
         \item[(B)]          $\deg(v)=2$, and there exists an induced walk  $w_1-\pathk{v}{w_2}{r}$ in $G$, such that $w_i\in V_{d\ne 2}(G)$ for $i\in[2]$ and $r$ is even.
         \label{thm:part:deg-two-only}
     \end{enumerate} 
Then
\begin{equation} \label{eq:setS}
    \gap(G)=\max\{0,\max_{v\in \mc S}\{\gap(G(v))-1\}\}.
\end{equation}
\end{lemma}

\begin{proof}
We first prove the following claim:
Let $G\in \GnoFourCycle$  be a graph with $|V(G)| \ge 2$ that contains no vertices of degree $0$. Then
\begin{equation}\label{eq:gap_inductive}
\gap(G)=\max\{0,\max_{v\in V(G)}\{\gap(G(v))-1\}\}.
\end{equation}
Note that $\gap(G)\geq \gap(G(v))-1$ for all $v \in V(G)$ by Lemma~\ref{lem:st and gap inq}, so $$\gap(G)\ge\max\{0,\max_{v\in V(G)}\{\gap(G(v))-1\}\}.$$   If there exists an optimal set-join cover $\C$ of $G$ containing $\mc K\vee \{w\}$ with $|\mc K| \geq 2$, then by the proof of Proposition~\ref{prop:Types} there exists an ABC cover of $G$ where $w$ is of type $A$. By Lemma~\ref{lem:st and gap inq}  we now have $\gap(G)=\gap(G(w))-1$, and the claim follows. Otherwise, $G$ has the unique optimal set-join cover $\C$ where all its components are singletons. Since $G$ has no vertices of degree $0$, this means that $\gap (G)=0$ and the claim follows again. (Note, the equation \eqref{eq:gap_inductive} already allows us to compute $\gap(G)$ inductively. We proceed to show that the set of vertices that need to be removed in the process can be reduced to the set $\mc S$.)

For $G=C_n$, $n\ne 4$, the set $\mc S$ is empty and $\gap(G)=0$,  so the equation~\eqref{eq:setS} holds. From now on we assume that $G$ is not an $n$-cycle. Using the above claim, it suffices to show that for every $u \in V(G) \setminus \mc S$, there exists $v \in \mc S$ so that $\gap(G(v))=\gap(G(u))$. 

Any $u \in V(G) \setminus \mc S$ satisfies $\deg(u)=2$, and since $G$ is not a cycle, $u$ is a vertex of an induced walk $w_1-u_1-\ldots-u_{r+1}-w_2$ in $G$, $r\geq 0$, such that $\deg(w_i)\ne 2$ for $i\in[2]$, $\deg(u_i)=2$ for $i\in[r+1]$, and $u=u_{i_0}$ for some $i_0\in [r+1]$. 
(We allow $w_1=w_2$ or $\deg(w_i)=1$.) Note $w_1, w_2 \in \mc S$, since they both satisfy the assumption (A). If $r$ is even, then $u_1 \in \mc S$, since it satisfies the assumption (B). 

We claim that we can choose $v \in \{w_1,w_2,u_1\}$  so that   $\gap(G(u))=\gap(G(v))$. Observe that $u \in  V(G) \setminus \mc S$ implies that at least one of its neighbors is $G$ becomes a leaf in $G(u)$. Hence, Proposition~\ref{prop:from KBS24}~(\ref{prop:2.1(1)}) allows us to conclude:
    $$\gap(G(u))=\gap(G(\{w_1,u_1\ldots,u_{i_0-1},u\}))=\gap(G(w_1))$$
    if $i_0$ is even, 
    $$\gap(G(u))=\gap(G(\{u,u_{i_0+1}\ldots,u_r,w_2\}))=\gap(G(w_2))$$
   if $i_0$ is odd and $r$ is even,
    and 
    $$\gap(G(u))=\gap(G(\{u_1,\ldots,u_{i_0-1},u\}))=\gap(G(u_1)).$$ 
    if $i_0$ is odd and $r$ is odd. 
\end{proof}

\subsection{Vertex removal for multigraphs}

In this section, we use Lemma~\ref{lem:verify-thm-for-HDV-NHDV-only} to prove a similar result for weighted multigraphs. 

\begin{lemma} \label{lem:gapzvezdica-one-vertex-removed}
Let $\Gamma=(V(\Gamma),E(\Gamma),w_{\Gamma})$ be an nonempty weighted multigraph with all weights equal to $0$ and $v \in V(\Gamma)$. Then
$$\gapzvezdica(\Gamma)\ge \gapzvezdica(\Gamma(v))-1.$$
\end{lemma}

\begin{proof}
Since $\gapzvezdica(\Gamma_1 \cup \Gamma_2) = \gapzvezdica(\Gamma_1) + \gapzvezdica(\Gamma_2)$ we may assume without loss of generality that $\Gamma$ is connected. If $\Gamma$ contains only one vertex, then $\gapzvezdica(\Gamma(v)) = 0$ and the claim is true. So suppose $|V(\Gamma)| \ge 2$. Since $\Gamma$ contains only edges of weight $0$, we have $\Gamma(v) = \Gamma_{-v}$ by Remark~\ref{rem:gamma-v}, so 
$$\gapzvezdica(\Gamma) = \gap(\zeta(\Gamma)) \ge \gap(\zeta(\Gamma)(v)) - 1 = \gapzvezdica(\Gamma_{-v}) - 1 = \gapzvezdica(\Gamma(v)) - 1 $$
by Lemmas~\ref{lem:st and gap inq} and \ref{lem:zetaGamma-v}.
\end{proof}

Let $\Gamma$ be a weighted multigraph. A set $\mc I \subseteq V(\Gamma)$ is called \emph{independent}, if the subgraph of $\Gamma$ induced on $\mc I$ contains no edges. The cardinality of the biggest independent set in $\Gamma$ is called \emph{independence number} and denoted by $\alpha(\Gamma)$. The independence number  $\alpha(\Gamma)$ can be used to find a lower bound on the $\gapzvezdica(\Gamma)$.

\begin{corollary}\label{cor:removing-k-vertices}
    Let $\Gamma$ be a weighted multigraph with all weights equal to  $0$. For any vertices $v_1,\ldots, v_{k}\in V(\Gamma)$ we have
    $$\gapzvezdica(\Gamma)\geq 
    \gapzvezdica(\Gamma(\{v_1,\ldots,v_{k}\}))-k.$$ 
    In particular, 
  \begin{equation}\label{eq:gap-alpha-ineq}\gapzvezdica(\Gamma)\geq 2 \alpha(\Gamma)-|V(\Gamma)|.
  \end{equation}
\end{corollary}
\begin{proof}
    The first claim follows from Lemma~\ref{lem:gapzvezdica-one-vertex-removed} by induction on $k$. Let $I$ be maximal independent set in $\Gamma$, let $k = |V(\Gamma)| - \alpha(\Gamma)$ and $V(\Gamma) \setminus I = \{v_1,\ldots,v_{k}\}$. Then 
    $$\gapzvezdica(\Gamma) \geq  \gapzvezdica(\Gamma(\{v_1,\ldots,v_{k}\}))-k = \alpha(\Gamma) - (|V(\Gamma)| - \alpha(\Gamma)) 
     = 2 \alpha(\Gamma)-|V(\Gamma)|,$$
    since $\Gamma(\{v_1,\ldots,v_{k}\})$ contains $\alpha(\Gamma)$ vertices and no edges.
\end{proof}

The above estimate of $\gapzvezdica$ in terms of independence number is sharp for some families of graphs (see e.g. Example~\ref{ex:Kmn}) while in general $\gapzvezdica(\Gamma)- 2 \alpha(\Gamma)+|V(\Gamma)|$ can be arbitrarily large (see e.g. Example~\ref{ex:Gamma12}).

Let $G \in \GnoFourCycle$, and $\Gamma = \tau(\kappa(G))$. ($\Gamma$ is a simple multigraph with all weights equal to $0$ and each vertex of degree at least 3.) By Theorem~\ref{thm:gap-tau} we have:
$$\gap(G) = \gapzvezdica(\Gamma) + t(\kappa(G)).$$ To compute $\gapzvezdica(\Gamma)$ we may assume that $\Gamma$ is a connected graph. In our examples, we will apply the following theorem only to multigraphs of the form $\Gamma = \tau(\kappa(G))$, but we state and prove it more generally. 

\begin{theorem} \label{thm:gapzvezdica-one-vertex-removed}
Let $\Gamma=(V(\Gamma),E(\Gamma),w_{\Gamma})$ be a weighted multigraph with all weights equal to $0$, $|V(\Gamma)| \ge 2$, and no vertices of degree $0$. 
Then
$$\gapzvezdica(\Gamma)=\max\{0,\max_{v\in V(\Gamma)}\{\gapzvezdica(\Gamma(v))-1\}\}.$$
\end{theorem}

\begin{proof}
     Suppose first that $\Gamma$ is connected. Then $G:=\zeta(\Gamma)$ is connected and $V(\Gamma) \subseteq V(G)$. 
     By definition we have $\gapzvezdica(\Gamma) = \gap(G)$. By Lemma~\ref{lem:verify-thm-for-HDV-NHDV-only} we have
     $$\gap(G)=\max\{0,\max_{v\in \mc S}\{\gap(G(v))-1\}\},$$
     where $\mc S$ is as defined in the lemma. 
     Since all edges in $\Gamma$ are of weight $0$, we can describe $G(v)$   for $v \in V(\Gamma)$ in terms of $\zeta(\Gamma(v))$ as follows: the graph $G(v)$ is isomorphic to the graph $\zeta(\Gamma(v))$ with added pendent  $P_2$ for each nonloop edge in $\Gamma$ incident to $v$ and added disjoint $P_2$ for each loop on $v$. By Proposition~\ref{prop:from KBS24}~(\ref{prop:2.1(1)}) it follows that 
     $$\gapzvezdica(\Gamma(v)) = \gap(\zeta(\Gamma(v))) = \gap(G(v)).$$
     Let $v$ be any vertex in $\mc S$. If $\deg(v)\neq 2$, then $v \in V(\Gamma)$. If $\deg(v)=2$, then there exists an induced walk  $w_1-\pathk{v}{w_2}{r}$ in $G$ with $\deg(w_i)\ne 2$ for $i\in[2]$, and $r\geq 0$ even. 
     Let $\pathk{v}{w_2}{r} =  v-u_1-\ldots-u_r-w_2$. 
     By construction of $G$ we have $r=1$ or $u_2 \in V(\Gamma)$. Since $r$ is even, it follows that $u_2 \in V(\Gamma)$. Now
     $$\gap(G(v))=\gap(G(\{v,u_1, u_2\}))=\gap(G(u_2)) = \gapzvezdica(\Gamma(u_2)),$$
     which implies that 
     $$\max_{v\in V(\Gamma)}\{\gapzvezdica(\Gamma(v))\} = \max_{v\in S}\{\gap(G(v))\} $$
     and the result follows. 

     Suppose now that $\Gamma = \Gamma_1 \cup \ldots \cup \Gamma_t \cup sK_1^\circ$, where the graphs $\Gamma_1, ..., \Gamma_t$ are connected and have at least $2$ vertices.
     Then 
     $$\gapzvezdica(\Gamma_i)=\max\{0,\max_{v\in V(\Gamma_i)}\{\gapzvezdica(\Gamma_i(v))-1\}\}$$
     for every $i \in [t]$, as proved above. 
     
     If $\gapzvezdica(\Gamma_i) = 0$ for all $i \in [t]$, then $\gapzvezdica(\Gamma(v)) \le 1$ for all $v\in V(\Gamma)$ and the claim follows. Assume now there exists a $j$ such that $\gapzvezdica(\Gamma_j)=\max_{v\in V(\Gamma_j)}\{\gapzvezdica(\Gamma_j(v))-1\}$. Then
     there exists $u \in \Gamma_j$ such that $\gapzvezdica(\Gamma_j) = \gapzvezdica(\Gamma_j(u))-1$. It follows that 
     $$\gapzvezdica(\Gamma(u)) = \sum_{i=1}^t \gapzvezdica(\Gamma_i) + 1 = \gapzvezdica(\Gamma)+1$$
    and so $\gapzvezdica(\Gamma) \le \max_{v\in V(\Gamma)}\{\gapzvezdica(\Gamma(v))-1\}$.
     On the other hand,  $\gapzvezdica(\Gamma) \ge \gapzvezdica(\Gamma(v))-1$ for all $v \in V(\Gamma)$, by Lemma~\ref{lem:gapzvezdica-one-vertex-removed}, so 
     $$\gapzvezdica(\Gamma)=\max_{v\in V(\Gamma)}\{\gapzvezdica(\Gamma(v))-1\},$$
     and the claim follows.
\end{proof}

\begin{example}\label{ex:P200}
To complete the computation of $\gapzvezdica(\Gamma)$ considered in 
Example~\ref{ex:big-example-algorithm}  we need to compute $\gapzvezdica(\Gamma_0^{\circ,\circ})$, where $\Gamma_0^{\circ,\circ}$  is the output graph in Figure~\ref{fig:algorithm-example}, i.e. a weighted multigraph obtained from $P_2^{\circ,\circ}$ (as in Figure~\ref{fig:P2-OO}) by assigning weight $0$ to all three edges.
Let $V(P_2^{\circ,\circ})=V(\Gamma_2^{\circ,\circ})=\{u,v\}$. 

Since $\Gamma_2^{\circ,\circ}(u)$ is isomorphic to a vertex with a loop of weight $0$, we have $\gapzvezdica(\Gamma_2^{\circ,\circ}(u))=\gap(C_6)=0$. Using symmetry  and Theorem~\ref{thm:gapzvezdica-one-vertex-removed} it follows that $\gapzvezdica(\Gamma_2^{\circ,\circ})=\max\{0,-1,-1\}=0$.
    Therefore, by Example~\ref{ex:big-example-algorithm} we obtain
    $$\gapzvezdica(\Gamma)=\gapzvezdica(\Gamma_2^{\circ,\circ})+4=4.$$

    Since $P_2^{\circ,\circ}\notin\GnoFourCycle$ (when viewed as a simple graph with loops), Lemma~\ref{lem:verify-thm-for-HDV-NHDV-only} cannot be used to compute its gap. However, the vertices $u$ and $v$ are twins, so by Proposition~\ref{prop:from KBS24}(\ref{prop:2.1(2)}) $\gap(P_2^{\circ,\circ})=\gap(K_1^{\circ})+1=1$. As $\gapzvezdica(\Gamma_2^{\circ,\circ})\neq \gap(P_2^{\circ,\circ})$, this example highlights the difference between the two parameters. 
\end{example}

\section{Applications}\label{sec:app}

With the results of the previous two sections, we can compute $\gapzvezdica$ of any multigraph $\Gamma$. We conclude the paper with a few applications. 

We first compute $\gapzvezdica$ for small multigraphs and then for some concrete examples. 

\begin{proposition}\label{prop:at-most-6-vertices}
   Let  $\Gamma=(V(\Gamma),E(\Gamma),w_{\Gamma})$ be a weighted multigraph with binomial weights.
  If every component of $\tau(\Gamma)$ has at most six vertices 
  then  $$\gapzvezdica(\Gamma)=t(\Gamma).$$

    In particular, if $\tau(\Gamma) = \Gamma$ and $|V(\Gamma)| \le 6$, then $\gapzvezdica(\Gamma) = 0$.
\end{proposition}

\begin{proof} Using Theorem~\ref{thm:gap-tau} we may assume that $\tau(\Gamma) = \Gamma$ and $|V(\Gamma)| \le 6$.
    In particular, every vertex in $\Gamma$ has degree at least $3$, which implies that for any pair of vertices $u,v \in V(\Gamma)$, all vertices in  $\Gamma(\{u,v\})$ have degree at least $1$. This allows us to apply  Theorem \ref{thm:gapzvezdica-one-vertex-removed} to $\Gamma$, $\Gamma(u)$, and $\Gamma(\{u,v\})$. Suppose that $|V(\Gamma)| = 6$, and let $u, v, w \in V(\Gamma)$ be three distinct vertices. Since $|V(\Gamma(\{u,v,w\}))| = 3$, we have $\gapzvezdica(\Gamma(\{u,v,w\})) \le 3$. By Theorem \ref{thm:gapzvezdica-one-vertex-removed} $\gapzvezdica(\Gamma(\{u,v\})) \le 2$, $\gapzvezdica(\Gamma(u)) \le 1$, and $\gapzvezdica(\Gamma) =0$. Similarly, if $|V(\Gamma)| = 5$ ($|V(\Gamma)| = 4$), then $\gapzvezdica(\Gamma(\{u,v,w\})) \le 2$ ($\gapzvezdica(\Gamma(\{u,v,w\})) \le 1$), again implying $\gapzvezdica(\Gamma) =0$. Finally, assuming $|V(\Gamma)| = 3$ or $|V(\Gamma)| = 2$, we have $\gapzvezdica(\Gamma(u)) \le 1$ and the conclusion follows. 
\end{proof}

Application of Theorem \ref{thm:gapzvezdica-one-vertex-removed} simplifies for graphs with a large degree of symmetry. This is illustrated in the following examples. 

\begin{example} \label{ex:Petersen}
    Let $P_{5,2}$ be the Petersen graph. 
    Note that the multigraph $\Gamma_{5,2} = \chdv(P_{5,2})$ is a weighted Petersen graph with all edge weights $0$. 
    Since $\tau(\Gamma_{5,2})=\Gamma_{5,2}$, we proceed to compute $\gapzvezdica(\Gamma_{5,2})$ using Theorem~\ref{thm:gapzvezdica-one-vertex-removed}. 
    
    Using symmetries of $\Gamma_{5,2}$, we observe that removing any vertex $v$ from $\Gamma_{5,2}$ results in a multigraph isomorphic to $\Gamma'$ in Figure~\ref{fig:petersen-gamma'}.
    From $\tau(\Gamma')=\emptyset$ and $t(\Gamma')=0$ we get $\gapzvezdica(\Gamma')=0$ by Theorem~\ref{thm:gap-tau}. Therefore,   $$\gapzvezdica(\Gamma_{5,2}) = \max\{0, \gapzvezdica(\Gamma')-1\}=0$$ by Theorem~\ref{thm:gapzvezdica-one-vertex-removed}, and so  $\gap(P_{5,2}) = 0$. 
    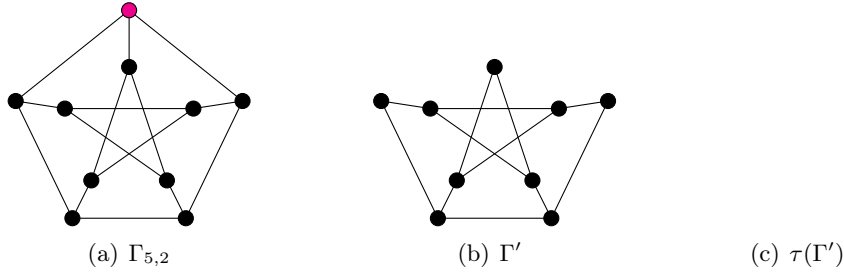
\begin{figure}[h]
\centering
\subfigure[$\Gamma_{5,2}$]{\label{fig:petersen}\begin{tikzpicture}[scale=0.5]
\tikzset{
every node/.style={draw, circle, inner sep=2pt}
}
\path (0,4) edge (1,1) edge (-1,1);
\path (-3,3.1) edge (-1.7,2.9) edge (-1.5,0);
\path (3,3.1) edge (1.7,2.9) edge (1.5,0);
\path (-1,1) edge (1.7,2.9) edge (-1.5,0);
\path (1,1) edge (-1.7,2.9) edge (1.5,0);
\path (-1.7,2.9) edge (1.7,2.9);
\path (-1.5,0) edge (1.5,0);
\path (0,5.5) edge (-3,3.1) edge (3,3.1)  edge (0,4);
\node[fill=magenta] at (0,5.5) {}; 
\node[fill=black] at (0,4) {}; 
\node[fill=black] at (-3,3.1) {}; 
\node[fill=black] at (3,3.1) {}; 
\foreach \i in {(-1.5,0),(1.5,0),(-1.7,2.9),(1.7,2.9),(-1,1),(1,1)} {
    \node[fill=black] at \i {}; }
\end{tikzpicture}}
\qquad
\subfigure[$\Gamma'$]{\label{fig:petersen-gamma'}
\begin{tikzpicture}[scale=0.5]
\tikzset{
every node/.style={draw, circle, inner sep=2pt}
}
\path (0,4) edge (1,1) edge (-1,1);
\path (-3,3.1) edge (-1.7,2.9) edge (-1.5,0);
\path (3,3.1) edge (1.7,2.9) edge (1.5,0);
\path (-1,1) edge (1.7,2.9) edge (-1.5,0);
\path (1,1) edge (-1.7,2.9) edge (1.5,0);
\path (-1.7,2.9) edge (1.7,2.9);
\path (-1.5,0) edge (1.5,0);
\foreach \i in {(-1.5,0),(1.5,0),(-1.7,2.9),(1.7,2.9),(-1,1),(1,1)} {
    \node[fill=black] at \i {}; }
\phantom{ \node[fill=black] at (-4.2,0) {}; 
\node[fill=black] at (4.2,0) {}; }
\node[fill=black] at (0,4) {}; 
\node[fill=black] at (-3,3.1) {}; 
\node[fill=black] at (3,3.1) {}; 
\end{tikzpicture}}
 \subfigure[$\tau(\Gamma')$]{ 
\phantom{\begin{tikzpicture}[scale=0.5]
\tikzset{
every node/.style={draw, circle, inner sep=2pt}
}
\node[fill=black] at (-3,3.1) {}; 
\node[fill=black] at (3,3.1) {}; 
\end{tikzpicture}}}
\caption{Graph $\Gamma_{5,2}=\kappa(P_{5,2})$ is a weighted multigraph with all black edges of weight $0$. $\Gamma'$  is obtained by removing a pink vertex from $\Gamma$, and $\tau(\Gamma')$  is an empty graph.}
\label{fig:Petersen-and-tau-gamma}
\end{figure}
\end{example}

\begin{example}\label{ex:Kn}
Let $n \ge 2$ and $\Gamma_n$   be a complete graph $K_n$ (without loops) and with all edge weights equal to 0.  We prove that $\gapzvezdica(\Gamma_n) =0$ by induction on $n$.
    
If $n\leq 3$ we have $\gapzvezdica(\Gamma_2) =\gap(P_4)=0$ and $\gapzvezdica(\Gamma_3)=\gap(C_9)=0$. If
$n \ge 4$ then $\tau(\Gamma_n)=\Gamma_n$ and by induction hypothesis $\gapzvezdica(\Gamma_{n-1}) =0$. For any  $v\in V(\Gamma_n)$ we have $\Gamma_n(v) = \Gamma_{n-1}$, so that by Theorem~\ref{thm:gapzvezdica-one-vertex-removed}
    $$\gapzvezdica(\Gamma_n)=\max\{0,\gapzvezdica(\Gamma_{n-1})-1\} = 0.$$
 The value of  $\gap(K_n)$ can be found in  \cite[Lemma 4.11.]{KBS24}. 
\end{example}

We show in the next example that Proposition~\ref{prop:at-most-6-vertices} is no longer true if $\tau(\Gamma)$ has a connected component with more than six vertices. 

\begin{example} \label{ex:Kmn}
    Let $m, n \ge 1$, $m \ge n$, and  $\Gamma_{m,n}$  be a complete bipartite graph $K_{m,n}$ with all edge weights equal to 0. We claim: $\gapzvezdica(\Gamma_{m,n}) = m-n$, and prove this claim by induction on $m+n$. 
    
    Note,  
    $\tau_2(\Gamma_{m,1}) = \emptyset$ and $t_2(\Gamma_{m,1}) = m-1$, so $\gapzvezdica(\Gamma_{m,1}) = m-1$. In particular,  $\gapzvezdica(\Gamma_{1,1})=\gapzvezdica(\Gamma_2) = 0$.
    
    Assume $n > 1$. The induction hypothesis asserts $\gapzvezdica(\Gamma_{(m-1),n)}) = |m-1-n|$ and $\gapzvezdica(\Gamma_{m,(n-1)}) = m-(n-1)$. Let  $V(\Gamma_{m,n}) = M \cup N$ with $|M|=m$, $|N|=n$ and $E(\Gamma_{m,n}[M])=E(\Gamma_{m,n}[N])=\emptyset$.
    If $m=n$, then $\Gamma_{m,n}(v) \cong \Gamma_{m,(m-1)}$  for any  $v\in V(\Gamma_{m,n})$. Applying Theorem \ref{thm:gapzvezdica-one-vertex-removed} we conclude: 
    $$\gapzvezdica(\Gamma_{m,n})=\max\{0,\gapzvezdica(\Gamma_{m,(m-1)})-1\} =\max\{0,1-1\} = 0 = m-n.$$
    If $m>n$, then  $\Gamma_{m,n}(v) \cong \Gamma_{(m-1),n}$ for any  $v\in M$ and $\Gamma_{m,n}(v) \cong \Gamma_{m,(n-1)}$  for any  $v\in N$. An application of Theorem \ref{thm:gapzvezdica-one-vertex-removed} now completes the proof:
    \begin{align*}
        \gapzvezdica(\Gamma_{m,n}) &= \max\{0,\gapzvezdica(\Gamma_{(m-1),n})-1, \gapzvezdica(\Gamma_{m,(n-1)})-1\}\\
        &=\max\{0, (m-1-n)-1,(m-(n-1))-1\} = m-n.
    \end{align*}
   In particular, $|V(\Gamma_{4,3})|=7$ and $\gapzvezdica(\Gamma_{4,3}) = 1$.

   Note that $\gap(K_{m,n})=m+n-2$ by Proposition~\ref{prop:from KBS24}(\ref{prop:2.1(2)}).
\end{example}

\begin{example}\label{ex:Gamma12} Let $m\ge 5$ be an integer,  let $\Gamma_{m,3}$ be as defined in Example~\ref{ex:Kmn},  and $\Gamma^{\circ}_{m-4}$ a weighted complete multigraph with all the loops on $m-4$ vertices, and all weights equal to $0$. A proof by induction on $m$ (similar as in Example~\ref{ex:Kn}) establishes $\gapzvezdica(\Gamma^{\circ}_{m-4}) = 0$. 

Choose a vertex $u \in V(\Gamma_{m,3})$ of degree $m$, and a vertex $v \in V(\Gamma^{\circ}_{m-4})$, and define
$\Gamma$ to be a weighted multigraph on $2m-1$ vertices with $V(\Gamma) = V(\Gamma_{m,3}) \cup V(\Gamma^{\circ}_{m-4})$, $E(\Gamma) = E(\Gamma_{m,3}) \cup E(\Gamma^{\circ}_{m-4}) \cup \{\{u,v\}\}$, and all weights equal to $0$. 
The biggest independent set of $\Gamma$ is the set of all vertices in $\Gamma_{m,3}$ of degree $3$. This can be seen by noting that a vertex with a loop cannot be a member of an independent set, and that in a complete bipartite multigraph, any independent set is a subset of one of the parts. 
It follows that $\alpha(\Gamma) = m$  and thus $2\alpha(\Gamma)-|V(\Gamma)| = 2m - (2m-1) = 1$. Furthermore, $\Gamma(u) \cong \Gamma_{m,2} \cup \Gamma^{\circ}_{m-4}$, so 
$$\gapzvezdica(\Gamma(u)) = \gapzvezdica(\Gamma_{m,2}) + \gapzvezdica(\Gamma^{\circ}_{m-4}) = m-2$$
by Example~\ref{ex:Kmn}.
Hence, by Lemma~\ref{lem:gapzvezdica-one-vertex-removed} $\gapzvezdica(\Gamma) \ge \gapzvezdica(\Gamma(u)) - 1 = m-3$, and
$$\gapzvezdica(\Gamma) - (2 \alpha(\Gamma)-|V(\Gamma)|) \ge m-4.$$
as desired. This example shows that there can be an arbitrarily large difference between the left and the right-hand side in the inequality~\eqref{eq:gap-alpha-ineq}.
\end{example}

Next example is an application of Corollary~\ref{cor:removing-k-vertices} and offers a planar multigraph with all vertices of degree at least $3$, all weights $0$, and an arbitrarily large $\gapzvezdica$.

\begin{example}
Let  $\Gamma$ be a weighed graph with $3k+2$ vertices $$V(\Gamma)=\{x,y\} \, \cup \, \{z_{1,j},z_{2,j},z_{3,j}\colon j\in[k]\},$$
and $7k$ edges
\begin{align*}
E(\Gamma)=\{\{x,&z_{i,j}\}\colon i\in [3], j\in [k]\} \cup \{\{y,z_{i,j}\}\colon i\in \{1,3\}, j\in [k]\} \\
&\cup \, \{\{z_{1,j},z_{2,j}\},\{z_{2,j},z_{3,j}\}\colon j\in [k]\},
\end{align*}
all with weight $0$, see Figure~\ref{fig:planar-graph-arbitrary-large-gap*}.

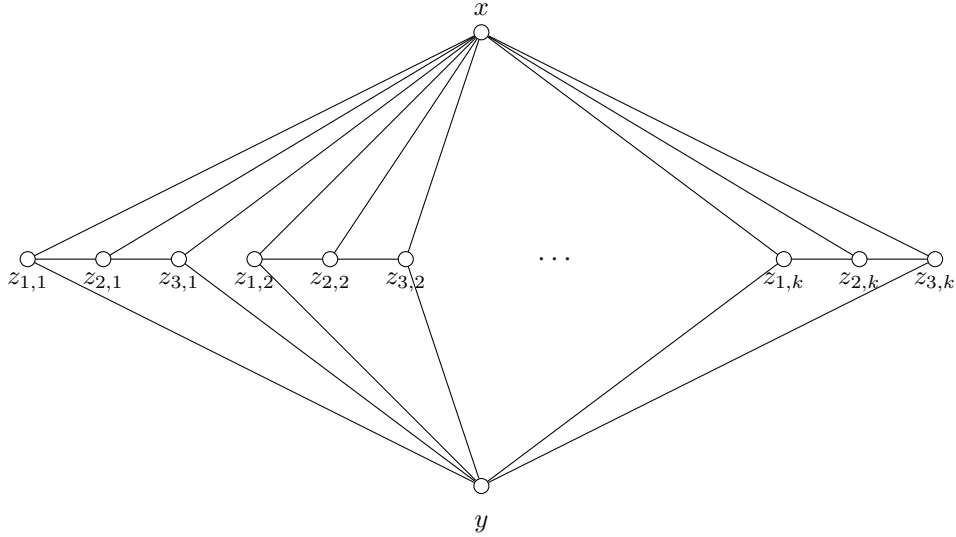
\begin{figure}[htb]
    \begin{tikzpicture}
\node (x)  at (0,3) {}; 
\node (y)  at (0,-3) {}; 

 \foreach \i in {-6,-3,4} {
    \draw (y)--(\i,0)--(\i+2,0)--(y);
 } 

\foreach \i in {-6,-5,-4,-3,-2,-1,4,5,6} {
    \draw (x)--(\i,0);
    \node[fill=white] at (\i,0) {}; 
 } 
 
\node[rectangle,draw=none] at (0,3.3) {\small $x$};
\node[rectangle,draw=none] at (0,-3.5) {\small $y$};
\node[rectangle,draw=none] at (-6,-0.3) {\small $z_{1,1}$};
\node[rectangle,draw=none] at (-5,-0.3) {\small $z_{2,1}$};
\node[rectangle,draw=none] at (-4,-0.3) {\small $z_{3,1}$};
\node[rectangle,draw=none] at (-3,-0.3) {\small $z_{1,2}$};
\node[rectangle,draw=none] at (-2,-0.3) {\small $z_{2,2}$};
\node[rectangle,draw=none] at (-1,-0.3) {\small $z_{3,2}$};
\node[rectangle,draw=none] at (4,-0.3) {\small $z_{1,k}$};
\node[rectangle,draw=none] at (5,-0.3) {\small $z_{2,k}$};
\node[rectangle,draw=none] at (6,-0.3) {\small $z_{3,k}$};
\node[rectangle,draw=none] at (1,0) {$\cdots$};
\end{tikzpicture}

\caption{Planar graph $\Gamma$, $|V(\Gamma)|=3k+2$, with   $\gapzvezdica(\Gamma)\geq k-2$.}\label{fig:planar-graph-arbitrary-large-gap*}
\end{figure}

Notice that $\{z_{1,1}, z_{3,1}, \ldots, z_{1,k}, z_{3,k}\}$ is an independent set of cardinality $2k$, so $\alpha(\Gamma) \ge 2k$ and
thus $\gapzvezdica(\Gamma)\geq 4k-(3k+2)=k-2$
by Corollary~\ref{cor:removing-k-vertices}. Using techniques obtained in Sections~\ref{sec:tau} and~\ref{sec:vtx-removal} it is possible to show that $\gapzvezdica(\Gamma)= k-2.$
\end{example}

In Sections~\ref{sec:definition-kappa}--\ref{sec:vtx-removal} we developed methods that can be applied to graphs in $\GnoFourCycle$. However, in some special cases a subgraph of a graph $G\not\in\GnoFourCycle$ that is forbidden in $\GnoFourCycle$ can be removed while maintaining control over the gap. In particular, we may be able to reduce the computation of $\gap(G)$ to a graph in $\GnoFourCycle$. This is illustrated in the following example, where pendent $4$-cycles (cycles of the form $\pathk{w}{w}{3}$) are removed using  Proposition~\ref{prop:from KBS24}(\ref{prop:2.1(2)}).

 A subgraph of $G$ of the form $\pathk{w}{w}{k-1}$, $k\geq 3$,  for some $w\in \hdv(G)$ is a pendent cycle $C_k$ at $w$.
   Let $G$ be a graph with $s$ pendent $4$-cycles $G_i$, $i=1,\ldots,s$, at a vertex $w\in V(G)$. 
   Note that each $4$-cycle $G_i$ contains two twin vertices, both in $N_G[w]$, therefore 
   by Proposition~\ref{prop:from KBS24}~(\ref{prop:2.1(2)}), we have 
   $\gap(G) = \gap(G_1) +s$, where $G_1$ is obtained by removing one of the twins in each pair.
   Note tha $G_1$ contains $s$ pendent paths $P_2$ and by Proposition~\ref{prop:from KBS24}~(\ref{prop:2.1(1)}) it follows that
    \begin{equation}
        \label{eq:pendent-cycles}
        \gap(G)= \gap(G')+s,
    \end{equation}
    where $G'$ is obtained from $G_1$ by removing all pendent paths $P_2$. If $G' \in \GnoFourCycle$, then we can compute $\gap(G)$ using the methods developed in this paper. We offer such an example below.

\begin{example} Let $\clover(f,e,o)$, $f+e+o\geq 2$, be the set of graphs isomorphic to a vertex $w$ with $f$ pendent $4$-cycles, $e$ pendent cycles with an even number of vertices not equal to $4$ and  $o$ pendent cycles with an odd number of vertices. 
Moreover, let $\cloverNoFour(e,o):=\clover(0,e,o)=\clover(f,e,o) \cap \GnoFourCycle$ denote the set of clovers without four cycles.

For every $G\in \clover(f,e,o)$ let $G'$  be a graph obtained by removing all $f$ four cycles of $G$ and therefore $\gap(G)=f+\gap(G')$ by~\eqref{eq:pendent-cycles}.

If $e+o \ge 2$ then $G'\in\cloverNoFour(e,o)$.
The multigraph $\Gamma = \chdv(G')$ contains only one vertex $v$ with $e$ loops of weight $1$ and $o$ loops of weight $0$.
If $e\ge 1$, then $\tau_1(\Gamma)$ is empty and $t_1(\Gamma) = e-1$, so 
$$\gap(G') = \gapzvezdica(\Gamma) = e-1.$$
If $e=0$, then  $\tau_1(\Gamma)$ is a vertex with a 0-loop and $t_1(\Gamma) = 0$, so  $\tau(\Gamma)$ is empty and $t(\Gamma) = 0$, thus  
$$\gap(G') = \gapzvezdica(\Gamma) = 0.$$

If $e+o = 1$ then $G'$ is a cycle and $\gap(G') = 0$. Finally, if $e+o = 0$, then $G' \cong K_1$ and $\gap(G') = 1$.

Therefore
\begin{equation}\label{eq:clover}
    \gap(G)=\begin{cases}
        f+\max\{0,e-1\};  & \text{if } e+o \ge 1, \\
        f+1;  & \text{if } e+o =0.
    \end{cases}
\end{equation}
for every $G\in \clover(f,e,o)$.  
Note that~\eqref{eq:clover} extends to the case $f+e+o=1$. In this case, $G$ is isomorphic to a cycle and so $\gap(G)$ coincides with~\eqref{eq:clover}.
\end{example}

\section*{Acknowledgments}

The authors acknowledge financial support from the ARIS (Slovenian Research and Innovation Agency, research core funding No. P1-0222 and project J1-70034).

\bibliographystyle{plain}
\bibliography{references}

\end{document}